\documentclass[12pt]{amsart}
\usepackage{amssymb}
\usepackage{amsfonts}
\usepackage{latexsym}
\usepackage{amscd}
\usepackage[mathscr]{euscript}
\usepackage{multicol}
\usepackage{tikz}
\usepackage{tkz-graph}

\usepackage[active]{srcltx}

\usepackage{enumerate}

\usepackage[utf8]{inputenc}
\usepackage[all,cmtip,2cell]{xy}

\usepackage[colorlinks  = true,
            citecolor   = blue,
            linkcolor   = blue,
            bookmarks   = true,
            linktocpage = true]{hyperref}

\usepackage{tikz}
\usetikzlibrary{matrix}

\numberwithin{equation}{section}

\theoremstyle{plain}
\newtheorem{lemma}{Lemma}[section]
\newtheorem{theorem}[lemma]{Theorem}
\newtheorem{corollary}[lemma]{Corollary}
\newtheorem{proposition}[lemma]{Proposition}
\newtheorem{definition}[lemma]{Definition}
\newtheorem*{proposition*}{Proposition}
\newtheorem*{theorem*}{Theorem}
\newtheorem*{definition*}{Definition}
\newtheorem*{claim*}{Claim}
\newtheorem{notation}[lemma]{Notation}

\newtheorem{thmx}{Theorem}

\newtheorem{remarks}[lemma]{Remarks}

\newtheorem{remark}[lemma]{Remark}

\newtheorem{point}[lemma]{}

\newcommand{\SGr}{\mathbf{SGr}}

\newcommand{\N}{{\mathbb{N}}}

\newcommand{\Z}{{\mathbb{Z}}}

\newcommand{\ol}{\overline}
\newcommand{\uloopr}[1]{\ar@'{@+{[0,0]+(-4,5)}@+{[0,0]+(0,10)}@+{[0,0] +(4,5)}}^{#1}}
\newcommand{\uloopd}[1]{\ar@'{@+{[0,0]+(5,4)}@+{[0,0]+(10,0)}@+{[0,0]+ (5,-4)}}^{#1}}
\newcommand{\dloopr}[1]{\ar@'{@+{[0,0]+(-4,-5)}@+{[0,0]+(0,-10)}@+{[0, 0]+(4,-5)}}_{#1}}
\newcommand{\dloopd}[1]{\ar@'{@+{[0,0]+(-5,4)}@+{[0,0]+(-10,0)}@+{[0,0 ]+(-5,-4)}}_{#1}}

\newcommand{\luloop}[1]{\ar@'{@+{[0,0]+(-8,2)}@+{[0,0]+(-10,10)}@+{[0, 0]+(2,2)}}^{#1}}

\newcommand{\Ifree}{I_{\mathrm{free}}}
\newcommand{\Ireg}{I_{\mathrm{reg}}}

\newcommand{\Ered}{E_{\mathrm{red}}}

\DeclareMathOperator{\rL}{L}

\renewcommand{\_}{\ \cdot\ }

\usepackage{mathtools}

\title[Realization for finitely generated refinement monoids]{\textbf{The realization problem for finitely generated refinement monoids}}
\date{}

\begin{document}

\author{Pere Ara}\author{Joan Bosa}
\address{Departament de Matem\`atiques, Universitat Aut\`onoma de Barcelona,
08193 Bellaterra (Barcelona), Spain.} \email{para@mat.uab.cat}\email{jbosa@mat.uab.cat}

\author{Enrique Pardo}
\address{Departamento de Matem\'aticas, Facultad de Ciencias, Universidad de C\'adiz,
Campus de Puerto Real, 11510 Puerto Real (C\'adiz),
Spain.}\email{enrique.pardo@uca.es}\urladdr{https://sites.google.com/a/gm.uca.es/enrique-pardo-s-home-page/}

\thanks{The three authors were partially supported by the DGI-MINECO and European Regional Development Fund, jointly, through the grant MTM2017-83487-P.
The first and second authors were partially supported by the Generalitat de Catalunya through the grant 2017-SGR-1725. The second author was partially supported by the Beatriu de Pinós postdoctoral programme of the Government of Catalonia's Secretariat for Universities and Research of the Ministry of Economy and Knowledge (BP2017-0079).
The third author was partially supported by PAI III grant FQM-298 of the Junta de Andaluc\'{\i}a.}
\date{\today}
\begin{abstract}
	We show that every finitely generated conical refinement monoid can be represented as the monoid $\mathcal V (R)$ of isomorphism classes of finitely generated projective modules over a von Neumann regular ring $R$. 
	To this end, we use the representation of these monoids provided by {\it adaptable} separated graphs. 
	Given an adaptable separated graph $(E,C)$ and a field $K$, we build a  von Neumann regular $K$-algebra $Q_K(E,C)$ and show that there is a natural isomorphism between the separated graph monoid $M(E,C)$ and the monoid $\mathcal V (Q_K(E,C))$. 
\end{abstract}

\dedicatory{Dedicated to the memory of Antonio Rosado P\'erez}

\maketitle
\tableofcontents


\section*{Introduction}

For a unital ring $R$, let $\mathcal V (R)$ denote the commutative monoid of isomorphism classes of finitely generated projective right $R$-modules with the operation given by $[A]+[B]= [A\oplus B]$.
The commutative monoid $\mathcal V (R)$ is always conical (i.e., it satisfies the axiom $x+y=0\implies x=y=0$ for $x,y\in \mathcal V (R)$),  and has an order-unit given by the class of the regular 
module $[R_R]$ in $\mathcal V (R)$.
By results of Bergman \cite[Theorems 6.2 and 6.4]{Bergman} and Bergman and Dicks \cite[page 315]{BD} every conical monoid with an order-unit can be realized in the form 
$\mathcal V (R)$ for some unital hereditary ring $R$.
The monoid $\mathcal V (R)$ can also be defined in the non-unital case (see Subsection \ref{subsec:RingsandAlgebras} below), and it has been shown by Goodearl and the first-named author \cite{AG12}
that every conical monoid is isomorphic to the monoid $\mathcal V (R)$ of a possibly non-unital hereditary ring $R$.  

The purpose of this paper is to show that every finitely generated conical refinement monoid can be realized as the monoid $\mathcal V (R)$ for some (unital) von
Neumann regular ring $R$ (see Theorem B below). This result follows immediately from our main result (Theorem A), and the representation theorem for finitely generated conical refinement monoids in terms of combinatorial data 
obtained in \cite{ABP}.     

The realization question for (von Neumann) regular rings was posed by K. R. Goodearl in \cite{directsum}. Indeed, Goodearl formulated there the following fundamental open problem:
``Which monoids arise as $\mathcal V(R) $ for regular rings $R$?". Because of the abundance of idempotents in regular rings, the knowledge of the structure of $\mathcal V (R)$ is a vital
piece of information for a regular ring $R$. For instance, $\mathcal V (R)$ contains full information on the lattice of ideals of $R$ (see \cite[Proposition 7.3]{GW}). 
For a regular ring $R$ it is well-known that the monoid $\mathcal V (R)$ satisfies the Riesz refinement axiom (see \cite[Theorem 2.8]{vnrr})
and this was the only additional property that was known at the time the above fundamental problem was formulated. An example of a conical refinement monoid of size $\aleph_2$ which cannot be realized as the 
$\mathcal V$-monoid of any regular ring was given by Wehrung in \cite{Weh}. It is still an open problem whether all the conical refinement monoids of size $\le \aleph_1$ can be realized by regular rings. 
The countable case is especially interesting since most direct sum decomposition problems involve only countably many modules. We refer the reader to \cite{Areal} for a survey on the realization problem for
regular rings.

A systematic approach to the realization problem was initiated in \cite{AB} through the consideration of {\it graph monoids}. Given a directed graph $E$ such that each vertex emits only a finite number of edges,  
the graph monoid $M(E)$ 
is the graph generated by elements $a_v$, with $v\in E^0$, subject to the relations $a_v= \sum _{e\in E^1: s(e)= v} a_{r(e)}$ for each vertex $v$ which is not a sink. 
By \cite[Proposition 4.4]{AMFP}, the graph monoid $M(E)$ is a conical refinement monoid, and it was shown in \cite{AB} that, for each fixed field $K$, there exists   
a von Neumann regular $K$-algebra $Q_K(E)$ such that $\mathcal V (Q_K(E)) \cong M(E)$. This immediately raised the question of whether all 
finitely generated conical refinement monoids can be 
represented as graph monoids. The answer to this question is negative even for antisymmetric refinement monoids, the most basic counter-example is the monoid 
$M= \langle p,a,b \mid p= p+a= p+b \rangle$ which was proved to not even be a retract of a graph monoid in \cite{APW08}. 
Another crucial step towards the  solution of the realization question, covering in particular the monoid $M$ just described, was provided by the first-named  author in \cite{Ara10}. Indeed, 
he showed that the realization problem has a positive answer for any finitely generated antisymmetric conical refinement monoid with all its prime elements free.

Although \cite{Ara10} covers a large class of examples, it became clear that a better combinatorial model was needed in order 
to understand the complexity of all finitely generated conical refinement monoids. After the work done in \cite{AP16} and \cite{AP17} (based on previous work by Pierce \cite{Pierce}, Dobbertin \cite{Dobb84} and Brookfield \cite{Brook}),
the main missing combinatorial tool was discovered in \cite{ABPS} and \cite{ABP}. The key idea is revealed through the consideration of the monoid $M$ described above, which is not a graph monoid. 
If we consider the following graph $E$: 
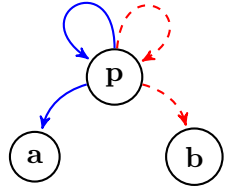
	\begin{multicols}{1}
	\hspace{6cm}\begin{tikzpicture}[->,>=stealth',shorten >=1pt,auto,node distance=1.5cm,
		thick,main node/.style={circle,draw,font=\small\bfseries}]
		
		\node (0) {};
		\node[main node] (1) [below left of=0]{p};
		\node[main node] (2) [below left of=1] {a};
		\node[main node] (3) [below right of=1] {b};
		
		\path
		(1) edge [out=90,in=145,looseness=8, color={blue} ] node {}  (1)
		edge [out=85,in=30,looseness=8, dashed, color={red}] node {}  (1)
		edge [bend right, color={blue}] node {} (2)
		edge [bend left, dashed, color={red}] node {} (3);
		\end{tikzpicture}
	\end{multicols}
\noindent along with a partition of the set of edges of $E$ into two classes, the ones with continuous lines, and the ones with dashed lines, we can localize 
the relations of the graph monoid to each set of the edge partition and obtain indeed the
two required relations $p= p+a$ and $p= p+b$. In general one defines a {\it separated graph} as a pair $(E,C)$, where $E$ is a directed graph and $C= \bigsqcup_{v\in E^0} C_v$ is a partition of $E^1$ which is finer than the partition $\{ s^{-1}(v)\mid v\in s(E^1)\}$, induced by the source map $s$.
Given a separated graph $(E,C)$ with $ |X|<\infty$ for all $X\in C$, we define the monoid $M(E,C)$ as the monoid generated by $a_v$, $v\in E^0$, with the 
relations $a_v= \sum _{x\in X} a_{r(x)}$ for all $v\in E^0$ and $X\in C_v$ (see Definition \ref{defsepgraph}).

The class of all separated graph monoids is too large for our purposes, and indeed it contains non-refinement monoids (see \cite[Section 5]{AG12}). 
In order to deal with our realization question, a special class is required, and this is precisely the
class of all {\it adaptable separated graphs} (introduced in \cite{ABP} and \cite{ABPS}), see Definition \ref{def:adaptable-sepgraphs} below for the precise definition. 

We can now state the main result of the paper:

\begin{thmx} \label{ThmA}
	Let $(E,C)$ be an adaptable separated graph and let $K$ be a field. Then there exists a von Neumann regular $K$-algebra $Q_K(E,C)$ and a natural monoid isomorphism 
	$$M(E,C) \linebreak \to \mathcal V(Q_K(E,C)).$$	
\end{thmx}

Using a result from \cite{ABP} and Theorem \ref{ThmA} we obtain:

\begin{thmx}\label{ThmB}
	Let $M$ be a finitely generated  conical refinement monoid and let $K$ be a field. Then there exists a von Neumann regular (unital) $K$-algebra $R$ such that  $M\cong \mathcal V(R)$. 
\end{thmx}

We can provide right away the proof of Theorem B (assuming Theorem \ref{ThmA} has been proved). Let $M$ be a finitely generated conical refinement monoid. By \cite[Theorem (2)]{ABP}, there exists an 
adaptable separated graph $(E,C)$ such that $M\cong M(E,C)$.
By Theorem \ref{ThmA}, we have $$\mathcal V (Q_K(E,C)) \cong M(E,C)\cong M$$ 
for the von Neumann regular $K$-algebra $Q_K(E,C)$ introduced in Section \ref{sect:algebras}.  
The algebra $Q_K(E,C)$ might be non-unital, but it can be replaced by a unital one
using a standard trick. Indeed, observe that $M$ has an order-unit, for instance the sum of all elements in a finite generating set is an order-unit for $M$. Let $e$ be a projection in $M_{\infty} (Q_K(E,C))$ 
corresponding to the order-unit 
through the isomorphism $M\cong \mathcal V (Q_K(E,C))$. Then we have that $Q: = eM_{\infty} (Q_K(E,C)) e$ is  a unital regular $K$-algebra with $\mathcal V (Q)\cong M$.    

\medskip

We now briefly discuss the realization problem for countable refinement monoids in the light of our present achievement.  
Goodearl and the first-named author made in \cite{AG15} a fundamental division in the class of all conical refinement monoids. Namely, they defined the class of {\it tame refinement monoids} 
as the class of those monoids which can be written as
a direct limit of finitely generated refinement monoids. A refinement monoid is {\it wild} if it is not tame. There are some fundamental distinctions between the classes of tame and wild refinement monoids.
All tame refinement monoids are well-behaved, in particular they are separative, unperforated and satisfy the Riesz interpolation property (see \cite[Section 3]{AG15} for details).
On the other hand, countable wild refinement monoids can fail to satisfy any of the above properties.
With respect to the realization problem, both classes seem to behave differently too. While it is conceivable --and plausible given the result in the present paper-- that any countable tame conical refinement monoid 
can be represented as the monoid $\mathcal V(R)$
of a regular $K$-algebra $R$ for an arbitrary field $K$, 
there are known examples of countable wild conical refinement monoids $M$ which are not representable by a regular $K$-algebra for any {\it uncountable} field $K$ (see \cite[Section 4]{Areal}).
Indeed a sufficient condition for this to happen is that $M$ is a conical non-cancellative refinement monoid with order-unit admitting a faithful state. 
An explicit example of such a wild refinement monoid is studied in detail in \cite{AG17} in connection with the semigroup algebra of the monogenic free inverse monoid.      
A natural next step in the realization problem is to extend the methods of the present paper to the study of the realization of homomorphisms between two finitely generated conical refinement monoids, with the 
objective of showing a realization theorem for the class of all the countable tame conical refinement monoids. 
Advances in the realization problem for wild monoids have been scattered through the literature up to this moment. Some interesting constructions in this direction are contained in \cite{AG17} and in \cite{Ruz}. 
See also \cite{BS} for realization results for semiartinian regular rings, and \cite{Vas} for realization results in the setting of graded algebras.  

\medskip

In the next subsection we briefly discuss the strategy we follow for the proof of our main result.

\subsection*{\sc Presentation of the techniques}\hfill \\

As already recalled above, our basic tool is the notion  of an adaptable separated graph (see Definition \ref{def:adaptable-sepgraphs}). 
The structure of such object $(E,C)$ is shaped by the poset $I:=E^0/{\sim}$. Indeed, let $\le$ be the pre-order relation on $E^0$ defined by $v\le w$ if there is a directed path from $w$ to $v$, and $\sim$ be the equivalence relation on $E^0$ defined by $v\sim w$ if $v\le w$ and $w\le v$. Then, $I:=E^0/{\sim}$ is a poset with respect to the partial order induced 
by the pre-order $\le$ on $E^0$.

We now give a brief sketch of the proof of Theorem \ref{ThmA}. In broad outline, the proof consists in decomposing our original adaptable separated graph $(E,C)$ into a family of non-separated graphs, 
where we can apply the results from \cite{AB},
and then reconstruct $(E,C)$, the monoid $M(E,C)$ and the $K$-algebra $Q_K(E,C)$ in terms of the ones corresponding to the above-mentioned family of non-separated graphs. This is done in such a way that we 
keep control of maintaining the desired isomorphisms between the graph monoids and the $\mathcal V$-monoids of the algebras.

Leaving apart Section \ref{sect:Preliminaries}, which establishes preliminary definitions and results, each of the points below corresponds to a section in the article.

\medskip

(1) In Section \ref{sect:algebras} we define and study the target $K$-algebras, denoted by $Q_K(E,C)$, for any adaptable separated graph $(E,C)$ and any field $K$.    

The algebra $Q_K(E,C)$ is a suitable {\it universal localization} of the algebra $\mathcal S_K(E,C)$ built in \cite{ABPS}. The latter should be understood as an analogue of the 
Leavitt path algebra $L_K(E)$ \cite{AAS}, although we warn the reader that $\mathcal S_K(E,C)$ is isomorphic to neither of the algebras $L_K(E,C)$ nor $L^{{\rm ab}}_K(E,C)$ defined in \cite{AG12} and
\cite{AE} respectively. As in \cite{AB} and \cite{Ara10}, the process of universal localization is required in order to get a von Neumann regular ring. 
We also develop in Section 2 some basic technical tools needed later. 

 \medskip

 After having established all the necessary properties of the algebras $Q_K(E,C)$, we proceed in the next sections to the proof of Theorem \ref{ThmA}. This is done by a method which is reminiscent to the 
 method employed in \cite{Ara10}, although we need to develop a new construction in the present paper. This transforms our original adaptable separated graph $(E,C)$ into a new one $(\tilde{E},\tilde{C})$ with an additional property,
described below.

\medskip

(2) In Section \ref{sect:covermap} we build, for each adaptable separated graph $(E,C)$, another adaptable separated graph $(\tilde E,\tilde C)$ satisfying a condition called condition {\rm (F)}.
This condition requires that each strongly connected component $[v] \in \tilde I=\tilde E^0/{\sim}$ receives edges from at most one strongly connected component $[w]\in \tilde I$ with $[w]\ne [v]$ and, moreover, if $|\tilde C_w|>1$, then 
only one set $X\in \tilde C_w$ emits edges 
ending at the strongly connected component $[v]$. This implies in particular that its associated poset $\tilde I$ is a forest (Lemma \ref{lem:tildeEforest}). 
Moreover, there is a cover map $\phi:(\tilde E,\tilde C)\to (E,C)$ that relates both adaptable separated graphs. Roughly speaking, to build $(\tilde E,\tilde C)$, we copy as many times as needed all the information arising from our original separated graph $(E,C)$ in order to both not loosing information and obtaining just one set of edges that leads to each vertex. The specific development of this machinery is described in Section \ref{sect:covermap}, which ends up with Theorem \ref{thm:covering-theorem}. An easy example of this first step is drawn below (different colours means different sets of edges).
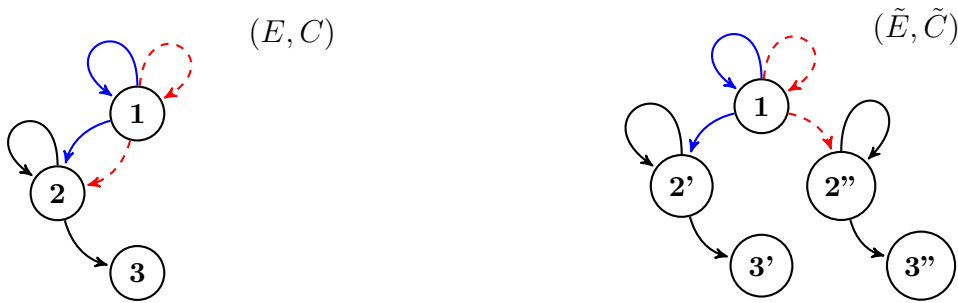
\begin{multicols}{2}
	\begin{tikzpicture}[->,>=stealth',shorten >=1pt,auto,node distance=1.5cm,
	thick,main node/.style={circle,draw,font=\small\bfseries}]
	
	\node (0) {\hspace{2cm}$(E,C)$};
	\node[main node] (1) [below left of=0]{1};
	\node[main node] (2) [below left of=1] {2};
	\node[main node] (3) [below right of=2] {3};
	
	\path
	(1) edge [out=90,in=145,looseness=8, color={blue} ] node {}  (1)
	edge [out=85,in=30,looseness=8, dashed, color={red}] node {}  (1)
	edge [bend right, color={blue}] node {} (2)
	edge [bend left, dashed, color={red}] node {} (2)
	(2) edge [out=90,in=145,looseness=8 ] node {}(2)
	edge [bend right] node {} (3);   
	
	\end{tikzpicture}
	\hspace{3cm}\begin{tikzpicture}[->,>=stealth',shorten >=1pt,auto,node distance=1.5cm,
	thick,main node/.style={circle,draw,font=\small\bfseries}]
	
	\node (0) {\hspace{2cm}$(\tilde E,\tilde C)$};
	\node[main node] (1) [below left of=0]{1};
	\node[main node] (2) [below left of=1] {2'};
	\node[main node] (3) [below right of=1] {2''};
	\node[main node] (4) [below right of=2] {3'};
	\node[main node] (5) [below right of=3] {3''};
	
	\path
	(1) edge [out=90,in=145,looseness=8, color={blue}] node {}  (1)
	edge [out=85,in=30,looseness=8, dashed, color={red} ] node {}  (1)
	edge [bend right, color={blue}] node {} (2)
	edge [bend left, dashed, color={red}] node {} (3)
	(2) edge [out=90,in=145,looseness=8 ] node {}(2)
	edge [bend right] node {} (4)
	(3) edge [out=90,in=45,looseness=8 ] node {}(3)
	edge [bend right] node {} (5);      
	\end{tikzpicture}
\end{multicols}

(3)  Let us now consider  an adaptable separated graph $(\tilde E,\tilde C)$ satisfying condition {\rm (F)}. In this third step, corresponding to Section \ref{sec:BuildingBlocks}, we reconstruct $(\tilde E,\tilde C)$ via successive pullbacks of what we have called building blocks. 
In particular, these building blocks are the connected components of the non-separated graphs obtained by choosing a single set $X\in \tilde C_v$ at each of the vertices $v$ of $\tilde E$
(see Definition \ref{def:buildingblocks-graphs}). 
Notice that the building blocks $E_i$ are non-separated directed graphs; therefore, they satisfy $M(E_i)\cong \mathcal V(Q_K(E_i))$ (\cite{AB}). In our easy separated graph displayed before, the associated 
building blocks are:

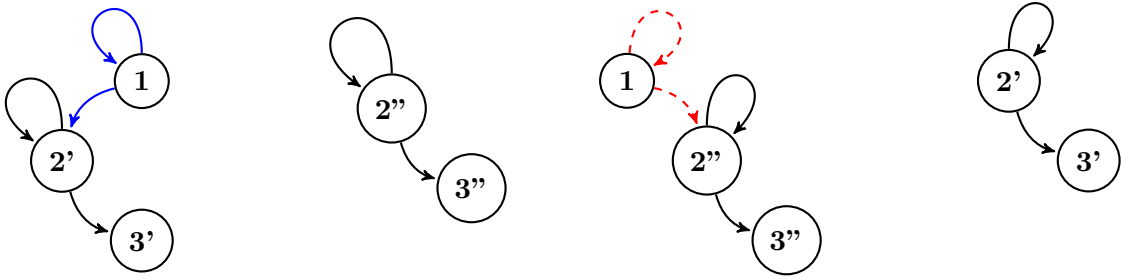
\begin{multicols}{4}
	
	\begin{tikzpicture}[->,>=stealth',shorten >=1pt,auto,node distance=1.5cm,
	thick,main node/.style={circle,draw,font=\small\bfseries}]
	
	\node (0) {};
	\node[main node] (1) [below left of=0]{1};
	\node[main node] (2) [below left of=1] {2'};
	\node[main node] (3) [below right of=2] {3'};
	
	\path
	(1) edge [out=90,in=145,looseness=8, color={blue} ] node {}  (1)
	edge [bend right, color={blue}] node {} (2)
	
	(2) edge [out=90,in=145,looseness=8 ] node {}(2)
	edge [bend right] node {} (3);   
	
	\end{tikzpicture}
	
	\begin{tikzpicture}[->,>=stealth',shorten >=1pt,auto,node distance=1.5cm,
	thick,main node/.style={circle,draw,font=\small\bfseries}]
	
	\node (0) {};
	\node[main node] (1) [below left of=0]{2''};
	\node[main node] (2) [below right of=1] {3''};

	\path
	
	(1) edge [out=90,in=145,looseness=8 ] node {}(1)
	edge [bend right] node {} (2);   
	
	\end{tikzpicture}
	
	\begin{tikzpicture}[->,>=stealth',shorten >=1pt,auto,node distance=1.5cm,
	thick,main node/.style={circle,draw,font=\small\bfseries}]
	
	\node (0) {};
	\node[main node] (1) [below left of=0]{1};
	\node[main node] (2) [below right of=1] {2''};
	\node[main node] (3) [below right of=2] {3''};
	
	\path
	(1)	edge [out=85,in=30,looseness=8, dashed, color={red}] node {}  (1)
	edge [bend left, dashed, color={red}] node {} (2)
	(2) edge [out=90,in=45,looseness=8 ] node {}(2)
	edge [bend right] node {} (3);   
	
	\end{tikzpicture}
	
	\begin{tikzpicture}[->,>=stealth',shorten >=1pt,auto,node distance=1.5cm,
	thick,main node/.style={circle,draw,font=\small\bfseries}]
	
	\node (0) {};
	\node[main node] (1) [below right of=0] {2'};
	\node[main node] (2) [below right of=1] {3'};
	
	\path
	(1) edge [out=90,in=45,looseness=8 ] node {}(1)
	edge [bend right] node {} (2);   
	
	\end{tikzpicture}
\end{multicols}

The behaviour of the above-mentioned  pullbacks is analyzed at the different frameworks: monoids, $K$-algebras and $\mathcal V$-functor. 
We finish this section showing in Theorem \ref{thm:main-for-F} our main result 
for the class of adaptable separated graphs satisfying condition {\rm (F)}, i.e. 
$$M(\tilde E,\tilde C)\cong \mathcal V(Q_K(\tilde E,\tilde C)).$$ 
It is worth to mention here that there are two technical difficulties we need to overcome in this step. First, at the level of 
the $K$-algebra for the building blocks, a slight variation of the usual Leavitt path algebra of a directed graph is needed. 
This has been worked out in \cite{A18}, so we only need to refer the results in that paper. Second, in the transition from the algebra setting to the monoid setting, 
we encounter the difficulty that the $\mathcal V$-monoid of a pullback of rings is {\it not} in general the pullback of the corresponding $\mathcal V$-monoids. A necessary and sufficient 
condition for this to hold, involving the $K_1$-groups of algebraic $K$-theory, was established (for a large class of rings) in \cite{Ara10}. We are able to verify this $K_1$-condition in our situation
(see Proposition \ref{prop:conditionforK1}).  

(4) In this final step, we return to the cover map $\phi:(\tilde E,\tilde C)\to (E,C)$ described in (2) in order to move back from the 
auxiliary separated graph $(\tilde E,\tilde C)$ to our original separated graph $(E,C)$. To this end, we use the {\em crowned push-out} construction. We consider diagrams of the form
\[
\begin{tikzpicture}
\matrix (m) [matrix of math nodes,row sep=3em,column sep=3em,minimum width=1em]
{
	I   &I\\
	I' &P\\
};
\path[-stealth]
(m-1-1) edge [->] node  [above] {=} (m-1-2)
edge [->] node [left] {$\varphi$} (m-2-1)

(m-1-2) edge [->] node [right] {$\iota_1$} (m-2-2)

(m-2-1)       edge [->] node  [below] {$\iota_{2}$} (m-2-2)
;
\end{tikzpicture} 
\]
where $I$ and $I'$  are order-ideals in $P$ that are isomorphic via $\varphi$ and satisfy $I\cap I'=\{0\}$. Then, we define the crowned pushout of $(P,I,I',\varphi)$ as the coequalizer of the maps $\iota_1$ and 
$\iota_2\circ\varphi$ (Definition \ref{def:crownedPush}). In Section \ref{sec:Pushouts} we show that this construction is well-behaved at all our settings: monoids, $K$-algebras and $\mathcal V$-functor, and agrees 
with what we expect at the level of adaptable separated graphs. In particular, we build a finite chain of adaptable separated graphs and cover maps 
$$(\tilde E,\tilde C)=(\tilde E_n, \tilde C_n)\overset{\phi_n}{\longrightarrow} (\tilde E_{n-1}, \tilde C_{n-1})\overset{\phi_{n-1}}{\longrightarrow} \ldots \overset{\phi_1}{\longrightarrow} (\tilde E_0, \tilde C_0)=(E,C),$$ 
satisfying that each $M(\tilde E_{k-1},\tilde C_{k-1})$ is the crowned push-out of a quadruple determined by $(\tilde E_{k},\tilde C_k)$ and $\phi_k$, for all $k\in \{1,\ldots, n\}$.
At the algebra level, we use the results in \cite{Ara10} to show in Theorem \ref{thm:main-pushoutthem} that if $M(\tilde E_k,\tilde C_k)\cong\mathcal V(Q_K(\tilde E_k,\tilde C_k))$ for some $k\in \{1,\ldots, n\}$, 
then $M(\tilde E_{k-1},\tilde C_{k-1})\cong\mathcal V(Q_K(\tilde E_{k-1},\tilde C_{k-1}))$, i.e.,  the realization theorem holds inductively along the displayed chain.

Combining Theorem \ref{thm:main-pushoutthem} with step (2) (Theorem \ref{thm:main-for-F}), where it is shown that $M(\tilde E,\tilde C)\cong\mathcal V(Q_K(\tilde E,\tilde C))$, one obtains the 
desired proof of Theorem \ref{ThmA} by induction.

\section{Preliminaries}\label{sect:Preliminaries}

In this section we collect some basic definitions and facts needed to follow the paper.

\subsection{Posets} A \emph{pre-ordered set} is a set $J$ endowed with a reflexive and transitive relation $\le$. If $\le $ is in addition antisymmetric, we say that $(J,\le )$ is a {\it poset}
(partially ordered set). We refer the reader to \cite{FF} for a recent interesting paper on the structure of pre-ordered sets.

Let $(I\le )$ be a poset. A subset $J$ of $I$ is a {\it lower subset} if $x\le y$ and $y\in J$ imply $x\in J$. We denote by $\mathcal L (I)$ the set of all the lower subsets of $I$. 
Note that $\mathcal L(I)$ is a complete distributive lattice, with $\bigwedge$ and $\bigvee$ given by intersection and union respectively.  
For $p\in I$, $ I\downarrow p := \{ x\in I : x\le p \}$ is the lower subset of $I$ generated by $p$.

For an element $p$ of a poset $I$, write
$$\rL(p)=\rL(I , p)=\{q\in I : q<p \text{ and } [q,p]=\{q,p\}\},$$
where $[q,p]= \{x\in I : q\le x\le p \}$ is the interval determined by $q$ and $p$. 
The set $\rL (p)$ is called the {\it lower cover} of $p$. 

We will also need the concepts of tree and forest for a poset, as follows: 

\begin{definition}
 \label{poset-be-a-tree}
{\rm  Let $(I,\le )$ be a poset. We say that $I$ is a {\it tree} in case there is a greatest element $i_0\in I$ and for every
 $i\in I$ the interval $[i,i_0] : = \{ j\in I \mid i\le j\le i_0 \}$ is a chain. The element $i_0$ will be called the {\it root} of the tree $I$.
 A {\it forest} is a disjoint union of trees, that is $I=\bigcup_{\alpha \in \Lambda} I_{\alpha}$ such that each $I_{\alpha}$ is a tree with the induced order, and for each $\alpha\ne \beta$ the elements
 of $I_{\alpha}$ and  $I_{\beta}$ are pairwise incomparable.} \qed
 \end{definition}

\subsection{Commutative monoids}
We will denote by $\N$ the semigroup of positive integers, and by $\Z^+$ the monoid of non-negative integers.
All the monoids appearing in this paper will be commutative and additive. 

A monoid $M$ is {\it conical} if $x+y=0$ implies $x=y=0$ for $x,y\in M$, and $M$ is said to be 
a {\it refinement monoid} if, for all $a, b, c ,d\in M$ such that $a+b=c+d$, there exist $x$, $y$, $z$, $t$ in $M$ such
that $a=x+y$, $b=z+t$,
$c=x+z$ and $d=y+t$. We can represent this situation in the form of a square:

$$\mbox{\begin{tabular}{|l|l|l|}
	\cline{2-3}
	\multicolumn{1}{l|}{} & ${c}$ & ${d}$ \\ \hline
	${a}$ & ${x}$ & ${y}$ \\ \hline
	${b}$ & ${z}$ & ${t}$ \\ \hline
	\end{tabular}}.$$

If $x, y\in M$, we write $x\leq y$
if there exists $z\in M$  such that $x+z = y$.
Note that $\le$ is a translation-invariant pre-order on $M$, called the {\it algebraic pre-order} of $M$. All inequalities in commutative monoids will be with respect to this pre-order.
An element $p$ in a monoid $M$ is a {\it prime element} if $p$ is not invertible in $M$, and, whenever
$p\leq x+y$ for $x,y\in M$, then either $p\leq x$ or $p\leq y$. A monoid $M$ is said to be {\it primely generated} if every non-invertible element of $M$
can be written as a sum of prime elements. By \cite[Theorem 6.8]{Brook}, every finitely generated refinement monoid is primely generated.

A monoid $M$ is said to be {\it separative} in case, whenever
$a,b\in M$ and $a+a=a+b=b+b$, then we have $a=b$. The reader is referred to \cite{AGOP} for 
information on the class of separative monoids and its connections with the non-stable $K$-theory of rings. 
We just remind the following useful characterization of separativity (see \cite[Lemma 2.1]{AGOP}). A monoid $M$ is separative if and only if the following {\it cancellation of small elements} holds:
$$(a+c=b+c \,\, \text{ and }\, \,  c\le na, \,  c\le mb \, \, \text{ for some } \, \, n,m\in \N ) \, \implies \,  a=b \, .$$
 By \cite[Theorem 4.5]{Brook}, every primely generated refinement monoid is separative. In particular every finitely generated refinement monoid is separative.

An element $x\in M$ is {\it regular} if $2x\leq x$. An element
$x\in M$ is {\it free} if $nx\leq mx$ implies $n\leq m$, for $n,m\in \N$. It is straightforward to show, using the above-mentioned characterization of separativity,
that any element 
of a separative monoid is either free or regular. In particular, this holds for every primely generated refinement monoid.

Let $M$ be a monoid. An {\it order-ideal} of~$M$ is a submonoid $I$ of~$M$ satisfying that 
$$\text{ if }x+y\in I,\text{ then } x\in I \text{ and } y\in I \hspace{0.3cm}\forall x,y\in M.$$
If $I$ is an order-ideal of $M$, the equivalence relation
$\equiv_I$ defined on~$M$ by the rule
\[
 x\equiv_Iy\ \Longleftrightarrow\ \exists u,v\in I\text{ such that }x+u=y+v,\quad
 \text{for all }x,y\in M
 \]
is a monoid congruence of~$M$. We put $M/I=M/{\equiv_I}$ and we
 shall say that $M/I$ is an
\emph{ideal quotient} of~$M$. 

It is worth to mention that order-ideals are called {\it divisor-closed submonoids} in some references, see for example \cite[Chapter 1]{FacBook} and \cite{Geroldinger}.

When $M$ is a conical refinement monoid, the set $\mathcal L(M)$
of order-ideals of $M$ forms a complete distributive lattice, with
suprema and infima given by the sum and the intersection of
order-ideals respectively.

Along the sequel we will denote the Grothendieck (or enveloping) group of a commutative semigroup $S$ by $G(S)$. Recall that there is a canonical
semigroup homomorphism $\iota \colon S \to G(S)$, which is injective if and only if $S$ is cancellative.

\subsection{Adaptable Separated Graphs}\label{Section1}
Using the theory of $I$-systems, as developed in \cite{AP16}, the authors have developed in \cite{ABP} a combinatorial model 
for all finitely generated conical refinement monoids. This combinatorial model encompasses indeed a larger class of refinement monoids, which can be studied using the same methods. 
The basic ingredient in this combinatorial description is the theory of separated graphs \cite{AG12}.
(Note that ordinary graphs are not sufficient to describe {\it all} finitely generated refinement monoids, see \cite{AP17, APW08}.)

We will use the notation and conventions from \cite{AAS} and \cite{AG12} concerning graphs and separated graphs respectively. In particular, for a directed graph $E=(E^0,E^1,s,r)$, 
we denote by $E^0$ the set of vertices, by $E^1$ the set of edges, and we use $s(e)$ and $r(e)$ to denote the source and the range of an edge 
$e\in E^1$. Throughout, we will use the symbol $\bigsqcup $ to denote the union of pairwise disjoint subsets of a given set.  

Let us now recall the definition of separated graphs.

\begin{definition}[{\cite[Definitions 2.1 and 4.1]{AG12}}]\label{defsepgraph}
    {\rm A \emph{separated graph} is a pair $(E,C)$ where $E$ is a directed graph and  $C=\bigsqcup
    _{v\in E^ 0} C_v$ is a partition of $E^1$ 
    such that $C_v$ is a partition of $s^{-1}(v)$ (into pairwise disjoint non-empty subsets) for every vertex $v\in E^0$. 
    (If $v$ is a sink, we take $C_v$ to be the empty family of subsets of $s^{-1}(v)$.)

    If all the sets in $C$ are finite, we shall say that $(E,C)$ is a \emph{finitely separated} graph.

    Given a finitely separated graph $(E,C)$, we define the monoid of the separated graph $(E,C)$ to
    be the commutative monoid given by generators and relations as
    follows:
      $$M(E,C)=\Big{\langle} a_v  \, : a_v=\sum _{ e\in X }a_{r(e)} \text{ for every } X\in C_v, v\in E^0\Big{\rangle} .$$ }
\end{definition}
We will make extensive use of the following basic concepts:
\begin{definition}
\label{def:Idefined}
    {\rm Given a directed graph $E=(E^0,E^1, s,r)$:
        \begin{enumerate}
        \item We define a pre-order on $E^0$ (the path-way pre-order) by $v\le w$ if and only if there is a directed path $\gamma $ in $E$ with $s(\gamma ) = w$ and $r(\gamma ) =v$.
                       \item Let $\sim$ be the equivalence relation on the set $E^0$ defined, for every
                       $v,w\in E^0$, by $v\sim w$ if $v\le w$ and $w\le v$. Set $I=E^0/{\sim}$, so that the
                       preorder $\le$ on $E^0$ induces a partial order on $I$. We will also denote by
                       $\le$ this partial order on $I$. Thus, denoting by $[v]$ the class of $v \in E^0$ in $I$, we have $[v] \le [w]$ if and only if $v\le w$. We will often 
                       refer to $[v]$ as the {\em strongly connected component} of $v$.
                       
            \item We say that $E$ is {\em strongly connected} if every two vertices of $E^0$ are connected through a directed path, i.e., if $I$ is a singleton.
            \end{enumerate}
    }
\end{definition}

We now define the main notion used throughout the paper, which was introduced in \cite{ABP,ABPS}. This is the class of adaptable separated graphs.

\begin{definition}
    \label{def:adaptable-sepgraphs}
    {\rm Let $(E,C)$ be a finitely separated graph and let $(I,\le )$ be the partially ordered set associated to the pre-ordered set $(E^0,\leq)$. We say that $(E,C)$ is {\it adaptable} if $I$ is finite,
    and there exist a partition $I=\Ifree \sqcup \Ireg $, and a family of subgraphs $\{ E_p \}_{p\in I}$ of $E$ such that the following conditions are satisfied:
        \begin{enumerate}
            \item $E^0=\bigsqcup_{p\in I} E_p^0$, where $E_p$ is a strongly connected row-finite graph if $p\in \Ireg$ and $E_p^0= \{ v^p \}$ is a single vertex if $p\in \Ifree $.
            \item For $p\in \Ireg$ and $w\in E_p^0$, we have that $|C_w|= 1$ and $|s_{E_p}^{-1} (w)|\ge 2$. Moreover, all edges departing from $w$ either belong to the graph $E_p$ or connect $w$ to a vertex $u\in E_q^0$,
            with $q<p$ in $I$.
            \item For $p\in \Ifree$, we have that $s^{-1}(v^p) = \emptyset $ if and only if $p$ is minimal in $I$. If $p$ is not minimal, then there is a positive integer $k(p)$ such that
            $C_{v^p}=\{ X^{(p)}_1,\dots ,X^{(p)}_{k(p)} \}$. Moreover, each $X^{(p)}_i$ is of the form
            $$X^{(p)}_i = \{ \alpha (p,i) ,\beta (p,i,1),\beta (p,i,2),\dots , \beta(p, i, g(p,i)) \} ,$$
            for some $g(p,i) \ge 1$, where $\alpha (p,i)$ is a loop, i.e., $s(\alpha(p,i))= r(\alpha (p,i)) = v^p$, and $r(\beta (p,i,t))\in E^0_q$ for $q<p$ in $I$.
            Finally, we have $E_p^1= \{ \alpha(p,1),\dots ,\alpha (p,k(p)) \}$.
        \end{enumerate}

        The edges connecting a vertex $v\in E_p^0$ to a vertex $w\in E_q^0$ with $q<p$ in $I$ will be called {\it connectors}.} \qed
\end{definition}

Following the work in \cite{AP16,AP17}, we have established in \cite{ABP} the following fundamental result, which links adaptable separated graphs and refinement monoids.  

\begin{theorem}\cite{ABP}
\label{thm:maingraphs-monoids}
    The following two statements hold:
    \begin{enumerate}
        \item If $(E,C)$  is an adaptable separated graph, then $M(E,C)$ is a primely generated conical refinement monoid.
        \item For any finitely generated conical refinement monoid $M$, there exists an adaptable separated graph $(E,C)$ such that $M\cong M(E,C)$. 
    \end{enumerate}
\end{theorem}

In particular, it is shown in \cite{ABP} that, for an adaptable separated graph $(E,C)$, all the
elements $a_v$, for $v\in E^0$, are prime elements of the monoid $M(E,C)$, and that $a_v$ is free
(respectively, regular) in $M(E,C)$ if and only if $[v]\in \Ifree$ (respectively, $[v]\in \Ireg$).
We often refer to the elements of $\Ifree$ as {\it free primes} and to the elements of $\Ireg$ as
{\it regular primes}.

Recall that a subset $H$ of vertices of a directed graph $E$ is said to be {\it hereditary} if $v\le w$ and $w\in H$ imply $v\in H$. Note that hereditary subsets of $E^0$ correspond to lower subsets of $I=E^0/{\sim}$.
If $(E,C)$ is a separated graph and $H$ is a hereditary subset of $E^0$, we denote by $(E_H,C^H)$ the restriction of $(E,C)$ to $H$. We thus have $(E_H)^0=H$ and $C^H_v= C_v$ for $v\in H$.    
The following lemma will be used through the article without an explicit mention.

\begin{lemma}
 \label{lem:restriction-monoids}
Let $(E,C)$ be an adaptable separated graph and let $H$ be a hereditary subset of $E^0$. Then the order-ideal $M(H)$ of $M(E,C)$ generated by
$H$ is isomorphic to the monoid $M(E_H, C^H)$ of the separated graph $(E_H,C^H)$.
 \end{lemma}

 \begin{proof}
 This follows exactly as in \cite[Lemma 2.18]{AP17}, due to the validity of the confluence property for the monoids of adaptable 
 separated graphs (\cite[Lemma 2.4]{ABP}).   
 \end{proof}

\subsection{Rings and algebras}\label{subsec:RingsandAlgebras}
A ring $R$ is called {\it von Neumann regular} if for every
$x\in R$ there is $y\in R$ such that $x=xyx$. We refer the reader to \cite{vnrr} for the general theory of von Neumann regular rings.
The rings appearing in this paper will not be unital in general, but they have {\it local units}, that is, there is a set of idempotents $\mathcal E$ in 
$R$, which is directed with respect to the order $e\le f \iff  e=ef=fe$, such that $R=\bigcup_{e\in \mathcal E} eRe$. By \cite[Example 1]{AM87}, any von Neumann regular ring 
is a ring with local units. 

For a ring $R$, let $M_{\infty}(R)$ be the directed union of
$M_n(R)$ ($n\in\mathbb N$), where the transition maps $M_n(R)\to
M_{n+1}(R)$ are given by $x\mapsto \left( \smallmatrix x&0\\
0&0\endsmallmatrix \right)$. Two idempotents $e,f\in M_{\infty}(R)$
are {\it equivalent} in case there are $x\in eM_{\infty}(R)f$ and
$y\in fM_{\infty}(R)e$ such that $xy=e$ and $yx=f$. We define
$\mathcal V (R)$ to be the monoid of equivalence classes $[e]$ of
idempotents $e$ in $M_\infty(R)$ with the operation
$$[e]+ [f] := [\bigl(   \smallmatrix e&0\\ 0&f
\endsmallmatrix  \bigr)]$$
for idempotents $e,f\in M_\infty(R)$. 
For unital $R$, the monoid $\mathcal V (R) $ is
the monoid of isomorphism classes of finitely generated projective
right $R$-modules, where the operation is induced by direct sum. 
We will occasionally use the expression ``$\mathcal V $-monoid of a ring'' to refer to the above construction.
It is straightforward to extend the above definition to a functor $\mathcal V$ from the category of rings to the category of commutative monoids. 

If $I$ is an ideal of a unital ring $R$, then $\mathcal V (I)$ can be
identified with the monoid of isomorphism classes of finitely
generated projective right $R$-modules $P$ such that $P=PI$ (see \cite[page 296]{MM}).

If $R$ is a ring with local units, then the $K$-theory group $K_0(R)$ can be computed as the Grothendieck group of the monoid $\mathcal V(R)$, that is, 
$K_0(R)= G(\mathcal V (R))$ (see \cite[Proposition 0.1]{MM}). Recall that 
a ring $R$ is said to be {\it separative} if its monoid $\mathcal V (R)$ is a separative monoid.

When $R$ is a von Neumann regular ring, the monoid $\mathcal V (R)$ contains a lot of information about the structure of $R$. 
In the next proposition, we collect various results on this connection, that we will need later.

\begin{proposition}
 \label{prop:basicpropsof-VR}
 Let $R$ be a von Neumann regular ring. Then the following hold:
 \begin{enumerate}
  \item $\mathcal V (R)$ is a refinement monoid.
  \item The lattice $\mathcal L (R)$ of (two-sided) ideals of $R$ is a complete distributive lattice and there is a lattice isomorphism
  $\mathcal L (R) \cong \mathcal L (\mathcal V (R))$
  sending $I\in \mathcal L (R)$ to $\mathcal V (I)\in \mathcal L (\mathcal V (R))$.
  \item If $I\in \mathcal L(R)$ then there is a natural monoid isomorphism $\mathcal V (R)/\mathcal V (I) \cong \mathcal V (R/I)$.
  \end{enumerate}
\end{proposition}

\begin{proof}
 (1) and (3) were proved in \cite[Corollary 1.3, Proposition 1.4]{AGOP} for the larger class of exchange rings. The proof of (2) is contained in \cite[Proposition 7.3]{GW}. 
\end{proof}


\section{The algebras}
\label{sect:algebras}

\subsection{Preliminaries on universal localization and rational series}
\label{subsect:prelms-on-univ-local}

In this subsection we introduce the tools from the theory of universal localization and rational series that we need for our main construction. 
We refer the reader to \cite{free} and  \cite{scho} for the general theory of universal localization.

We will need the following particular instance of universal localization.  Given a family of idempotents  $\{ e_i \}_{i\in I}$ of a possibly non-unital $K$-algebra $R$ and sets of square matrices
$\Sigma_i$ over $e_iRe_i$, for $i\in I$, we consider the $K$-algebra unitization $R^1$ of $R$, and 
the algebra $R^1(\bigcup _{i\in I} \Upsilon_i)^{-1}$, where $\Upsilon_i$ is the set of all  
the matrices $(1-e_i)I_n +A\in M_n(R^1)$, where $A$ in an 
$n\times n$ matrix in $\Sigma_i$ and $I_n$ is the $n\times n$ identity matrix. Then we define $R(\bigcup_{i\in I}\Sigma_i)^{-1}$ as the ideal of $R^1(\bigcup _{i\in I} \Upsilon_i)^{-1}$ 
generated by $R$.  
Note that this corresponds to universally inverting each $A\in M_n(e_iRe_i)$ over $e_iRe_i$, that is, the canonical map 
$\iota  \colon R \to R(\bigcup_{i\in I}\Sigma_i)^{-1}$ satisfies the following universal property:

\medskip

{\it Given any $K$-algebra $T$ and any $K$-algebra homomorphism $\varphi \colon R\to T$ such that $\varphi (A)$ is invertible over $M_n(\varphi (e_i) T\varphi (e_i))$ for any $A\in \Sigma_i \cap M_n(R)$, 
there exists a unique $K$-algebra homomorphism $\widetilde{\varphi}\colon R(\bigcup_{i\in I}\Sigma_i)^{-1}\to T$ such that $\widetilde{\varphi}\circ \iota = \varphi$. }

\medskip

We now recall and extend some constructions from \cite{AB}. The notation we follow here is slightly  different from the one in \cite{AB}, but it agrees with the notation followed 
in \cite{A18}. 

Let $K$ be a field an let $E$ be a row-finite graph. We denote by $P_K(E)$ the usual path algebra over $E$ with coefficients in $K$.
Denote by $\mathcal F$ the directed family of all the finite complete subgraphs of $E$ (see \cite[Definition 1.6.7]{AAS}). 
Note that $P_K(E) = \varinjlim_{F\in \mathcal F} P_K(F)$. 
We define, for a finite graph $F$, $P_K((F))$ as the $K$-algebra of power series on $F$ (see \cite{AB}) and then we define  
$$P_K((E))= \varinjlim_{F\in \mathcal F} P_K((F)).$$
If $F$ is a finite graph, the algebra of rational series $P_K^{{\rm rat}}(F)$ is the division closure of $P_K(F)$ in $P_K((F))$, see \cite{AB}.
We define $P_K^{{\rm rat}}(E):= \varinjlim _{F\in \mathcal F} P_K^{{\rm rat}}(F)$. Using \cite[Theorem 1.20]{AB}, it is easy to see that $P_K^{{\rm rat}}(E)$ is the universal localization
$P_K(E) \Sigma^{-1}$, where $\Sigma =\bigcup_{F\in \mathcal F} \Sigma _F$, and $\Sigma _F$ is the set of square matrices $A$ over $P_K(F)$ such that $\epsilon_F(A)$ is invertible, where 
$\epsilon_F \colon P_K(F) \to \bigoplus_{v\in F^0} vK$ is the augmentation map. Note that $\Sigma_F$ is a set of square matrices over the corner algebra $v_FP_K(E)v_F$, where $v_F:= \sum _{v\in F^0} v$,  
of the possibly non-unital algebra $P_K(E)$.

Following \cite{AB}, we define, for $e\in E^1$, the right transduction
$\tilde{\delta}_e\colon P_K((E))\to P_K((E))$ corresponding to $e\in E^1$ by $$\tilde{\delta_e} \Big( \sum_{\alpha\in
\mathrm{Path}(E)}\lambda_\alpha \alpha \Big) =\sum_{\substack{\alpha\in
\mathrm{Path}(E)\\s(\alpha)=r(e)}} \lambda_{e\alpha } \alpha.$$

Note that, by the argument given after the proof of Proposition 2.7 in \cite{AB}, the algebras
$P^{\rm rat}_K(E)$ are stable under all the right transductions $\tilde{\delta}_e$. 

It will be convenient for our purposes to slightly modify the definition of the maps $\tau _e$ given in \cite[page 220]{AB}. (Again our definition here is the same as the one presented in \cite{A18}.) 

For each $e\in E^1$, define the map $\tau _e$ as the endomorphism of $P_{K}((E))$
given by the composition
$$P_{K}((E))\to \bigoplus_{v\in E^0} Kv\to \bigoplus_{v\in E^0} Kv\to P_{K}((E)),$$
where the first map is the augmentation homomorphism, the third map is the canonical inclusion, and the middle map is the $K$-lineal map
given by sending $s(e)$ to $r(e)$, and any other idempotent
$v$ with $v\ne s(e)$ to $0$.
The reader can check that the 
right transduction $\tilde{\delta}_e\colon
P_{K}((E))\to P_{K}((E))$ corresponding to $e\in E^1$ is a right
$\tau_e$-derivation, that is,
\begin{equation}
\label{right-der} \tilde{\delta}_e
(rs)=\tilde{\delta}_e(r)s+\tau_e(r)\tilde{\delta}_e(s)
\end{equation}
for all $r,s\in P_{K}((E))$ (see \cite[Lemma 2.4]{AB}).

We now review the main construction in 
\cite[Section 2]{AB}.  

\begin{proposition} \cite[Proposition 2.5]{AB}
 \label{prop:algebraSinAB}
  Let $E$ be a row-finite graph, let $E^*$ be the opposite graph,  and let 
$R$ be a subalgebra of $P_K((E))$ closed under all the right transductions $\tilde{\delta}_e$, $e\in E^1$. Then there exists a ring $S$ such that:
\begin{enumerate}[(i)]
 \item There are embeddings 
 $$L\colon R\to S,\,\, r\mapsto  L_r,\qquad  z\colon P(E^*)\to S, \,\, w^*\mapsto z_{w^*},$$
 such that $z_v=L_v$ for all $v\in E^0$, and
 $$z_{e^*}L_r = L_{\tau _e(r)}z_{e^*} + L_{\tilde{\delta}_e (r)}$$
 for all $e\in E^1$ and all $r\in R$. 
 \item $S$ is projective as a left $R$-module. Indeed, $S= \bigoplus_{\gamma \in \mathrm{Path} (E)} S_{\gamma}$ with $S_{\gamma} \cong Rr(\gamma)$ as $R$-modules.
 Moreover, every element of $S$ can be uniquely written as a finite sum $\sum_{\gamma\in \mathrm{Path}(E)} L_{a_{\gamma}}z_{\gamma^*}$, where $a_{\gamma} \in Rr(\gamma)$ for all
 $\gamma \in \mathrm{Path}(E)$. 
 \end{enumerate}
\end{proposition}

\begin{proof} Set $T=\mathrm{End}_K(R)$. For $r\in R$ let us denote by $L_r$ the element of $T$ given by left multiplication by $r$. The map $L\colon R\to T$ is clearly an injective algebra homomorphism.
For each $e\in E^1$ consider the elements $z_{e^*}$ of $T$ defined by
$$z_{e^*} (r) =\tilde{\delta}_e(r).$$
Let $S$ be the subalgebra of $T$ generated by $L(R)$ and all the elements $z_{e^*}$ defined above. It is proved as in \cite[Proposition 2.5]{AB} that
$S$ satisfies the stated properties. 
\end{proof}

The algebra $S$ defined above will be denoted by $R\langle E^*; \tau, \tilde{\delta} \rangle $. Note that both $R$ and $P_K(E^*)$ embed into $S$. Identifying the elements
of $R$ with their images in $R\langle E^*; \tau, \tilde{\delta} \rangle $, we see from Proposition \ref{prop:algebraSinAB} that every element in $R\langle E^*; \tau, \tilde{\delta} \rangle $
can be uniquely written as a finite sum $\sum_{\gamma \in \mathrm{Path}(E)} a_{\gamma}\gamma^*$, where $a_{\gamma}\in Rr(\gamma)$.  

We can now introduce the following definition, which generalizes the definition of $Q(E)$ given in \cite[page 234]{AB}. 

\begin{definition}
 \label{def:relative-regular-algebra}
{\rm  Let $E$ be a row-finite graph, let $K$ be a field, and let $X$ be a subset of regular vertices of $E^0$. Set $R:=P_K^{\rm rat} (E)$ be the algebra of rational series over $E$ and 
$S := R\langle E^*;\tau,\tilde{\delta} \rangle $.
For each $v\in X$ let $q_v= v-\sum_{e\in s^{-1}(v)} ee^*\in R$ and let $I^X$ be the ideal of $S$ generated by all the idempotents $q_v$ with $v\in X$.  
Then the {\it regular algebra of $E$ relative to $X$}, denoted by $Q^X_K(E)$, is the algebra
$$Q^X_K(E):= S/I^X.$$
The {\it regular algebra} of $E$ is the algebra $Q_K(E)= Q_K^{\text{Reg}(E)}(E)$, where $\text{Reg}(E)$ is the set of all the regular vertices of $E$.} \qed
\end{definition}

We summarize below some of the main properties of the algebras $Q_K^X(E)$. The proof of such properties is basically the same as in the non-relative case, see \cite{AB}. 

\begin{theorem}
 \label{thm:properties-relativeQ}
 Let $E$ be a row-finite graph, let $K$ be a field, and let $X$ be a subset of regular vertices of $E^0$. Then the algebra $Q_K^X(E)$ satisfies the following properties:
 \begin{enumerate}[(a)]
  \item $Q_K^X(E)$ is a von Neumann regular ring. 
  \item The subalgebra of $Q_K^X(E)$ generated by $P_K(E)$ and $P_K(E^*)$ is isomorphic to the relative Cohn path algebra $C_K^X(E)$ defined in \cite[Definition 1.5.9]{AAS}.
  \item We have a natural isomorphism $C_K^X(E)\Sigma ^{-1}\cong Q_K^X(E)$, where $\Sigma = \bigcup_{F\in \mathcal F} \Sigma_F$ is the set of square matrices over $P_K(E)$ which are sent
  to invertible matrices by the augmentation map.  
   \end{enumerate}
\end{theorem}

\subsection{Definition and first properties of the regular algebra $Q_K(E,C)$.}
\label{subsect:def-of-regular-algebra}
In this subsection  we will define our algebras $Q_K(E,C)$, where $(E,C)$ is an adaptable separated graph. The algebra $Q_K(E,C)$ is defined as a certain universal localization 
of the algebra $\mathcal S _K(E,C)$ introduced in \cite{ABPS}. We briefly recall the definition of $\mathcal S _K(E,C)$ below. Throughout, let $K$ denote a field. 

Let $(E,C)$ be an adaptable separated graph with associated poset $I:=E^0/{\sim}$ (see Definition \ref{def:adaptable-sepgraphs}).

\begin{notation}
\label{not:sigmas}
{\rm \begin{enumerate}
 \item If $p\in I$ is {\bf non-minimal} and {\bf free}, we denote by  $\sigma^p$ the map $\N \to \N$ given by
$$\sigma^p (i) = i+k(p)-1.$$
Moreover, if $1\le j \le  k(p)$, we denote by $\sigma_j^p$ the unique bijective, non-decreasing map from $\{1, \dots ,k(p)\}\setminus \{j\}$ onto $\{1, \dots , k(p)-1\}$.
\item Recall that a connector is an edge $e\in E^1$ such that $s(e)\in E_p^0$ and $r(e) \in E_q^0$, with $q<p$. We will use $\beta$ to denote general connectors, and we remind the reader that the connectors
departing from $v^p$, with $p\in \Ifree$, are of the form $\beta (p,j,s)$ for some $1\le j\le k(p)$ and some $1\le s \le g(p,j)$ (see Definition \ref{def:adaptable-sepgraphs}).
\end{enumerate}}
\end{notation}

The algebra $\mathcal S_K(E,C)$ is the $*$-algebra over $K$ with generators $E^0\cup E^1 \cup \{t_i^v, (t_i^v)^{-1} \mid i\in \N, v\in E^0 \}$ subject to the following relations:

\begin{point}[Relations] \label{pt:KeyDefs} {\rm There are two blocks of relations. In the first block we specify the natural relations arising from the separated
graph structure (cf. \cite{AG12}). In the second block, we give the relations between the generators of $\mathcal S_K (E,C)$, using the special form of our adaptable separated graph.

{\bf Block 1}:

\begin{enumerate}[\rm (i)]
\item For all $v, w\in E^0$, we have $v\cdot w=\delta_{v,w}v$ and $v=v^*$.
\item For all $e\in E^1$, we have:
\begin{enumerate}[\rm (a)]
\item $e=s(e)e=er(e)$
\item $e^*e=r(e)$
\item $e^*f=\delta_{e,f}r(e)$ if $e,f\in X\subseteq C_{s(e)}$.
\item $v=\sum_{e\in X}ee^*$, for $X\in C_v$, $v\in E^0$.
\end{enumerate}
\end{enumerate}

{\bf Block 2}:

\begin{enumerate}
 \item For each {\bf free} prime $p\in I$ and $i=1,\ldots,k(p)$,  we have:
  \begin{enumerate}[\rm (i)]
   \item $ \quad \alpha(p,i)^*\alpha(p,i)=v^p$
    \item $$\alpha(p,i)\alpha(p,i)^*=v^p-\sum^{g(p,i)}_{t=1} \beta(p,i,t)\beta(p,i,t)^* $$
     \item For $i\neq j$, $\,\, \alpha(p,i)\alpha(p,j)=\alpha(p,j)\alpha(p,i)$,   and   $\,\, \alpha(p,i)\alpha(p,j)^*=\alpha(p,j)^*\alpha(p,i)$.
  \item $\beta(p,i,s)^*\beta(p,j,t)=0$ if either $i\neq j$,  or $i=j$ and $s\neq t$. (Note that when $i=j$ and $s\neq t$, these relations follow from the separated graph relations).
  \item $\alpha (p,i)^* \beta (p,i,t) = 0 = \beta (p,i,t)^* \alpha (p,i) $ for all $1\le  i \le k(p)$ and all $1\le t\le g(p,i)$.

  Note that relations (i), (ii) and (v) follow from the separated graph relations, i.e., from the relations given in Block 1.
 \end{enumerate}
\item Moreover, in terms of the $\{t^v_{i}\}$, we impose the following relations:
\begin{enumerate}[\rm (i)]
\item For each $v\in E^0$, $\{ (t_i^v)^{\pm} :  i\in \N \}$ is a family of mutually commuting elements such that
$$ vt_i^v = t_i^v =t_i^v v,\qquad t_i^v(t_i^v)^{-1} = v= (t_i^v)^{-1}t_i^v,\qquad (t_i^v)^*= (t_i^v)^{-1} .$$
\item If $p\in I$ is {\bf regular}, $e\in E^1$ is such that $s(e) \in E_p^0$ and $i\in \N$,
$$t^{s(e)}_i e = e t^{r(e)}_i.$$
\item If $p\in I$ is {\bf free}, $i\in \N$,  $1\le j \le k(p)$ and  $1\le s \le g(p,j)$,
$$(t_i^{v^p})^{\pm} \beta (p,j,s) = \beta (p,j,s) (t^{r(\beta (p,j,s))}_{\sigma^p (i)})^{\pm}, $$
\item If $p\in I$ is {\bf free}, $i\neq j$, and $1\le s \le g(p,j)$,
   $$\alpha(p,i)\beta(p,j,s)=\beta(p,j,s)t^{r(\beta(p,j,s))}_{\sigma^p_j(i)},  \text{ and }\,\,  \alpha (p,i)^* \beta (p,j,s) = \beta(p,j,s)(t^{r(\beta(p,j,s))}_{\sigma^p_j(i)})^{-1}.$$
   \item If $p\in I$ is {\bf free}, $t^{v^p}_i\alpha(p,j)=\alpha(p,j)t^{v^p}_i$ and $t^{v^p}_i\alpha(p,j)^*=\alpha(p,j)^*t^{v^p}_i$ for all $i\in\mathbb N$ and $j\in \{ 1,\dots , k(p) \}$.
\end{enumerate}
\end{enumerate}}
 \end{point}

 \begin{remarks}
  \label{rem:starrelationsOK} {\rm 
  \begin{enumerate}
   \item Since we are working within the category of $*$-algebras, the  $*$-relations of all the relations described in~{\em\ref{pt:KeyDefs}} are enforced in
  the $*$-algebra $\mathcal S_K(E,C)$. However, we warn the reader that the involution $*$ cannot be extended in general to the algebra $Q_K(E,C)$ that we will consider later. 
  \item Although it will not be used in the present paper, we point out that, by \cite[Theorem 4.14]{ABPS}, there is a $*$-isomorphism $\mathcal S_K(E,C) \cong A_K(\mathcal G(E,C))$, 
  where $\mathcal G (E,C)$ is a natural ample groupoid associated to $(E,C)$
  and $A_K(\mathcal G (E,C))$ is the Steinberg algebra of $\mathcal G(E,C)$.  
  \end{enumerate}}
 \end{remarks}

We are now ready to define the algebra $Q_K(E,C)$ as a suitable universal localization of $\mathcal S _K(E,C)$. 

\begin{definition}
 \label{def:Q(E,C)}
 {\rm 
For $v\in E^0$, let $\Sigma _1^v \subseteq v\mathcal S_K(E,C) v$ be the set of all polynomials $p(t_i^v)= 1 + \lambda _1t_i^v+\cdots + \lambda_n(t_i^v)^n\in v\mathcal S_K(E,C)v$, ($n\ge 1$, $\lambda_n\ne 0$).
Consider the universal localization $\mathcal S_K^1(E,C): = \mathcal S_K(E,C)(\bigcup_{v\in E^0} \Sigma_1^v)^{-1}$.

Let $L=K(t_1,t_2,\dots ,)$ be an infinite purely transcendental
extension of $K$. 
For each $v\in E^0$ there is a natural unital embedding $L\to v \mathcal S_K^1(E,C) v$ sending $t_i$ to $t_i^v$. 
For $p(t_i)\in L$, we will denote by $p(t_i^v)$ its image under this embedding. Note that 
relations \ref{pt:KeyDefs}(2)(iii) hold in the form 
\begin{equation}
 \label{eq:commutingwithL}
 f(t_i^{v^p})  \beta (p,j,s) = \beta (p,j,s) \sigma^p (f(t_i^{v^p}))
\end{equation}
whenever $p$ is non-minimal and free, where $\sigma^p\colon L\to L$ is the natural extension of $\sigma^p$ to an endomorphism of $L$.

We now proceed to define sets $\Sigma (p)$ for $p\in I$. We will differentiate between the free and regular cases.
\begin{itemize}
	\item Take $p\in \Ifree$ (cf. \cite{Ara10}). For each polynomial $f(x_i)\in L[x_i:1 \le i\le k(p)]$ in commuting
	variables $\{x_i:1 \le i \le k(p)\}$ and each $j\in \{ 1,\dots , k(p)\}$, write
	$v_{j}(f)$ for the valuation of $f(x_i)$, seen as a polynomial in
	the one-variable polynomial ring $(L[x_i:i \ne j])[x_j]$, at the
	ideal generated by $x_j$. In other words, $v_{j}(f)$ is the
	highest integer $n$ such that $x_j^n$ divides $f$. Write
	$$v(f)=\text{max}\{v_j(f): 1\le j\le k(p)\}.$$
	Note that $\{ \alpha (p,i) : 1\le i \le k(p) \}$ is a family of commuting variables so there is a well-defined evaluation map
	$$L[x_1,\dots ,x_{k(p)}]\to L[\alpha(p,1),\dots ,\alpha(p,k(p))], \quad f(x_i)\mapsto f(\alpha (p,i)).$$
	Let $\Sigma(p)$ be the set of all elements of $v^p\mathcal S_K^1(E,C) v^p$
	given by
	\begin{equation}
	\label{eq:A.6} \Sigma (p)=\{f(\alpha (p,i)): f\in L[x_i] \text{ and } v(f)=0\}.
	\end{equation}
	\item Take $p\in \Ireg$. Here we follow the inspiration provided by \cite{AB}. 
	We consider the graph $E_p$, and we write it in the form $E_p= \varinjlim_F F$, where each $F$ is a complete finite subgraph
	of $E_p$ (see \cite[Section 1.6]{AAS}). Given such complete finite subgraph $F$, we consider the usual path $L$-algebra $P_L(F)$ with coefficients in $L$,
	seen as subalgebra of the corner $v_F \mathcal S _K^1 (E,C) v_F$, where $v_F=\sum_{v\in F^0} v$, and the canonical augmentation map
	$\epsilon ^F \colon P_L(F) \to \oplus_{v\in F^0} vL$. Then $\Sigma (\epsilon^F)$ is the set of all square matrices $A$ over $P_L(F)$
	such that $\epsilon^F(A)$ is invertible as a matrix over $\oplus_{v\in F^0} vL$. Now define
	$$\Sigma (p) = \bigcup_{F} \Sigma (\epsilon^F),$$
	where $F$ ranges over all the complete finite subgraphs of $E_p$.
	\end{itemize}

We can finally define the algebra
$$Q_K(E,C) := \mathcal S _K^1 (E,C)\Big(\bigcup_{p\in I} \Sigma (p) \Big)^{-1}\, ,$$
which will be called {\it the regular algebra} of $(E,C)$. } \qed
 \end{definition}

 The proof of the following lemma follows the same steps as in the proof of \cite[Lemma 2.9]{Ara10}, so we omit it.
 We point out that the idempotent $e(p,q)$ used in that proof must be replaced by the idempotent 
 $v^p-\alpha(p,i)\alpha (p,i)^*$ in our setting. 

\begin{lemma}
\label{lem:commuting} For $p\in \Ifree $, $1 \le i \le  k(p)$ and $f\in
\Sigma (p)$ we have
\begin{equation}
\label{eq:A.19}
(v^p- \alpha (p,i)\alpha (p,i)^*)f^{-1}=(f_0')^{-1}w^*(v^p- \alpha (p,i)\alpha (p,i)^*)=(v^p- \alpha (p,i)\alpha (p,i)^*)(f_0')^{-1}w^*,
\end{equation}
and
\begin{equation}
\label{eq:A.20}
\alpha(p,i)^*f^{-1}=f^{-1}\alpha(p,i)^* + f^{-1}(f_0')^{-1}gw^*(v^p- \alpha (p,i)\alpha (p,i)^*),
\end{equation}
 where $w$ is a monomial in $\{\alpha (p,j) : j\ne i\}$, 
$f_0'\in L[\alpha (p,j): j\ne i]\cap \Sigma (p)$, and $g\in
L[\alpha (p,j) : j=1,2, \dots , k(p)]$.
\end{lemma}

We now present a suitable spanning family for the algebra $Q_K(E,C)$. This extends the 
work done in \cite{Ara10} and in \cite{ABPS}. Recall from \cite{ABPS} that every element of
$\mathcal S_K(E,C)$ can be written as a $K$-linear combination of terms of the form
$$\lambda \mathfrak{m} (p) \nu^*$$
where $\lambda $ and $\nu$ are ``connector paths'' (c-paths for short) and $\mathfrak m (p) $ is a ``monomial'' based at $p\in I$. 
We refer the reader to \cite[Section 2]{ABPS} for the exact meaning of these terms. 

Here we will generalize these notions to the concepts of ``fractional connector path'' (fractional c-path for short)
and ``fractional monomial'' in order to describe $Q_K(E,C)$ in Theorem \ref{thm:linearspanofQ}.

\begin{definition}[Fractional c-path]
\label{def:finitepath}
{\rm Let $(E,C)$ be an adaptable separated graph with associated poset $I$. Then, we define a {\em step} from a vertex $v\in E_p^0$ to a vertex $w\in E_q^0$ with $q<p$, denoted by $\hat\gamma_{v,w}$,as follows:
 \begin{enumerate}
  \item if $v=v^p$  for a free prime $p$, then a step from $v^p$ to $w$ is defined as $$\hat\gamma_{v,w}:=f^{-1} \alpha(p,i)^m\beta(p,i,t)\text{ for some }f\in \Sigma (p),\text{ some }i \text{ and some } m\ge 0, 
  \text{ where } r(\beta(p,i,t))=w.$$
  \item if $v\in E^0_p$ for a regular prime $p$, then a step from $v$ to $w$ is defined as $$\hat\gamma_{v,w}:= f\beta , \text{ with }s(\beta) =v', r(\beta)=w ,$$
  where $v'\in E_p^0$, $f\in vP_L^{{\rm rat}} (E_p)v'\setminus \{ 0 \}$  and $\beta$ is a connector from $v'$ to $w$.
 \end{enumerate}
Then, given two vertices $v\in E_p^0$ and $w\in E_q^0$ in $I$ with $p>q$, we define a fractional c-path from $v$ to $w$ as the concatenation of steps, i.e. we find $p=q_0 > q_1>q_2>\ldots>q_n=q$, and
vertices $v_i\in E^0_{q_i}$, with $v_0=v$ and $v_n=w$, such that $$\gamma_{v,w}:=\hat\gamma_{v_0,v_1}\ldots\hat\gamma_{v_{n-1},v_n}.$$  Moreover, we will say that the fractional $c$-path $\gamma_{v,w}$  has  depth $n$,
and it will be denoted by $\text{depth}(\gamma_{v,w}) =n$.

A {\em trivial} fractional c-path consists of a single vertex $v\in E^0$. These are the fractional c-paths of depth $0$.}\qed

\end{definition}

\begin{remark}\label{Rk:FinitePaths}
	{\rm Note that a c-path in the sense of \cite[Definition 2.4]{ABPS} is a special sort of fractional c-path. Indeed, one just needs to modify $f$ in the latter definition of a step. 
	In particular, in the free prime case one just deletes $f^{-1}$, and one substitutes $f$ by a directed path of finite length connecting $v$ and $v'$ in $E_p$, in the regular case.} 
	\end{remark}

\begin{definition}[Fractional monomial]
\label{def:monomial}{\rm
We continue with our standing assumptions on $(E,C)$.
Now define the fractional monomials as the possible multiplicative expressions one can form using generators (excluding connectors) corresponding to a given prime.
They will be denoted by $\mathbf m(p)$ for $p\in I$. Namely,
\begin{enumerate}
 \item if $p$ is a {\bf free} prime, we define $$\mathbf m (p)=f^{-1}\prod^{k(p)}_{j=1}\alpha(p,j)^{k_j}(\alpha(p,j)^*)^{l_j}, \, \, r\geq 0,
 k_j,l_j\ge 0,$$
 where $f\in \Sigma (p)$. 
 \item if $p$ is a {\bf regular} prime, we define $$\mathbf m(p)= f\nu^*,$$ where $f\in vP_L^{{\rm rat}}(E_p)v'\setminus \{ 0 \}$, and $\nu$ is a finite directed path in $E_p$ with $r(\nu) = v'$ and $v,v'\in E_p^0$. \qed
\end{enumerate}}
\end{definition}

We are now ready to obtain a nice spanning family for our algebra $Q_K(E,C)$. Note that the path $\nu$ in Theorem \ref{thm:linearspanofQ} can be chosen to be a c-path in the sense of \cite[Definition 2.4]{ABPS}.

\begin{theorem}
 \label{thm:linearspanofQ}
 The algebra $Q_K(E,C)$ is the $K$-linear span of the elements of the form
 $\gamma \mathbf m (p) \nu^*$ 
 where $\gamma $ is a fractional c-path, $\mathbf m (p)$ is a fractional monomial at some $p\in I$, and $\nu $ is a c-path . 
 \end{theorem}

\begin{proof}
 We start by noting that the product of two fractional monomials corresponding to the same $p\in I$ is a finite sum of fractional monomials. 
 For $p\in \Ifree$, this follows from \eqref{eq:A.20}, since $f^{-1}\alpha (p,i)= \alpha (p,i)f^{-1}$ for all $i$ and all $f\in \Sigma (p)$,
 together with \cite[Lemma 2.7]{ABPS}. 
  For $p\in \Ireg$, use the formula $e^*f= \tilde{\delta} _e(f) + \tau_e (f)e^*$  for $f\in P_L^{{\rm rat}}(E_p)$ and $e\in E_p^1$ and the fact that $P_L^{{\rm rat}}(E_p)$
 is closed under all the right transductions $\tilde{\delta}_e$, for $e\in E_p^1$ (see Subsection \ref{subsect:prelms-on-univ-local}).
 
 Let $S$ be the $K$-linear span of all the terms $ \gamma \mathbf m (p) \nu^*$ as in the statement. Note that $S$ contains (the image of) $\mathcal S_K^1(E,C)$. Moreover, if $p\in \Ifree$ and $f\in \Sigma (p)$, then
 $S$ clearly contains the element $f(\alpha (p,i))^{-1}$, and if $p\in \Ireg$ and $F$ is a complete subgraph of $E_p$, then, by Proposition \ref{prop:algebraSinAB},  $Q_L(F)$ is contained in $v_FSv_F$.   
If we show that $S$ is a ring, then it follows from the above that all the matrices in $\Sigma (p)$ are invertible over the corresponding corner of $S$, and thus we get that $S=Q_K(E,C)$.   
So, to show that $S= Q_K(E,C)$, it is enough to prove that $S$ is closed under multiplication, which amounts to show that a product of two terms $\gamma _1 \mathbf m (p)  \eta_1^*$
and $\gamma _2 \mathbf n (p')\eta_2^*$ as in the statement can be expressed as a $K$-linear combination of terms of the stated form. 
This was shown to be the case when $\gamma _1, \gamma_2 $ are c-paths and $\mathbf m (p), \mathbf n (p')$ are monomials (in the sense of \cite{ABPS}) in \cite[Proposition 2.13]{ABPS}.
Using the observations in the first paragraph, Lemma \ref{lem:commuting}, and the rules established in \cite[Definitions 2.9 and 2.10]{ABPS}, we see that it suffices to check that 
$\beta (p,j,s)^*f^{-1}\in L\beta (p,j,s)^*$ for $f\in \Sigma (p)$ in case $p\in \Ifree$, and that $\beta^* f \in L\beta^*$ for a connector $\beta$ starting at a vertex of $E_p$
and $f\in P_L^{{\rm rat}}(E_p)$ in case $p\in \Ireg$.  
We have, writing $e(p,j)= v^p-\alpha (p,j) \alpha (p,j)^*$ and using \eqref{eq:A.19} and the relations in Block 2 of \ref{pt:KeyDefs},
\begin{align*}
 \beta (p,j,s)^* f^{-1}  & = \beta (p,j,s)^* e(p,j) f^{-1} = \beta (p,j,s)^* (f_0')^{-1} w^* e(p,j) \\
 & =  (\sigma ^p (f_0') (t_{\sigma_j (i)}^{r(\beta (p,j,s))}))^{-1} w((t_{\sigma _j(i)}^{r(\beta (p,j,s))})^{-1}) \beta (p,j,s)^*
\end{align*}
where $w$ is a monomial in $\{\alpha(p,i): i\ne j\}$ (involving only positive powers of the $\alpha (p,i)$),
and $f_0'\in L[\alpha (p,i): i\ne j]\cap \Sigma (p)$,
and $\sigma ^p(f_0')(t_{\sigma_j(i)}^{r(\beta (p,j,s))})\in L$ is the polynomial
obtained by applying $\sigma ^p$ to all the coefficients of $f_0'$
and replacing $\alpha (p,i)$ with $t_{\sigma_j(i)}^{r(\beta (p,j,s))}$,
for $i \ne j$. 

We now consider the case where $p\in \Ireg$. In this case, we have to deal with a term of the form $\beta^* f$ where $\beta$ is a connector starting at a vertex of $E_p$
and $f\in P_L^{{\rm rat}}(E_p)$. 
By \cite[Theorem 7.1.2]{free}, we can write $f= (a_1\cdots  a_m)(v_FI_m -A)^{-1}(b_1 \cdots b_m)^t$, where $a_i,b_i\in P_L(F)$ and $A\in M_m(P_L(F))$ with $\epsilon_F (A) =0$, where $F$ is a 
suitable finite complete subgraph of $E_p$.
Now using that $\mbox{diag} (\beta^*,\dots ,\beta^*)(v_FI_m-A)= \mbox{diag} (\beta^*,\dots ,\beta^*)$, we get that 
$\mbox{diag} (\beta^*,\dots ,\beta^*)(v_FI_m-A)^{-1}= \mbox{diag} (\beta^*,\dots ,\beta^*)$. Since $\beta^* a_i\in L\beta^*$ and $\beta^* b_i\in L\beta^*$ for all $i$, we get that
$\beta^*f\in L\beta^*$, as desired.  

This concludes the proof.
 \end{proof}

 \subsection{A representation of $Q_K(E,C)$.}
 \label{subsect:representation}
 
We are now going to extend the representations studied in \cite{Ara10} to our context.
These are far-reaching extensions of the usual Toeplitz representation, which provide useful information about the structure of the algebra $Q_K(E,C)$. 
 
We start with an elementary (and well-known) lemma:

\begin{lemma}
 \label{lem:elementary}
 Let $L$ be a field and $z$ an indeterminate, and denote by $L[z]_{(z)}$ the localization of $L[z]$ at the maximal ideal $(z)$.
 Let $\epsilon \colon L[z]\to L$ be the augmentation map. The map $\delta \colon L[z] \to L[z]$ given by
 $$\delta (a_0+a_1z+\cdots + a_mz^m)= a_1+ a_2z+\cdots + a_mz^{m-1}$$
is an $\epsilon$-derivation and can be uniquely extended to an $\epsilon$-derivation $\delta \colon L[z]_{(z)}\to L[z]_{(z)}$ by the formula
$$\delta (fg^{-1})  = \delta (f) g^{-1} - \epsilon (f) \epsilon (g)^{-1} \delta (g) g^{-1}$$
for $f, g \in L[z]$ with $\epsilon (g) \ne 0$.  
Moreover $\text{id} -z\delta = \epsilon$. 
\end{lemma}

\begin{proof}
 The proof is routine. We just check the equality $\text{id} -z\delta = \epsilon$. We clearly have $f= z\delta (f) +\epsilon (f)$ for $f\in L[z]$, so that
 \begin{align*}
 (\text{id} - z \delta) (fg^{-1}) & = fg^{-1} - z\delta (f) g^{-1} + \epsilon (f)\epsilon (g)^{-1} (z\delta (g))g^{-1} \\
 & = \epsilon (f) g^{-1} + \epsilon (fg^{-1}) - \epsilon (f) g^{-1} = \epsilon (fg^{-1}).
   \end{align*}
This completes the proof. 
\end{proof}

Let $(E,C)$ be an adaptable separated graph with associated poset $I$. Let $\mathcal L (I)$ be the lattice (under set inclusion) of lower subsets of $I$. 
For each $J\in \mathcal L (I)$, let $(E_J,C^J)$ be the restriction of $(E,C)$ to the set $E_J^0= \{ v\in E^0 \mid [v]\in J \}$. Let $e(J)$ be the idempotent in the multiplier algebra
of $Q_K(E,C)$ given by $e(J)= \sum_{v\in E_J^0} v$. Then there is a natural homomorphism $\psi_J \colon Q_K(E_J, C^J) \to e(J) Q_K(E,C) e(J)$ sending the generators of $Q_K(E_J, C^J)$ to the 
corresponding generators in $Q_K(E,C)$. Note that this map is surjective by Theorem \ref{thm:linearspanofQ}.

\begin{theorem}
\label{theor:repres} The algebra $Q_K(E,C) $ acts faithfully
by $K$-linear maps on a $K$-vector space
$$V(I)=\bigoplus _{p\in I} V_I  (p),$$
If $J$ is a lower subset of $I$ then the canonical map 
$\psi_J\colon Q_K(E_J,C^J)\longrightarrow e(J)Q_K(E,C)e(J)$ is an isomorphism, and $V(J)=\oplus _{p\in J}V_{I}(p)$. Moreover
$\psi _J(x)(v)=x(v)$ for all $x\in Q_K(E_J,C^J)$ and all $v\in V(J)$.
\end{theorem}

\begin{proof}
We proceed to build the vector spaces and the corresponding actions by order-induction.
Define the action $\tau_I$ with the property that a vertex $v\in E_p^0$ acts non-trivially only on the component $V(p)$, so that, by definition,
the action of $v$ on a component $V(q)$ with $q\ne p$ is $0$.
Therefore, if we want to define the action of a certain element $x$ with $x= v_1xv_2$, where $v_1,v_2\in E_p^0$, we only have to define its action
on $V(p)$. When $p\in \Ifree$, we will have $\tau _I(v^p)(b)= b$ for $b\in V(p)$, and when $p\in \Ireg$, for each $b\in V(p)$ there will be a finite subset $F$ of $E_p^0$ 
such that $\tau _I(v_F)(b)= b$, where $v_F= \sum_{v\in F} v$.  

Set $I^0:=\text{Min}(I)$, the set of minimal elements of $I$. 
For $p\in I^0\cap \Ifree$, set $V(p)=L$ and let $Q_K(E_p,C^p)=L$ act on $V(p)$ by left multiplication. For 
$p\in I^0\cap \Ireg$, set $V(p)= Q_L(E_p)$, and let $Q_K(E_p,C^p)=Q_L(E_p)$ act on $V(p)$ by left multiplication.
Observe that the spaces $V(p)$ defined here are $L$-vector spaces in a natural way. 

Now assume that $J$ is a lower subset of $I$ containing all the minimal elements of $I$, and that we have defined
the $K$-vector spaces $V(q)$ for $q\in J$, with $wV(q)\ne 0$ for each $w\in E_q^0$, and a faithful action $\tau_J$ of $Q_K(E_J,C^J)$ on $V(J)=\bigoplus_{q\in J} V(q)$
with the desired properties. If $J=I$, we have the desired result. If $J\ne I$, let $p$ be a minimal element in $I\setminus J$, and consider the lower subset $J'=J\cup \{ p\}$. 
We will define an action of $Q_K(E_{J'},C^{J'})$ on $V(J')= V(J)\oplus V(p)$.

First we define a structure of $L$-vector space on each $V(q)$ with $q\in J$. If $q\in \Ifree \cap J$, then we define 
$$f(t_i) \cdot b = \tau_J( f(t_i^{v^q}))(b) \in V(q)$$
for $f(t_i)\in L$ and $b\in V(q)$. If $q\in \Ireg \cap J$ and $b\in V(q)$ then there exists a finite subset $F$ of $E_p^0$ such that 
$\tau_J(v_F)(b) = b$. We define, for $f(t_i)\in L$,
$$f(t_i)\cdot b = \tau_J\Big( \sum_{v\in F} f(t_i^v) \Big) (b).$$
It is easy to see that this does not depend on the choice of $F$ and that this gives a structure of $L$-vector space on $V(q)$.

Now we define a suitable vector space $V(p)$. 
Suppose first that $p\in \Ifree$. Let $z_1,z_2,\dots , z_{k(p)}$ be $k(p)$ indeterminates, and let $L[z_j]_{(z_j)}$ denote the localization of $L[z_j]$ at the maximal ideal $(z_j)$. 
Then we define the $K$-vector space 
$$V(p) = \bigoplus _{j=1}^{k(p)} \bigoplus_{s=1}^{g(p,j)} L[z_j]_{(z_j)}\otimes _L V(\beta (p,j,s)),$$
where $V(\beta (p,j,s))$ is the $L$-vector space $\tau_J(r(\beta (p,j,s)))V([r(\beta (p,j,s))])$. 
(Actually $V(p)$ is a left $L$-vector space in a natural way, but this natural structure is not the one induced by $\tau_{J'}$, because of \eqref{eq:ts-actingonVp}.)  
To define the action $\tau_{J'}$ of $Q_K(E_{J'},C^{J'})$, it is enough to define the action of the generators, and to show that the defining relations are preserved in the representation. 
The image $\tau _{J'}(x)$ of any generator $x$ coming from $(E_J,C^J)$ is defined as $\tau _J(x)$, extended trivially to the new factor $V(p)$. By induction, all the relations 
involving these generators will be preserved by $\tau_{J'}$. It remains to define the action on the rest of generators. The action of $v^p$ on $V(p)$ is the identity (and $0$ on $V(J)$).
The action of the generators $\alpha (p,i)$ is given as follows, for $a_j\in L[z_j]_{(z_j)}$ and $b_j\in V(\beta (p,j,s))$:
$$\tau_{J'}(\alpha (p,i)) (a_j\otimes b_j) =\begin{cases}
                                             t_{\sigma^p_j(i)} a_j\otimes b_j & \text{ if } i\ne j\\
                                             z_ja_j\otimes b_j        & \text{ if } i=j . 
                                            \end{cases}$$
                                            The action of $\alpha (p,i)^*$ is given as follows, for $a_j\in L[z_j]_{(z_j)}$ and $b_j\in V(\beta (p,j,s))$:
$$\tau_{J'} (\alpha(p,i)^*) (a_j\otimes b_j) = \begin{cases}
                                             t_{\sigma^p_j(i)}^{-1} a_j\otimes b_j & \text{ if } i\ne j\\
                                             \delta (a_j)\otimes b_j        & \text{ if } i=j , 
                                            \end{cases}$$
where $\delta$ is the endomorphsim of $L[z_j]_{(z_j)}$ described in Lemma \ref{lem:elementary}. One can easily show that the relations (1)(i) and (1)(iii) in \ref{pt:KeyDefs} 
are preserved by this assignment.
The action on the generators $(t_i^{v^p})^{\pm}$ is defined by
\begin{equation}
 \label{eq:ts-actingonVp}
\tau _{J'} ((t_i^{v^p})^{\pm}) (a_j\otimes b_j)  = t_{\sigma^p(i)}^{\pm} a_j\otimes b_j.
 \end{equation}
Note that relations \ref{pt:KeyDefs}(2)(i) and \ref{pt:KeyDefs}(2)(v) are clearly preserved. 
Now observe that, by Lemma \ref{lem:elementary}, for all $1\le i\le k(p)$ we have that the action of $v^p - \alpha (p,i)\alpha (p,i)^*$ on $V(p)$ is precisely the projection onto 
$$\bigoplus _{s=1}^{g(p,i)} L\otimes_L V(\beta (p,i,s)).$$
This gives us a clue on how to define the action of the generators $\beta (p,i,s)$. Namely we define the action of $\beta (p,i,s)^*$ as the 
natural isomorphism  from $L \otimes _L V(\beta (p,i,s))$ onto $V(\beta (p,i,s))$. 
The action of $\beta (p,i,s)^*$ is trivial on the complement of $L \otimes _L V(\beta (p,i,s))$. The action of $\beta (p,i,s)$ is determined by the inverse of the
above isomorphism, so that the image of $\beta (p,i,s)^* \beta (p,i,s)$ is the identity on $V(\beta (p,i,s))$. 
Relations (1)(ii),(iv),(v) in \ref{pt:KeyDefs} are easily checked. 

We check now relation \ref{pt:KeyDefs}(2)(iv). For $i\ne j$ and $b_j\in V(\beta (p,j,s))$, we compute
\begin{align*}
\tau_{J'} & (\alpha (p,i))(\tau_{J'}(\beta (p,j,s))(b_j)) = \tau_{J'} (\alpha (p, i))(1\otimes b_j) = t_{\sigma^p_j(i)}\otimes b_j\\
& = 1\otimes t_{\sigma^p_j(i)}\cdot b_j = \tau_{J'}(\beta (p,j,s))(\tau_{J'}(t^{r(\beta (p,j,s))}_{\sigma^p_j(i)})(b_j)), 
\end{align*}
proving that $\tau_{J'}(\alpha (p,i))\circ \tau_{J'}(\beta (p,j,s)) = \tau_{J'}(\beta (p,j,s)) \circ \tau_{J'}(t^{r(\beta (p,j,s))}_{\sigma_j(i)})$.
The proof of the other equality in \ref{pt:KeyDefs}(2)(iv) and the corresponding $*$-relations is similar. We also leave to the reader to check 
\ref{pt:KeyDefs}(2)(iii). 
Since all the relations in the definition of $\mathcal S_K(E,C)$ are preserved we obtain a well-defined $K$-algebra homomorphism $\mathcal S_K(E,C) \to
\text{End}_K(V(J'))$, which clearly extends to $\mathcal S_K^1(E,C)$.  Let $f(\alpha (p,i))\in \Sigma (p)$. We have to show that $\tau_{J'}(f)$ is invertible 
as a endomorphism of $V(p)$. But for every $1\le j\le k(p)$ the component of $\tau_{J'}(f)$ on the factor $L[z_j]_{(z_j)}\otimes V(\beta (p,j,s))$
is given by left multiplication by a polynomial $p(z_j) = f_0+f_1z_j+\cdots +f_mz_j^m\in L[z_j]$, where $f_i\in L$ and $f_0\ne 0$ because $f\in \Sigma (p)$. Therefore 
$p(z_j)$ is invertible in $L[z_j]_{(z_j)}$ and multiplication by $p(z_j)^{-1}\in L[z_j]_{(z_j)}$ gives the inverse of the restriction of $\tau_{J'}(f)$ to this factor.
This shows that we have a well-defined representation $\tau_{J'}$ from $Q_K(E_{J'}, C^{J'})$ on $V(J')= V(J) \oplus V(p)$.

Now we show that $\tau _{J'}$ is injective. For this purpose, suppose that $x= \sum_{i=1}^r a_i \lambda _i \mathbf m_i (p_i)  \nu_i^*$ is a nonzero element of $Q_K(E_{J'}, C^{J'})$ such that
$\tau_{J'}(x)=0$, where $a_i\in K\setminus \{0 \}$, $\lambda_i $ are fractional c-paths, $\mathbf m _i (p_i)$ are fractional monomials, and $\nu_i$ are c-paths (see Definitions 
\ref{def:finitepath} and \ref{def:monomial} and Theorem \ref{thm:linearspanofQ}).
If $v^px=0$ and $xv^p= 0$, then $x=0$ using the induction hypothesis. So we can assume that either $v^px\ne 0$ or $xv^p\ne 0$. Let us assume the former, a similar argument can be done for the latter.
If $v^px\ne 0$, then we can assume that $\lambda _i = v^p \lambda_i $ for all $i$. Left multiplying by a suitable element of $\Sigma (p)$, we can further assume that each $\lambda _i$ 
is either trivial (i.e. $\lambda_i =v^p$) or of the form $\alpha (p,j)^{m_j}\beta (p,j,s)\cdots $ for some $j's$ and $m_j\ge 0$. 

If $\mathbf m_i(p_i)$ is a fractional monomial in the support 
of $x$ such that $p_i=p$, then by successive replacements of terms $\alpha (p,i)\alpha(p,i)^*$ by $v^p- \sum_{s=1}^{g(p,i)} \beta (p,i,s)\beta (p,i,s)^*$ and after rearranging the terms,
we may assume that these monomials are of the form 
$$f(t_i^{v^p})\alpha(p,1)^{r_1}\alpha(p,2)^{r_2} \cdots \alpha(p,k(p))^{r_{k(p)}},$$ 
where $r_t\in \Z$ and $f$ is a nonzero element of $L$. 
(Here we use the convention that
$\alpha(p,j)^0=v^p$ and $\alpha (p,j)^{-i}= (\alpha(p,j)^*)^i$ for $i>0$.)

With all these standing assumptions, we can further suppose that the family $\{ \lambda_i \mathbf m_i (p_i) \nu_i^*\mid i=1,\dots ,r \}$ is $K$-linearly independent.
Now we observe that the expression of $x$ must involve terms $\gamma_i \mathbf m _i(p_i) \nu_i^*$ with $p_i=p$, which with our present assumptions, means that
$\gamma_i = v^p$, $\nu_i= v^p$ and that $\mathbf m _i = f(t_i^{v^p}) \prod_{j=1}^{k(p)} \alpha(p,j)^{r_j}$ with $r_j\in \Z$ and $f\in L\setminus \{ 0 \}$. Otherwise we can find a 
term $\beta (p,i,s) ^*(\alpha (p,i)^*)^k$ so that
$$ x':= \beta (p,i,s) ^*(\alpha (p,i)^*)^k x \ne 0 .$$
 Now, if $x'v^p\ne 0$, we can find another term $\alpha (p,j)^l \beta (p,j,t) $ so that 
$x'':= x'\alpha (p,j)^l \beta (p,j,t) $ is a nonzero element in the kernel of $\tau_{J'}|_{e(J)Q_K(E_{J'},C^{J'})e(J)}$. Since $v^px''=x''v^p=0$,
this is impossible by the induction hypothesis. If already $x'v^p=0$, then $x'$ itself gives the desired contradiction.    

In conclusion, we can assume that $x= x_0+x_1$, where $x_0$ is of the form $f(\alpha (p,i),\alpha (p,i)^*)$ for a nonzero polynomial $f\in L[x_1^{\pm}, x_2^{\pm},\dots , x_{k(p)}^{\pm}]$,
and where $x_1= \sum_{i=1}^m \gamma _i\mathbf m (p_i) \eta _i^*$, where each $\eta_i$ is a non-trivial c-path, and so it starts with a term of the form
$\alpha (p,j_i)^{u_i}\beta (p,j_i,t_i)$ for some $u_i\ge 0$ and some $j_i,t_i$. Now select any $j$ and $t$ and let $b$ be a nonzero element of $V(\beta (p,j,t))$. 
Let $N$ be a positive integer larger than all the integers $u_i$ above, for $i=1,\dots ,m$. Then we will have that $\tau_{J'}(x_1)(z_j^N\otimes b)=0$.
If we also choose in addition $N$ bigger than all the powers of $\alpha (p,j)^*$ appearing in the expression of $f$, we will obtain that
$$\tau _{J'} (x_0) (z_j^N \otimes b) = g(z_j)\otimes b ,$$
where $g(z_j)\in L[z_j]\setminus \{ 0 \}$.
This shows that $\tau_{J'}(x)(z_j^N\otimes b ) \ne 0$, and we have reached a contradiction. 

Assume now that $p\in \Ireg$. Let $C$ be the family of connectors $\beta $ such that $s(\beta) \in E_p^0$, and set $R=P_L^{\rm rat} (E_p)$
and $V(r(\beta) ) := r(\beta) V([r(\beta)])$ for $\beta \in C$. Also we set $X= E_p^0\setminus s(C)$ and note that by the assumption that $|s_{E_p}^{-1}(v)|\ge 2$, 
$X$ is a subset of regular vertices of $E_p$, 
so that we can consider the relative regular algebra $Q^X_L(E_p)$ (see Definition \ref{def:relative-regular-algebra}).

Define
$$V(p) =  \bigoplus _{\beta \in C} Rs(\beta) \otimes _L  V(r(\beta))  .$$ 
Now we define the action of the generators corresponding to $E_p$. Let $e\in E_p^1$.
Then $\tau_{J'}(e)$ is given by left multiplication by $e$ in any of the factors of the above sum, and similarly, the action of any vertex $v\in E_p^0$ is given by left multiplication. 
For $e\in E_p^1$, the element $e^*$ acts on a
factor $Rs(\beta ) \otimes _L V(r(\beta ))$, by
$$\tau_{J'} (e^*) (r\otimes b) = \tilde{\delta}_e (r)\otimes b .$$
The action of the elements $(t_i^v)^{\pm}$, for $v\in E_p^0$, is also given by left multiplication in all the factors.
Note that $R$ acts by left multiplication on $V(p)$. In particular the image by $\tau_{J'}$ of every element of $\Sigma (p)$ is invertible in (matrices over) $\text{End}_K(V(p))$. 
Now observe that for any $w\in E_p^0$ and $r\in R$, we have the following identity:
$$\sum_{e\in s_{E_p}^{-1}(w)} (L_e\circ \tilde{\delta}_e) (r) + \epsilon_w (r) = wr ,$$
where $\epsilon_w(r)$ is the $w$-component of the augmentation map $\epsilon \colon R\to \bigoplus_{v\in E_p^0} vL$. 
Using this we may easily check that the above assignments give an action of $Q^X_L(E_p)$ on $V(p)$, denoted also by $\tau_{J'}$, and that 
$\tau_{J'} (w- \sum_{e\in s_{E_p}^{-1} (w)} ee^*) $ is nonzero if and only if $w=s(\beta )$ for some $\beta \in C$, and that in this case we have
that $\tau_{J'} (w- \sum_{e\in s_{E_p}^{-1} (w)} ee^*) $ is the projection onto the factor
$$\bigoplus _{\beta \in C\cap s_E^{-1}(w)} wL \otimes_L V(r(\beta ))$$
of $V(p)$. Again this gives us a clue on how to define the action of the connectors. Namely for a connector $\beta\in C$, the action of $\beta ^*$ on $V(p)$ is 
given by the natural isomorphism from  $wL \otimes_L V(r(\beta ))$ onto $V(r(\beta ))$ on this factor, and $0$ on the complement. The action of $\beta$ is given by the inverse of the above isomorphism,
so that $\beta^*\beta $ is precisely the projection onto $V(r(\beta ))$ for every $\beta \in C$. Also it is obvious from the above calculation that
$$ \sum _{e\in s_E^{-1}(w)} ee^* = \sum_{e\in s_{E_p}^{-1}(w)} ee^* + \sum _{\beta \in C\cap s_E^{-1}(w)} \beta \beta^* $$
acts by left multiplication by $w$ on $V(p)$, so that relations of the form \ref{pt:KeyDefs}(ii)(d), for $v\in E_p^0$, are preserved by the representation.

Now using that $t_i \cdot b= \tau_J(t_i^{r(\beta )})(b)$ for each $b\in V(r(\beta))$, we see that relations \ref{pt:KeyDefs}(2)(ii) are satisfied for $p$. 

Therefore we have obtained a representation of $Q_K(E_{J'},C^{J'})$ on $V(J')=V(J)\oplus V(p)$. Note that $wV(p)\ne 0$ for each $w\in E_p^0$. 
It remains to show that it is injective.  
This is similar to the argument above. Suppose that $x$ is a nonzero element in $Q_K(E_{J'}, C^{J'})$ such that
$\tau_{J'}(x)=0$. By an argument similar to the one used above, we can assume that there is $v\in E_p^0$ such that $vx=x$.
Suppose that in the expression $x= \sum _{i=1}^r a_i \gamma _i \mathbf m _i (p_i) \nu_i^*$ given by Theorem \ref{thm:linearspanofQ} we have that $p_i< p$ for all $i=1,\dots ,r$.
We can then write $x$ as a finite sum $x= \sum _{\beta \in C} (\sum _{i=1}^{d_{\beta}} a_i^{(\beta)} \beta b_i^{(\beta)})$, where $a_i^{(\beta)} \in P_L^{\rm rat} (E_p)$, 
and $b_i^{(\beta)} \in r(\beta )Q_K(E_{J'}, C^{J'})$.
Select $\beta _0 \in C$ such that $\sum _{i=1}^{d_{\beta_0}} a_i^{(\beta_0)} \beta_0 b_i^{(\beta_0)} \ne 0$. Using \ref{pt:KeyDefs}(2)(ii), we may assume that the family 
$\{ b^{(\beta_0)}_i : i=1,\dots , d_{\beta_0} \}$ is $L$-linearly independent. We may also assume that all paths $\lambda $ in the support of each
$a^{(\beta _0)}_i$ satisfy that $s(\lambda ) = v$ and $r(\lambda ) = s(\beta _0)$. 
Now let $\gamma \in \text{Path} (E_p)$ be a path of minimal length appearing in the support of the elements $a^{(\beta_0)}_i$, $i=1,\dots , d_{\beta_0}$. By simplicity of notation, let us assume that 
$\gamma $ appears in the support of $a^{(\beta_0)}_1$. 
Setting $v':= s(\beta_0)$, we see that 
all paths in the support of $\gamma^* a^{(\beta_0)}_1$ start and end at $v'$, and in addition we have that $\epsilon _{v'}(\gamma^* a^{(\beta_0)}_1) \ne 0$. 
It follows that $g := \gamma^* a^{(\beta_0)}_1$ is invertible in $v'P^{\rm rat}_L(E_p)v'$. We denote by $g^{-1}$ its inverse in $v'P^{\rm rat}_L(E_p)v'$.

We now compute 
$$\beta_0^* (g^{-1} \gamma^* a_i^{(\beta_0)})= \epsilon_{v'} (g^{-1} \gamma ^* a_i^{(\beta_0)} ) \beta_0^* ,$$
and  
$$\beta_0^* (g^{-1} \gamma ^* a_i^{(\beta )}) \beta = 0 \qquad \text{for } \beta \ne \beta_0 .$$
It follows that $x':=\beta_0^* g^{-1} \gamma^* x = b_1^{(\beta_0)} + \sum_{i=2}^{d_{\beta_0}} \epsilon_{v'} (g^{-1} \gamma ^* a_i^{(\beta_0)} ) b_i^{(\beta_0)}$, and this element is nonzero because the 
family $\{ b_i^{(\beta_0)}\}$ is $L$-linearly independent. We have thus obtained a nonzero  element $x' $ in the kernel of $\tau_{J'}$  such that $wx'=0 $ for all $w\in E_p^0$. 
If $x'w=0$ for all $w\in E_p^0$, then we arrive to a contradiction with the induction hypothesis. If $x'w\ne 0$ for some $w\in E_p^0$ then an easier argument enables us to pick a term of the form
$\gamma' \beta' $ with $\gamma'\in \text{Path}(E_p)$ and $\beta'\in C$ such that $x'':= x'\gamma ' \beta'$ is a nonzero element in the kernel of $\tau_{J'}$ and
belongs to the image of the natural map from $Q_K(E_J,C^J)$ to $Q_K(E_{J'}, C^{J'})$. This is a contradiction. 

Therefore, we can assume that $x= x_0+x_1$, where $x_0$ is nonzero and belongs to the image of the natural map from $Q^X_L(E_p)$ to $Q_K(E_{J'},C^{J'})$ and 
$ x_1= \sum_{i=1}^s \lambda _i \mathbf m _i (p_i) \eta_i^*$, where each $\eta_i$ is a c-path starting with $\alpha _i\beta_i$ with $\alpha_i \in \text{Path}(E_p)$ and $\beta_i \in C$.
Let $v\in E_p^0$ such that $xv\ne 0$. Then we must have that $x_0v\ne 0$. Indeed, if $x_0v=0$ we may apply the above argument to $xv$ to get a contradiction.
We may thus assume that $x=xv$ and $x_iv=x_i$ for $i=0,1$. 
Let $ u = \sum_{\gamma \in \text{Path}(E_p)}a_{\gamma }\gamma^*$ be a lifting in $S:= R\langle E_p^*; \tilde{\delta},\tau \rangle $ of $x_0$ such that $uv=v$, and all
$a_{\gamma}\in R r(\gamma )$. If $xe=0$ for all $e\in s_{E_p}^{-1}(v)$ then we have
$$x= xv= x\Big( \sum _{e\in s_{E_p}^{-1}(v)} ee^* + \sum _{\beta \in C\cap s_E^{-1}(v)} \beta\beta^* \Big) = x\Big( \sum _{\beta \in C\cap s_E^{-1}(v)} \beta\beta^* \Big).$$
This gives again a contradiction, using the above arguments. Iterating this reasoning, we obtain that for any length $k>0$ there exists a path $\mu\in \text{Path} (E_p)$ of length $k$ such that 
$x\mu \ne 0$. Again, we may then deduce that $x_0 \mu \ne 0$. By taking $k$ larger than all the lengths of the paths $\gamma $ in the support of $u$, we shall obtain that
$0\ne u \mu\in P_L^{\rm rat}(E_p)$. Moreover if we choose in addition the length $k$ strictly larger than the lengths of all the paths $\alpha_i$,
$i=1,\dots ,s$, we obtain that $x_1 \mu =0$. Therefore $0\ne x\mu = x_0 \mu $ and it is represented by the element $u':= u\mu \in P_L^{\rm rat} (E_p)$. 
Let $w := s(\beta) $ for some $\beta \in C$, and let $\eta $ be a path in $E_p$ with $s(\eta) = r(\mu)$ and $r(\eta)= w$. Let $b$ be any nonzero element of $V(r(\beta ) )$.
Then $u' \eta \ne 0$ and for the element $ w \otimes b\in Rw\otimes V(r(\beta ))$ we have
$$\tau_{J'}(x\mu \eta )(w\otimes b)= \tau_{J'}(x_0 \mu \eta )(w\otimes b)= u'\eta \otimes b \ne 0.$$
Therefore we get a contradiction. This shows that the action $\tau_{J'}$ is faithful.

Since the poset $I$ is finite, this process terminates and we get the desired faithful representation of $Q_K(E,C)$. Now it is easy to deduce that 
$\psi_J$ is injective for any lower subset $J$ of $I$. Indeed, if $\psi _J(x)=0$ for $x\in Q_K(E_J,C^J)$, then $0=\tau _I (\psi_J(x)) (v) = \tau_J (x) (v)$ for all $v\in V(J)$, and so
$\tau_J(x) = 0$ . Since $\tau_J$ is faithful, we get that $x= 0$.  
\end{proof}

Let us draw some immediate consequences of Theorem \ref{theor:repres}. For $p\in I$ denote by $e(p)$ the idempotent $\sum_{v\in E_p^0} v$ in the multiplier algebra of $Q_K(E,C)$.
Note that $e(p)= v^p$ if $p\in \Ifree$.

Recall from Section \ref{sect:Preliminaries} that $\mathcal L (I)$ denotes the lattice of lower subsets of the poset $I=E^0/{\sim}$, and that, for a ring $R$, $\mathcal L (R)$ denotes the lattice of ideals of $R$.
For each $J\in \mathcal L (I)$, define $\mathcal I (J)$ to be the $K$-linear span of terms of the form $\lambda \mathbf m  (p) \nu^*$ with $p\in J$, where $\lambda $ is a
fractional c-path, $\mathbf m (p)$ is a fractional monomial and $\nu$ is a c-path (see Definitions \ref{def:finitepath} and \ref{def:monomial}). 
It is easy to show that $\mathcal I (J)$ is an ideal of $Q_K(E,C)$, and we have the following result:

\begin{corollary}
 \label{cor:conseqs-rep-thm}
With the above notation, let $p\in I$ and set $J= \{ q\in I : q<p \}$, which is a lower subset of $I$. Then the following holds:
 \begin{enumerate}[(i)]
  \item If $p\in \Ifree$, then 
  there is a natural isomorphism $$v^p Q_K(E,C)v^p/v^p\mathcal I (J) v^p \cong L(z_1,\dots ,z_{k(p)}).$$
  \item If $p\in \Ireg$, then there is a natural isomorphism
  $$e(p)Q_K(E,C)e(p)/e(p) \mathcal I (J) e(p) \cong Q_L(E_p).$$
 \end{enumerate}
\end{corollary}

\begin{proof} Let $J'= J \cup \{ p \}$. 
(i) We can define a $K$-algebra homomorphism
$$\rho \colon Q_K(E_{J'}, C^{J'}) \to L(z_1,\dots , z_{k(p)})$$
by sending $\alpha (p,i)$ to $z_i$, $\alpha (p,i)^*$ to $z_i^{-1}$, $t_i^{v^p}$ to $t_i$ and all the other generators to $0$. By Theorem \ref{theor:repres}, 
$Q_K(E_{J'}, C^{J'})$ can be identified with the algebra $e(J')Q_K(E,C) e(J')$. 
Note that $\rho$ is surjective and that its kernel is precisely the 
ideal $e(J') \mathcal I (J) e(J')$ of $Q_K(E_{J'},C^{J'})$. It now follows that 
$v^p Q_K(E,C)v^p / v^p \mathcal I (J) v^p \cong L(z_1,\dots , z_{k(p)})$, as desired. 

(ii) As in (i), there is a surjective algebra homomorphism $Q_K(E_{J'}, C^{J'})\to Q_L(E_p)$, and using Theorem \ref{theor:repres} we can identify
$Q_K(E_{J'}, C^{J'})$ with $e(J')Q_K(E,C) e(J')$. The same argument as before gives us the desired isomorphism
from $e(p) Q_K(E,C) e(p)/e(p)\mathcal I (J) e(p)$ onto $Q_L(E_p)$. 
\end{proof}

\subsection{A direct sum decomposition of $Q_K(E,C)$}

We now introduce another set of monomials into the picture, which we call the reduced fractional monomials. Basically these monomials constitute a suitable lifting of natural
generating systems of $L(z_1,\dots , z_{k(p)})$ and $Q_L(E_p)$, respectively, with respect to the maps introduced in Corollary \ref{cor:conseqs-rep-thm}.
We will use these monomials to define certain linear subspaces $Q_{(\gamma_1,\gamma_2)}$, which provide a useful direct sum decomposition of $Q_K(E,C)$ (see Theorem \ref{thm:reducedmonomialsdirectsum}). 

\smallskip

As usual we differentiate the free and the regular cases. Let $p$ be a free prime. Set $k:=k(p)$ and $\Sigma := \{ f\in L[z_1,\dots ,z_k]\mid v(f) =0 \}$. 
Note that $L(z_1,\dots ,z_k) $ is the directed union of the $L$-vector spaces $L_f$, for $f\in \Sigma$, where $L_f$ is the $L$-linear span of the family 
$\{ z_1^{r_1}\cdots z_k^{r_k} f^{-1} \mid r_i\in \Z \}$. Here the order in $\Sigma $ is induced by divisibility: $f\le h$ for $f,h\in\Sigma $  if and only if $h=fg$ for $g\in \Sigma$.  

As in \cite{Ara10}, there is a well-defined linear map $T$ sending 
an element $z_1^{r_1}\cdots z_k^{r_k} f^{-1}$ of $L_f$ to the element $(\prod_{j=1}^r \alpha (p,j)^{r_j}) f(\alpha (p,i))^{-1}$. 
Here we note that it is important to place $f(\alpha(p,i))^{-1}$ to the right. Indeed if $h=fg$ is larger than $f$, then the embedding of $L_f$ into $L_h$ is given by replacing 
$f^{-1}$ by $g(fg)^{-1} = gh^{-1}$, expanding $g$ as $L$-linear combination of monomials, and multiplying these monomials with the original monomial $z_1^{r_1}\cdots z_k^{r_k}$.
Since the monomials of $g$ only involve non-negative exponents, this gives a well-defined linear map $T\colon L(z_1,\dots ,z_k) \to v^pQ_K(E,C)v^p$ because we have the identity 
$\alpha(p,i)^*\alpha (p,i) = v^p$, but if we write $f(\alpha (p,i))^{-1}$ on the left, we 
would obtain terms of the form 
$\alpha (p,i)\alpha (p,i)^*$ which cannot be simplified to $v^p$ in $Q_K(E,C)$. Keeping this in mind, we introduce the following definition:

\begin{definition}[Reduced fractional monomial, $p\in \Ifree$]
 \label{def:reducedmons-freecase}
{\rm Let $p$ be an element of $ \Ifree$. A {\it reduced fractional  monomial at} $p$ is a monomial in $Q_K(E,C)$ of the form 
$$(\prod_{j=1}^r \alpha (p,j)^{r_j}) f(\alpha (p,i))^{-1},$$ where $r_j\in \Z$ and $f\in \Sigma (p)$.   
The $L$-linear map $T\colon L(z_1,\dots ,z_k)\to v^pQ_K(E,C)v^p$ defined above provides a linear section of the projection map 
in Corollary \ref{cor:conseqs-rep-thm}(i)} \qed
 \end{definition}

For the case of regular primes, we need the following lemma, whose proof is inspired by the one in \cite[Proposition 1.5.11]{AAS}. See also \cite{AAJZ}. 

\begin{lemma}
 \label{basisforQ}
 Let $E$ be a row-finite graph, let $K$ be a field and set $R= P_K^{\rm rat} (E)$. For each $e\in E^1$, let $\mathcal B_e$ be a $K$-basis for $Re$, and set
 $\mathcal B_v = \bigcup_{e\in s^{-1}(v)} \mathcal B_e \bigcup \{v\}$ for $v\in E^0$. For each non-sink $v\in E^0$ select an edge $e_v\in E^1$ such that
 $s(e_v)= v$. Then the family
 $$\mathcal B '' = \{ f \eta^*\mid f\in  \mathcal B_{r(\eta)} ,\eta \in \mathrm{Path} (E) \} \setminus \{ fe_v^*\nu^* \mid f\in \mathcal B_{e_v} \}$$
 is a $K$-basis for $Q_K(E)$. 
 \end{lemma}

\begin{proof}
 Set $S= R\langle E^*; \tilde{\delta},\tau \rangle $ and $Q= Q_K(E,C)$ . We then have $Q =S/I$, where $I$ is the ideal of $S$ generated by all the idempotents
 $q_v = v - \sum _{e\in s^{-1}(v)} ee^*$ (see \cite[Proposition 2.13]{AB}).  
 We now work in $S$, and recall from \cite{AB} that each element of $S$ can be written in a unique way as a finite sum $\sum_{\gamma \in \text{Path}(E)} a_{\gamma} \gamma ^*$,
 where $a_{\gamma}\in Rr(\gamma )$. It follows that the set 
 $$\mathcal B =  \{ f \eta^*\mid f\in  \mathcal B_{r(\eta)} ,\eta \in \mathrm{Path} (E) \} $$
 is a basis for $S$. We now consider a suitable basis $\mathcal B'$ for $I$. Let
 $$\mathcal B ' = \bigcup _{v\in E^0\setminus \text{Sink}(E)} \{ f q_v \gamma ^* \mid f\in \mathcal B_v, r(\gamma ) = v \}.$$
 To show that $\mathcal B' $ is a basis of $I$, note first that they generate $I$, because 
 $$\gamma ^* q_v = 0 = q_v f$$
 if $\gamma $ has positive length and if $\epsilon_v(f) =0$.
 To show that the elements of $\mathcal B'$ are linearly independent, it suffices to consider terms corresponding to a single idempotent $q_v$.
 Suppose that $\sum _{i=1}^r \lambda _i f_i q_v \gamma _i^*= 0$ in $S$ with all $\lambda_i\in K\setminus \{ 0\}$, where $\{ (f_i,\gamma_i) \}$ are distinct elements in
 $\mathcal B_v\times \{\eta \in \text{Path}(E) : r(\eta ) = v \}$.  
 We may assume that $\gamma _1$ is of maximal length among the $\gamma_i$'s. Expanding this expression, we see that for $e\in s^{-1} (v)$ we have
 $(\sum _{i:\gamma _i=\gamma _1} \lambda_i f_i e) (\gamma_1 e)^*= 0$, which implies that   $\sum _{i:\gamma _i=\gamma _1} \lambda_i f_i e = 0$.
 Since $f_ie$ are linearly independent, we get that $\lambda _1 = 0$, a contradiction. 
 
 Hence we obtain that $\mathcal B '$ is a basis of $I$. To conclude the proof, we only have to check that $\mathcal B'\cup \mathcal B ''$ is a basis
 of $S$. Let $f\gamma ^*$ be an element in the basis $\mathcal B$, with $r(\gamma ) = v$. Since $\gamma $ has finite length, after a finite number of substitutions of the form 
 $e_ve_v^* = v- \sum_{e\in s^{-1}(v), e\ne e_v} ee^* -q_v$, we will arrive at an expression of $f\gamma^*$ as a linear combination of the elements of 
 $\mathcal B'\cup \mathcal B''$.  To show that the family $\mathcal B ' \cup \mathcal B''$ is linearly independent, note that by the above argument (i.e. considering a path $\gamma $ of highest length),
 when we consider a linear combination of different terms $fq_v \gamma ^*$ in $\mathcal B '$ and we expand it, we will obtain a basis term of the form $g e_v^*\gamma ^*$, with $g\in \mathcal B_{e_v}$.
 Therefore this linear combination cannot belong to the linear span of $\mathcal B ''$ in $S$. This shows that $\mathcal B'\cup \mathcal B ''$ is a linear basis of $S$. 
 \end{proof}

We can now define the notion of a reduced fractional monomial for a regular prime $p$. 

\begin{definition}[Reduced fractional monomial, $p\in \Ireg$]
 \label{def:redfract-regcase} {\rm Let $p\in \Ireg$.
Fix choices of edges in $E_p^1$ and $L$-basis in $P_L^{\rm rat} (E_p)e$ as in Lemma \ref{basisforQ} for the graph $E_p$, 
and consider the corresponding $L$-basis $\mathcal B''$ for $Q_L(E_p)$. This choice provides a linear section of the surjection $e(p)Q_K(E,C)e(p) \to Q_L(E_p)$ (see Corollary \ref{cor:conseqs-rep-thm}). 
A {\it reduced fractional monomial at} $p$ is a term of the form $a f\gamma ^*$, where $a\in L\setminus \{0\}$ and  $f\gamma ^*\in \mathcal B ''$, seen in $Q_K(E,C)$.}\qed
 \end{definition}

We are going to use the notion of reduced fractional monomial to prove Theorem \ref{thm:reducedmonomialsdirectsum}. Before doing that, we introduce the definition of the reduced graph, $E_{red}$, 
and use it to define the subspace $Q_{(\gamma_1,\gamma_2)}$ in Definition \ref{def:spacesQgammas}.

 \begin{definition}
 \label{def:reducedgraph}
{\rm  Let $E$ be a directed graph with associated poset $I$. The {\it reduced graph} of $E$ is the graph $\Ered$ such that $\Ered ^0= I$ and such that $\Ered ^1$ is the set of connectors $\beta $  of $E$,
 with $s_{\Ered}(\beta ) = [s(\beta)]$ and $r_{\Ered} (\beta ) = [r(\beta )]$.} \qed
 \end{definition}

 \begin{definition}
  \label{def:spacesQgammas}
 {\rm  Let $\gamma_1= \beta_1\beta _2\cdots \beta_r$ and $\gamma _2 = \beta_1'\beta_2'\cdots \beta_s'$ be paths in $\Ered$ such that $r(\gamma_1 ) = p = r(\gamma _2)$.
  We define the subspace $Q_{(\gamma_1,\gamma_2)}$ as the span of all the terms $\lambda \mathbf m (p)\nu^*$ in $Q_K(E,C)$ such that the connectors involved in 
  the fractional c-path $\lambda $ are exactly $\beta_1,\beta_2, \dots ,\beta_r$, the connectors involved in the c-path $\nu$ are exactly $\beta'_1,\beta'_2,\dots , \beta'_s$,
  and $\mathbf m(p)$ is a reduced fractional monomial at $p$.} \qed
   \end{definition}

 \begin{theorem}
 \label{thm:reducedmonomialsdirectsum}
 Let $(E,C)$ be an adaptable separated graph with associated poset $I$. Then we have
\begin{equation}
\label{eq:A22}
Q_K(E,C) = \bigoplus _{(\gamma_1,\gamma_2)\in \mathcal P} Q_{(\gamma_1,\gamma_2)}  
\end{equation}
where $\mathcal P $ is the set of pairs of finite paths $(\gamma_1,\gamma_2)$ in $\Ered$
such that $r_{\Ered}( \gamma _1) = r_{\Ered}(\gamma_2)$. 
\end{theorem}

\begin{proof}
 The fact that $Q_K(E,C)$ is spanned by the spaces $Q_{(\gamma_1,\gamma_2)}$ follows easily from Theorem \ref{thm:linearspanofQ} and Corollary \ref{cor:conseqs-rep-thm}, due to the 
 fact that the graph $\Ered$ is finite. 
 
 To show that the sum is direct, we follow a strategy similar to the one in the proof of \cite[Lemma 2.11]{Ara10}.  
 
Let $x\in Q_{(\gamma,\gamma')}\setminus \{ 0 \}$ for some $(\gamma,\gamma')\in \mathcal P$. 
 We write $\gamma = \beta_1 \beta_2\cdots \beta_r$ for connectors $\beta_1,\beta_2,\dots ,\beta_r$. 
 Set $p= r_{\Ered}(\gamma) =r_{\Ered}(\gamma')$. We claim  that there are $y_1,y_2\in Q_K(E,C)$ such that
 \begin{enumerate}
  \item $y_2$ is a c-path involving exactly the connectors in $\gamma'$,
  \item $y_1= y_{1r}\cdots y_{12}y_{11}$ where each $y_{1j}$ is either of the form
$\beta (p_j,i_j,s_j)^* (\alpha (p_j,i_j)^*)^{m_j} f_j^{-1}$ for $f_j\in \Sigma (p_j)$, if $s_{\Ered}(\beta_j) =p_j\in \Ifree$, where $\beta_j = \beta (p_j,i_j,s_j)$ for some $i_j,s_j$, or
$y_{1j} = \beta_j^*g_j \gamma_j ^*$, where $g_j\in P_L^{\rm rat} (E_{p_j})$ and $\gamma _j$ is a  finite path in $E_{p_j}$ if $p_j:=s_{\Ered}(\beta_j)\in \Ireg$,
\item $y_1xy_2$ is nonzero and a finite sum of reduced fractional monomials at $p$.  
 \end{enumerate}

 We indicate how to build $y_1$. The easier construction of $y_2$ is left to the reader. 
 Let $\gamma _1 = \beta_2\beta_3\cdots \beta_r$. We will build $y_{11}$ of the desired form so that $y_{11}x$ is a nonzero element
 of $Q_{(\gamma_1,\gamma')}$. This is clearly enough for our purposes. Assume first that $p_1= s_{\Ered}(\beta_1)$ belongs to $\Ifree$, and write $\beta _1 = \beta (p_1,i_1,s_1)$ for some $i_1,s_1$.  
 Then we can choose $f\in \Sigma (p_1)$ such that 
 $$fx = \sum _{l=0}^m \alpha (p_1,i_1)^l \beta (p_1,i_1,s_1) x_l ,$$
 where $x_l \in Q_{(\gamma_1,\gamma')} $ and $x_m\ne 0$. Now take $y_{11}= \beta (p_1,i_1,s_1)^* (\alpha (p_1,i_1)^*)^mf_1^{-1}$ and observe that $y_{11}x = x_m$ has the desired properties. 
 Now suppose that $s_{\Ered}(\beta_1) = p_1 \in \Ireg$. In this case we just follow the proof of the injectivity of $\tau_{J'}$ in Theorem \ref{theor:repres} for $p$ regular.
 Indeed, we take $v\in E_{p_1}^0$ such that $vx\ne 0$ and write $x= \sum _{i=1}^d a_i \beta_1x_i$ where $x_i \in Q_{(\gamma_1,\gamma')}$ are $L$-linearly independent, and $a_i\in vP_L^{\rm rat}(E_{p_1})s(\beta_1)$ for all $i$.   
 Following the above-mentioned proof, we find a finite path $\gamma $, and an invertible element $g$ in a corner of $P_L^{\rm rat}(E_{p_1})$ such that, with $y_{11}= \beta_1^* g^{-1} \gamma ^*$, we have that
 $y_{11}x$ is a nonzero element of $Q_{(\gamma_1,\gamma')}$. This completes the proof of the claim.

  Now suppose that we have a relation $\sum_{i=1}^r x_{(\gamma_i,\gamma_i')}  = 0$, with each $x_{(\gamma_i,\gamma'_i)}\in Q_{(\gamma_i,\gamma_i')}\setminus \{ 0\}$.
 Consider the following partial order on $\mathcal P$:
say that  $(\gamma _1,\gamma _1')\succeq (\gamma _2,\gamma
_2')$ if $\gamma _2=\gamma _1\gamma _3$ and $\gamma
_2'=\gamma_1'\gamma _3'$ for some paths $\gamma _3,\gamma _3'$ in
$\mathcal P$. We may assume that
$(\gamma_1,\gamma _1')$ is maximal with respect to $\succeq$
(among the pairs $(\gamma _i,\gamma _i')$). Let $y_1,y_2$ be the terms 
build in the above paragraph corresponding to the term $x_{(\gamma_1,\gamma_1')}$. 
Then $y_1x_{(\gamma_1, \gamma_1')}y_2$ is nonzero, and a finite sum of reduced fractional monomials at the strongly connected component 
of $r(\gamma _1)$. We set $p=[r(\gamma_1)]$. By the form of the elements $y_1$ and $y_2$ and the maximality of $(\gamma_1, \gamma_1')$, we see that the only terms $(\gamma_i,\gamma_i')$
such that $y_1x_{(\gamma_i,\gamma_i')}y_2$ might be nonzero are precisely those such that $(\gamma _1,\gamma_1') \succeq (\gamma_i,\gamma_i')$.
Therefore we see that $y_1xy_2\in e(p) Q_K(E,C) e(p)$. Assume first that $p\in \Ifree$. Then by Corollary \ref{cor:conseqs-rep-thm}(i) there is a natural surjective homomorphism 
$\pi \colon e(p)Q_K(E,C) e(p) \to L(z_1,\dots , z_{k(p)})$ and by the definition of the reduced fractional monomials, we see that  
$$0 = \pi (y_1(\sum_{i=1}^r x_{(\gamma_i,\gamma_i')})y_2) = \pi (y_1x_{(\gamma_1,\gamma_1')}y_2) \ne 0$$
which is a contradiction. The same argument, using in this case the surjection of Corollary \ref{cor:conseqs-rep-thm}(ii), gives a contradiction in the case where $p\in \Ireg$.
This concludes the proof.
\end{proof}

We are now ready to present a key result, which will be needed later. Recall that, for a lower subset $J$ of $I$,  $\mathcal I(J)$ stands for the $K$-linear span of terms of the form $\lambda{\bf m}(p)\nu^*$, with $p\in J$.

\begin{proposition}
\label{prop:injective-map-lattice ideals}
With the above notation, the map $\mathcal I \colon \mathcal L (I) \to \mathcal L (Q_K(E,C))$, $J\mapsto \mathcal I (J)$, is an injective lattice homomorphism. 
\end{proposition}

\begin{proof}
For $J\in \mathcal L (I)$, Theorem \ref{thm:reducedmonomialsdirectsum} gives the following decomposition:
\begin{equation}
 \label{eq:decomposition-of-IJ}
\mathcal I (J) = \bigoplus _{(\gamma_1,\gamma_2)\in \mathcal P : \, r(\gamma_1)=r(\gamma_2) \in J} Q_{(\gamma_1,\gamma_2)}.
\end{equation}
The injectivity of $\mathcal I$ follows easily from this.  
 
It is quite easy to show directly that $\mathcal I (J_1 \cup J_2) = \mathcal I (J_1) + \mathcal I (J_2)$ for $J_1,J_2\in \mathcal L (I)$. The inclusion $\mathcal I (J_1\cap J_2) \subseteq \mathcal I (J_1) \cap \mathcal I (J_2)$
is obvious. 
To show the remaining inclusion $\mathcal I (J_1) \cap \mathcal I (J_2)\subseteq \mathcal I (J_1 \cap J_2) $, we use the formula \eqref{eq:decomposition-of-IJ}.
We thus obtain that $\mathcal I$ is an injective lattice homomorphism, as desired.
\end{proof}

\section{A cover map}
\label{sect:covermap}

After settling all tools we need for our study, in this section we explain the first step in our strategy to prove our main result. 

Throughout this section, $(E,C)$ will denote an adaptable separated graph. The main goal of this section is to build a certain
adaptable separated graph  $(\tilde E,\tilde C)$
such that the associated poset $\tilde I$ is a forest, and a surjective morphism
$\phi \colon (\tilde E, \tilde C) \to (E,C)$. This new separated graph $(\tilde E, \tilde C)$ will satisfy the following key condition:

\medskip

\noindent {\it Condition {\rm (F)}: Let $(\tilde I , \le )$ be the partially ordered set associated to the pre-ordered set $(\tilde E^0,\leq)$ (see Definition \ref{def:Idefined}(2)).
If $[v]\in \tilde I$ is not a maximal element in $\tilde I$, then there is a unique element $[w]\in
\tilde I\setminus \{ [v] \}$ such that the vertices in the strongly connected component of $w$ emit arrows to the
vertices in the strongly connected component of $v$. Moreover, if $[w]\in \tilde I_{{\rm free}}$, then there is a unique
$X\in \tilde C_w$ such that all the edges from $w$ to $[v]$ belong to $X$.
Specifically there are $v'\in [v]$
and $w' \in [w]$ and $e\in E^1$ such that $s(e)=w'$ and $r(e)=v'$,
and $r(f)\notin [v]$ if $f\in E^1$ and $s(f)\notin ([w]\cup [v])$.
Moreover, if $w\in \tilde I_{{\rm free}}$, then there is a unique
$X\in \tilde C_w$ such that
$$ s_{\tilde E}^{-1}(w)\cap r_{\tilde E}^{-1} ([v])\subseteq  X .$$
(Recall that $[w]=\{ w \}$ if $w\in \tilde{I}_{{\rm free}}$.) The strongly connected component $[w]$ will be called the {\it predecessor component} of $[v]$, and if $w$ is  free, then the element $X\in \tilde C_w$ will be called the
{\it predecessor part} of $[v]$.}

\medskip

We then obtain that the poset $\tilde I$ associated to $(\tilde E, \tilde C)$ is a forest, as follows.

\begin{lemma}
 \label{lem:tildeEforest}
 Let $(\tilde E,\tilde C)$ be an adaptable separated graph. If $(\tilde E, \tilde  C)$ satisfies condition {\rm (F)}, then $\tilde I$ is a forest.
\end{lemma}

\begin{proof}
 Let $[v]\in \tilde I$ be a non-maximal element. Let $[v_0]>[v_1]>\cdots >
[v_r] =[v]$ be a maximal chain, so that $[v_0]$ is a maximal element in $\tilde I$ and each $[v_{i+1}]$ belongs to
the lower cover of $[v_i]$ for each $i=0,1,\dots , r-1$. (This maximal chain exists because the poset $\tilde I$ is finite.)
We claim that $[v_i]$ is the predecessor
of $[v_{i+1}]$. Indeed, let $\gamma =
e_1e_2\cdots e_m $ be a path starting at $v_i$ and ending at
$v_{i+1}$. There exists $1\le t \le m$ such that $[s(e_t)] >
[r(e_{t+1})]= [v_{i+1}]$. Then, by condition {\rm (F)}, $[s(e_t)]$ is the predecessor of $[v_{i+1}]$. Moreover $[v_i]\ge
[s(e_t)]>[v_{i+1}]$, and since $[v_{i+1}]$ is in the lower cover of
$[v_i]$, we conclude that $[v_i]=[s(e_t)]$, so that $[v_i]$ is the
predecessor of $[v_{i+1}]$.

Now suppose that $[w] > [v]$, and let $[w]> [u_1] > \cdots > [u_s]=
[v]$ be a maximal chain between $[w]$ and $[v]$. Just as before, we obtain that $[u_{s-1}]$ is the predecessor of $[v]$, and thus $[u_{s-1}]=[v_{r-1}]$.
Using induction, we obtain that $r\ge s$ and that $[w]= [v_{r-s}]$. Thus every element in the interval $[[v], [v_0]]$ must
be one of the elements in the chain $[v_0]>[v_1]>\cdots > [v_r]
=[v]$. This shows that $\tilde I$ is a forest, whose trees are the posets of the form $\tilde I\downarrow i_0$, where $i_0$ is a maximal element of  $\tilde I$.
 \end{proof}

Notice that the surjective morphism $\phi \colon (\tilde E, \tilde C) \to (E,C)$ is not a morphism in the category $\SGr$ as defined in \cite{AG12}. 
It rather resembles a cover from topology, and thus we have adopted this term to our context
in the next definition.

\begin{definition}
 \label{def:covermorphism}
 Let $(E,C)$ and $(F,D)$ be two adaptable separated graphs. A cover $\phi \colon (F,D) \to (E,C)$ is a graph homomorphism $\phi = (\phi^0,\phi^1)\colon E\to F$ such that the following conditions hold:
 \begin{enumerate}
  \item $\phi ^0$ and $\phi^1 $ are surjective.
  \item For each $v\in F^0$, the map $\phi^1$ induces a bijection $\phi^1_v\colon s_F^{-1} (v) \to s^{-1}_E (\phi^0(v))$ such that $\phi ^1_v(X)\in C_{\phi^0(v)}$ for each $X\in D_v$. In particular, $\phi^1$ induces a bijection
  $X\mapsto \phi ^1(X)$ from $D_v$ onto $C_{\phi^0 (v)}$.
 \end{enumerate}\qed
\end{definition}

Observe that covers are stable under composition, that is, if $\phi
\colon (F_1,D_1) \to (F,D)$ and $\psi \colon (F,D) \to (E,C)$ are
two covers, then $\psi \circ \phi \colon (F_1,D_1) \to (E,C)$ is
also a cover.

Let $(E,C)$ be an adaptable separated graph, with corresponding
poset $I$. If $J$ is a lower subset of $I$, the restriction graph
$E_J$ has a natural structure of separated graph $(E_J,C_J)$, under
which is clearly adaptable. When $J= I\downarrow [v]$ for a vertex
$v\in E^0$, we will denote by $T(v)$ the separated graph
$(E_J,C_J)$, and we will say that $T(v)$ is the {\it tree} of $v$.
Of course, the vertex set of $T(v)$, denoted by $T^0(v)$, is the set
of all $w\in E^0$ such that $v\ge w$. We will also need to consider
the {\it strict tree} of $v$, which is the separated graph $\tilde T
(v)$ obtained by restricting $(E,C)$ to the lower subset $J' = \{
[w] \in E^0 \mid [w] < [v] \}$ of $I$.

We need a last piece of notation. Let $(E,C)$ and $I$ be as above.
For $[v]\in \Ireg$, set
$$X_{[v]} = \{ e\in E^1 \mid s(e)\in [v] \}= s^{-1}([v]).$$ That is, $X_{[v]}$ is the set of arrows departing
from the strongly connected component $[v]$. 
Now write
$$\ol{C} =  \Big( \bigsqcup_{v\in I_{{\rm free}}} C_v\Big) \sqcup \{ X_{[v]} \mid [v] \in \Ireg \}.$$
Set $\ol{C}_v = C_v$ if $v\in \Ifree$ and $\ol{C}_v= \{ X_{[v]}  \}$
if $v\in \Ireg$. Now observe that $(E,C)$ satisfies condition {\rm (F)} if
and only if for each $v\in E^0$ there is at most one $X\in
\ol{C}\setminus \ol{C}_v$ such that $r(X)\cap [v] \ne \emptyset$.

\begin{theorem}
 \label{thm:covering-theorem}
 Let $(E,C)$ be an adaptable separated graph. Then there exists an adaptable separated graph $(\tilde E, \tilde C)$ satisfying condition {\rm (F)} and a cover morphism
 $\phi \colon (\tilde E, \tilde C) \to (E,C)$.
 \end{theorem}

 \begin{proof}
  We proceed by order-induction. Let $J$ be a lower subset of the poset $I$, where $(I,\le )$ is the partially ordered set associated to the pre-ordered set $(E^0,\leq)$. 
  Suppose that we have built an adaptable separated graph $(F,D)$ and a cover morphism $\psi \colon (F,D) \to (E,C)$ satisfying the following properties:
  \begin{enumerate}
   \item For each $v\in  (\psi^0)^{-1}(E_J^0)$ there is at most one
   $X\in \ol{D}\setminus \ol{D}_v$ such that $r(X)\cap [v] \ne
   \emptyset $.
   \item For each $v\in F^0\setminus (\psi^0)^{-1}(E_J^0)$ we have
   $(\psi^0)^{-1}(\{ \psi^0(v) \}) = \{ v \}$.
  \end{enumerate}
If $J=I$ then we have obtained the desired cover map. If $J\ne I$
then we select a minimal element $[v]$ in $I\setminus J$. (Note that
to start the induction we can take $J=\emptyset$.) We set $J'= J\cup
\{ [v] \}$, and let $\ol{v}\in F^0$ be the unique vertex such that
$\psi^0(\ol{v}) = v$ (use condition (2)). We will define an
adaptable separated graph $(F',D')$ and a suitable cover morphism
$\psi ' \colon (F',D') \to (F,D)$. Let $\{ X_1, X_2, \dots , X_r \}$
be the set of those $X\in \ol{D}\setminus \ol{D}_{\ol{v}}$ such that
$r(X)\cap [\ol{v}] \ne \emptyset $. Let $T_1(\ol v), T_2(\ol v),
\dots , T_r(\ol v)$ be a family of mutually disjoint separated
graphs isomorphic to the tree $T(\ol v)$ of $\ol v$, and let
$\sigma_j \colon T_j(\ol v) \to T(\ol v)$ be isomorphisms of
separated graphs, for $j=1,\dots , r$.

Set
$$(F')^0= (F^0 \setminus T^0(\ol v)) \sqcup \Big( \bigsqcup_{j=1}^r
T_j^0(\ol v) \Big) .$$ We now observe that if $e\in F^1\setminus
T^1(\ol v)$ and $r(e)\in T^0(\ol v)$ then necessarily $e\in X_j$ for
some $j$. This is clear by definition of the sets $X_j$ in case
$r(e) \in [\ol v]$. If $r(e) \in T^0(\ol v) \setminus [\ol v]$, then
by minimality of $[v]$ in $I\setminus J$, we have $r(e) \in
(\psi^0)^{-1} (E_J^0)$, so by condition (1) there exists at most one
$X\in \ol{D} \setminus \ol{D}_{r(e)}$ such that $r(X) \cap [r(e)]
\ne \emptyset $. But $r(e) \in T^0(\ol v)\setminus [\ol v]$ and thus
it follows that $X\in \ol{D}_w$ for some $w\in T^0(\ol v)$ and thus
$e\in X\subseteq T^1(\ol v)$, a contradiction. This shows our claim.

Now we define sets of arrows $X_1', X_2',\dots , X_r'$ in our new
graph $F'$. The sets $X_j'$ are in bijection with the sets $X_j$
through a map $e' \longleftrightarrow e$, for $e\in X_j$. For $e\in X_j$, define
$s_{F'}(e') = s_F(e) \in F^0\setminus T^0(\ol v)$ and
$$ r_{F'}(e') =
\begin{cases} r_F(e)\in F^0\setminus T^0(\ol v)  &
\text{if\ } r_F(e) \notin [\ol v] \\
\sigma_j^{-1} (r_F(e))\in T_j^0(\ol v) & \text{if\ } r_F(e) \in [\ol v]
\end{cases} ,$$
where one needs to use the above argument to show that, for $e\in X_j$, $r_F(e) \in F^0\setminus T^0(\ol v)$ if $r_F(e)\notin [\ol v]$.
With this, we can define
the set $(F')^1$ as follows:
$$(F')^1= \Big( F^1\setminus (T^1(\ol v)\sqcup \bigsqcup _{j=1}^r X_j ) \Big)\sqcup \Big(
\bigsqcup_{j=1}^r T_j^1(\ol{v}) \sqcup \bigsqcup_{j=1}^r X'_j \Big).$$ By
the above observation, we have that $r_F(e) \in F^0\setminus T^0(\ol
v)$ for $e\in F^1\setminus (T^1(\ol v)\sqcup \bigsqcup _{j=1}^r
X_j)$, so we can define $s_{F'}(e)= s_F(e)$ and $r_{F'}(e)=r_F(e)$
for these edges $e$.

We now define $D'$. For $w\in \bigsqcup_{j=1}^r T_j^0(\ol v)$, the
set $D'_w$ is the set induced by the structure of separated graph of
$T_j(\ol v)$. If $w\in F^0\setminus T^0(\ol v)$ is free, we simply take the
elements from $D_w$ which are not in the set $\{X_1,\dots X_r \}$,
and we replace the sets $X_j$ such that $X_j\in D_w$ with the
corresponding sets $X_j'$. If $w\in F^0\setminus T^0(\ol v)$ is regular, then $w$ will also be regular in $F'$ and
$D'_w = s_{F'}^{-1}(w)$.

Define $\psi' \colon (F',D') \to (F,D) $ by $(\psi ')^0 (w)= w$ if $w\in F^0\setminus T^0(\ol v)$ and $(\psi ')^0(w) = \sigma_j (w)$ for
$w\in T_j^0(\ol v)$, and
$$(\psi ')^1 (f) =  \begin{cases}
                     f  &
\text{if\ } f\in F^1\setminus (T^1(\ol v)\sqcup \bigsqcup _{j=1}^r X_j) \\
\sigma_j (f) &
\text{if\ } f\in T_j^1(\ol v) \\
e  &
\text{if\ } f= e' \text{ for\ } e\in \bigsqcup_{j=1}^r X_j
 \end{cases}.$$
 Then $\psi'$ is clearly a cover morphism, and thus $\psi \circ \psi' \colon (F',D')\to (E,C)$ is a cover morphism. By construction, the map $\psi \circ \psi '$
 satisfies properties (1) and (2) with respect to $J'=J\cup\{v\}$. This completes the induction step.
\end{proof}

\begin{definition}
\label{def:auxiliarygraph}
{\rm Let $(E,C)$ be an adaptable separated graph. The adaptable separated graph $(\tilde E,\tilde C)$ build in Theorem \ref{thm:covering-theorem} will be called an {\it auxiliary separated graph} of $(E,C)$.}\qed
\end{definition}

\section{Adaptable separated graphs with condition {\rm (F)} }\label{sec:BuildingBlocks}

This section is the milestone of the current paper. Here we study the realization problem for the adaptable separated graphs satisfying condition {\rm (F)}. In particular, the main result obtained in this part is the following:

\begin{theorem}
 \label{thm:main-for-F} Let $(E,C)$  be  an adaptable separated graph satisfying condition {\rm (F)} and such that the associated poset $I$ has a largest element
$i_0= [v_0]$. Then $Q_K(E,C)$ is a separative von Neumann regular ring and the natural map $M(E,C)\to \mathcal V (Q_K(E,C))$ is a monoid isomorphism.
 \end{theorem}

To show this result, we reconstruct $(E,C)$ from a family of ordinary (non-separated) graphs obtained from it, which will be called the {\it building blocks} of $( E, C)$. 
We further show that this reconstruction is well-behaved at all the needed settings: Monoids, K-algebras and the $\mathcal V$-functor. In order to facilitate the understanding of this material, 
we have divided it in different subsections, in each of which we study the different frameworks.

{\bf  Throughout this section, $(E,C)$ will denote an adaptable separated graph satisfying condition {\rm (F)} and such that the associated poset $I$ has a largest element
	$i_0= [v_0]$.}

\subsection{Definition of building blocks}
\begin{definition}
\label{def:buildingblocks-graphs}
{\rm We define a building block of $(E, C)$ as a connected component of an ordinary graph obtained by choosing an element $X_v\in C_v$ for each $v\in E^0\setminus \text{Sink}(E)$.
 More precisely, say that $\varphi \colon E^0\setminus \text{Sink}(E) \to C$ is a {\it choice function}
 if $\varphi (v)\in C_v$ for each $v\in E^0\setminus \text{Sink}(E)$. Given such a choice function $\varphi$, define a graph $E_{\varphi}$ by
 $E_{\varphi}^1 = \bigsqcup _{v\in E^0\setminus \text{Sink}(E)} \varphi (v)$ and $E_{\varphi}^0 = s_E(E_{\varphi}^1) \cup r_E(E_{\varphi}^1)$. The source and range maps
 in $E_{\varphi}$ are defined in
 such a way that the inclusion map $E_{\varphi} \to E$ becomes a graph homomorphism.
 A {\it building block} of $(E,C)$ is a connected component of a graph of the form $E_{\varphi}$. We will denote by $\mathcal F$ the collection
 of all the building blocks of $(E,C)$. Observe that, since $(E,C)$
 satisfies condition {\rm (F)}, the associated poset of each building block is a tree.}\qed
\end{definition}

We will reconstruct $(E,C)$ from the collection $\mathcal F$, and we will be interested in the effect of this process at the monoid level.
Later we will extend this procedure to algebras.

We now define a family of separated graphs $\mathcal F _J$ for each
lower subset $J$ of $\Ifree$. (Here $\Ifree $ has the order induced
from the order of $I$.) Since $[v]=\{ v \}$ when $[v]\in \Ifree$, we
will identify $\Ifree$ with the corresponding subset of vertices of
$E^0$. Given a lower subset $J$ of $\Ifree$, a {\it choice function}
is a function $\varphi \colon \Ifree \setminus (\text{Sink}(E) \cup J) \to C$ such that
$\varphi (v)\in C_v$ for each $v\in  \Ifree \setminus (\text{Sink}(E) \cup J)$. For each
choice function $\varphi$ on $ \Ifree \setminus (\text{Sink}(E) \cup J)$, define a
separated graph $(E_{\varphi}, C^{\varphi})$ by setting
$$E^1_\varphi= \Big(\bigsqcup _{w\in  \Ifree \setminus (\text{Sink}(E) \cup J)} \varphi (w)\Big) \sqcup \Big(
\bigsqcup_{w\in E^0\setminus (\Ifree \setminus J)} s_E^{-1} (w)
\Big)$$ and $E_{\varphi}^0=s_E(E_{\varphi}^1) \cup
r_E(E_{\varphi}^1)$.
 The source and range maps, and the structure of
$C^{\varphi}$ are the natural ones, making the inclusion map
$(E_{\varphi}, C^{\varphi}) \to (E,C)$ a morphism in the category
$\SGr$ defined in \cite[Definition 3.2]{AG12}.  

Let $\mathcal F_J$ be the family of all the connected components of
the separated graphs of the form $(E_{\varphi}, C^{\varphi})$, where
$\varphi$ is a choice function for $\Ifree \setminus J$. For
$J=\emptyset $, we obtain $\mathcal F_{\emptyset} = \mathcal F$. For
$J=\Ifree$, we get $\mathcal F _{\Ifree} = \{ (E,C) \}$.
Note that all the separated graphs in $\mathcal F_J$ 
are adaptable and satisfy condition {\rm (F)}.

In particular, for a lower subset $J$ of $\Ifree$, the members $(F,D)$ of $\mathcal
F_J$ have the following properties:
\begin{enumerate}[\rm(i)]
\item $(F,D)$ is a connected separated graph satisfying condition
{\rm (F)}.
\item $F^0\subseteq E^0$ and  $F^1\subseteq E^1$,
\item For each $v\in F^0$ we have $D_v\subseteq C_v$,
\item For all $v\in J\cap F^0$, we have $D_v = C_v$,
\item For all $v\in F^0\setminus J$ such that $v$ is not a sink in $E$, we have $|D_v|=1$.
\end{enumerate}

Let $J$ be a lower subset of $\Ifree$ containing all the sinks of $E$, and let $v\in \Ifree$ such
that $v$ is minimal in $\Ifree \setminus J$. We may further assume
that $|C_v|>1$, otherwise we would have $\mathcal F _{J}= \mathcal
F_{J\cup \{v\}}$. Let $\varphi \colon \Ifree \setminus (J\cup \{ v
\}) \to C$ be a choice function and let $(F,D)$ be the unique connected component of 
$(E_{\varphi},C^{\varphi})$ such that $v\in F^0$. Write $C_v = \{ X_1,\dots , X_r \}$, and let $\varphi_i \colon 
E^0\setminus J\to C$ be the unique choice function which extends $\varphi$ and such that $\varphi_i (v)= X_i$. 
Let $(F_i,D^i)$ be the unique connected component of $(E_{\varphi_i}, C^{\varphi_i})$ with $v\in F_i^0$.
Observe that $(F,D)\in \mathcal F_{J'}$ and $(F_i,D^i) \in \mathcal F_J$, where $J':= J\cup \{ v \}$. 

In the following lemma, recall that $\tilde T_E(v)$ denotes the strict tree of a vertex $v$ of a separated graph $(E,C)$, thought as a separated graph.

\begin{lemma}
\label{lem:tildeTF}
In the above notation, one has that $\tilde T_E(v)= \tilde T _F (v)= \bigsqcup_{i=1}^r \tilde T_{F_i} (v)$. Moreover $F^0\setminus T^0_F(v) = F_i^0\setminus T_{F_i}^0(v)$,
and the restriction of $(F,D)$ to $F^0\setminus T_F^0(v)$ agrees with the restriction of $(F_i,D^i)$ to $F^0\setminus T_F^0(v)$, for all $i=1,\dots ,r$.
\end{lemma}

\begin{proof} Since all vertices in $\Ifree \cap \tilde T _{F_i}(v)$ belong to $J$, it is clear that $\tilde T _{F}(v)= \tilde T _E (v)$ and that 
$\tilde T_F(v)= \bigcup_{i=1}^r \tilde T_{F_i} (v)$.
 The fact that $\tilde T _{F_i}(v)\cap \tilde T _{F_j} (v)= \emptyset$ if $i\ne j$ follows from the fact that $(E,C)$ satisfies condition {\rm (F)}. 
 Hence, we get $\tilde T_E(v)= \tilde T _F (v)= \bigsqcup_{i=1}^r \tilde T_{F_i} (v)$. The second statement follows directly from the definitions
 of the involved separated graphs, because $\varphi_i$ extends $\varphi $ for all $i$. 
\end{proof}

At this point we need to recall some concepts from \cite{AG12}. 

\begin{definition} \label{defCSsat} 
{\rm Let $(E,C)$ be a finitely separated graph. 
Recall the relation $\ge$ defined on $E^0$ by setting $v\geq w$ if and only if there
is a path $\mu$ in $E$ with $s(\mu)=v$ and $r(\mu)=w$. A subset $H$
of $E^0$ is called \emph{hereditary} if $v \geq w$ and $v\in H$ always imply
$w\in H$. The set $H$ is called $C$-\emph{saturated} provided the following condition holds:

\hspace{4cm} If $X\in C_v$ for some $v\in E^0$ and $r(X) \subseteq H$, then $v\in H$. \qed
}
\end{definition}

By \cite[Corollary 6.10]{AG12}, the set $\mathcal H (E,C)$ of hereditary $C$-saturated subsets of $E^0$ parametrizes the set of order-ideals of $M(E,C)$.

\begin{definition}  \label{defACS}
{\rm Let $H$ be a hereditary, $C$-saturated subset of $E^0$. For any subset $X\subseteq E^1$, define
$$X/H := X\cap r^{-1}(E^0\setminus H).$$
For $H\in \mathcal H (E,C)$, define the quotient separated graph $(E/H,C/H)$, where $(E/H)^0 = E^0\setminus H$,
$(E/H)^1= \{ e\in E^1 \mid r(e)\notin H \}$, and $(C/H)_v= \{X/H \mid X\in C_v \}$ for $v\in E^0\setminus H$. 
(Note that since $H$ is $C$-saturated we get that $X/H \ne \emptyset$ whenever $X\in C_v$ and $v\in E^0\setminus H$.)  

There is a natural quotient map $M(E,C) \to M(E/H,C/H)$ which sends to $0$ all the vertices in $H$. If $M(H)$ denotes the order-ideal of $M(E,C)$ generated by $H$, then $M(E,C)/M(H)\cong M(E/H,C/H)$
(see \cite[Construction 6.8]{AG12}).} \qed
\end{definition}

We now fix the notation used during the remainder of this section. The second paragraph of Notation \ref{not:NotationPullbacks} reproduces, for latter reference, the
hypothesis and notation under which Lemma \ref{lem:tildeTF} has been established.  

\begin{notation}\label{not:NotationPullbacks}
	{\rm Let $(E,C)$ be an adaptable separated graph satisfying condition {\rm (F)}. Denote by $(I,\leq)$ the natural associated poset and let $J$ be a lower subset of $\Ifree$ containing all the sinks of $E$. 
	
	Let $v\in \Ifree$ such that $v$ is minimal in $\Ifree \setminus J$. As observed before, we may further assume that $|C_v|>1$.
	Let $\varphi \colon \Ifree \setminus (J\cup \{ v \}) \to C$ be a choice function and let $(F,D)$ be the unique connected component of 
	$(E_{\varphi},C^{\varphi})$ such that $v\in F^0$. Write $C_v = \{ X_1,\dots , X_r \}$, and let $\varphi_i \colon 
	E^0\setminus J\to C$ be the unique choice function which extends $\varphi$ and such that $\varphi_i (v)= X_i$. 
	Let $(F_i,D^i)$ be the unique connected component of $(E_{\varphi_i}, C^{\varphi_i})$ with $v\in F_i^0$.
	Observe that $(F,D)\in \mathcal F_{J'}$ and $(F_i,D^i) \in \mathcal F_J$, where $J':= J\cup \{ v \}$. 
	
	 Denote $H = \tilde T_F^0(v)$ and $H_i= \tilde T_{F_i}^0(v)$. Then $H$ is a hereditary $D$-saturated subset of $F^0$ and each $H_i$ is a hereditary
$D^i$-saturated subset of $F_i^0$ and a hereditary and $D$-saturated subset of $F^0$. Note that, by Lemma \ref{lem:tildeTF}, we have $H=\bigsqcup_{i=1}^r H_i$. The separated graphs $(F/H,D/H)$ and $(F_i/H_i,D^i/H_i)$ are not equal, but they are similar.
The only difference is that the vertex $v$ emits $r$ different loops $\alpha_1,\alpha _2,\dots , \alpha_r$ in the graph $F/H$, which belong to the $r$ different sets $X_1/H,X_2/H, \dots ,X_r/H$ respectively, and 
the same vertex $v$ emits only one loop $\alpha _i$ in the graph $F_i/H_i$.}
\end{notation}

\subsection{Monoids}\label{subsection:Monoids}

Assuming Notation \ref{not:NotationPullbacks}, we observe that the difference between $(F/H,D/H)$ and $(F_i/H_i, D^i/H_i)$ is not detected by the monoid $M(\cdot,\cdot)$. Namely, we get that 
$$M(F/H,D/H)= M(F_i/H_i, D^i/H_i)$$ for all $i\in\{1,\ldots,r\}$. We denote this common  monoid by $\ol{M}$. 

Gathering everything, we have natural surjective monoid homomorphisms
$$M(F,D) \overset{\theta_i}{\longrightarrow } M(F_i,D^i) \overset{\rho_i}{\longrightarrow} \ol{M}$$
such that $\rho_i\circ \theta_i = \rho_j\circ \theta_j$ for all $1\le i,j \le r$. Note that 
the maps $\theta _i$ can be identified with the quotient map $M(F,D) \to M(F/(\oplus_{j\ne i}H_j), D/(\oplus _{j\ne i}H_j)=M(F,D)/M(\oplus_{j\ne i} H_j)$
and, similarly, $\rho_i$ can be identified with the quotient map $M(F_i,D^i) \to M(F_i/H_i,C/H_i)=M(F_i,D^i)/M(H_i)$. 

We prove next that the maps $\theta _i$ are the limit (pullback) of the maps $\rho_i$.

\begin{theorem}
 \label{thm:monoid-limit}
 Assuming Notation \ref{not:NotationPullbacks}, we have that the family of maps $$ \{ \theta _i \colon M(F,D)\to M(F_i,D^i) \mid i=1,\dots ,r \}$$ is the limit (in the category of commutative monoids) of the system
 of maps $$\{ \rho_i\colon M(F_i,D^i) \to \ol{M} \mid i=1,\dots , r \}.$$  
 \end{theorem}

\begin{proof}
 Let $\{ \gamma_i \colon P\to M(F_i,D^i) \mid i=1,\dots , r \} $ be the limit of the system $\{ \rho_i \}$ in the category of commutative monoids. We will use the usual description of $P$, namely
 $$P = \{ (x_1,x_2,\dots , x_r)\in \prod _{i=1}^r M(F_i,D^i) \mid \rho_i (x_i) = \rho _j(x_j)\,\,  \forall i,j =1,\dots , r \} ,$$
 and $\gamma_i \colon P\to M(F_i,D^i)$ are defined by the projection maps. We have a canonical morphism $\theta \colon M(F,D) \to P$ defined by
 $$\theta (x) = (\theta_1(x),\theta _2 (x) ,\dots , \theta _r(x)),$$
 and we need to show that $\theta$ is a monoid isomorphism. 
 Let us denote by $I_i$ the order-ideal $M(H_i)$ generated by the hereditary $D$-saturated set $H_i = \tilde T^0_i (v)$. Note that $\theta _i (\oplus _{j\ne i} I_j) = 0$.

{\bf Surjectivity of $\theta$:} We provide the proof for $r=2$ since an easy inductive argument allows to extend it to the general case.

In this case, we have the following situation:

\[
\begin{tikzpicture}
  \matrix (m) [matrix of math nodes,row sep=1em,column sep=1em,minimum width=1em]
  {
      & & 0& & 0& \\
         & & I_1& = & I_1& \\
            0&I_2 & M(F,D)& & M(F_1,D^1)&0 \\
               & =& & P& & \\
                  0& I_2& M(F_2,D^2)& & \ol{M}&0 \\
                     & & 0& & 0& \\};

  \path[-stealth]
    (m-1-3) edge [->] node  {} (m-2-3)
    (m-2-3) edge [->] node  {} (m-3-3)
    (m-3-3) edge [->] node  [left] {$\theta_2$} (m-5-3)
             edge [->] node  [above] {$\theta_1$} (m-3-5)
             edge [->] node  [above] {$\theta$} (m-4-4)
    (m-4-4) edge [->] node   {} (m-5-3)
             edge [->] node  {} (m-3-5)
    (m-5-3) edge [->] node  [below] {$\rho_2$} (m-5-5)
            edge [->] node  [] {} (m-6-3)

    (m-3-5) edge [->] node  [right] {$\rho_1$} (m-5-5)

    (m-5-5) edge [->] node  [] {} (m-5-6)
            edge [->] node  [] {} (m-6-5)

    (m-3-5) edge [->] node  [] {} (m-3-6)

    (m-1-5) edge [->] node  {} (m-2-5)
    (m-2-5) edge [->] node  {} (m-3-5)
      (m-3-2) edge [->] node  {} (m-3-3)
      (m-3-1) edge [->] node  {} (m-3-2)
      (m-5-2) edge [->] node  {} (m-5-3)
      (m-5-1) edge [->] node  {} (m-5-2)
    ;

\end{tikzpicture}
\]where $\theta(x)=(\theta_1(x),\theta_2(x))$. Now let $(x,y)\in P$ be such that $\rho_1(x)=\rho_2(y)$ with $x\in M(F_1,D^1)$ and $y\in M(F_2,D^2)$. 
By the diagram, there exists $\tilde x\in M(F,D)$ such that $\theta_1(\tilde x)=x$. Hence, $\theta(\tilde x)=(\theta_1(\tilde x),\theta_2(\tilde x))=(x,\theta_2(\tilde x))$.

By construction, it follows that $\rho_2\circ\theta_2=\rho_1\circ\theta_1$; hence, $$\rho_2(\theta_2(\tilde x))=\rho_1(\theta_1(\tilde x))=\rho_1(x)=\rho_2(y),$$ 
implying the existence of $u_2,u_2'\in I_2$ such that $$\theta_2(\tilde x)+u_2=y+u_2'.$$

Running a similar argument, but now with $y$, one finds $\tilde y\in M(F,D)$ such that $\theta_2(\tilde y)=y$, and so,
$$\theta_2(\tilde x+u_2)=\theta_2(\tilde x)+u_2=\theta_2(\tilde y)+u_2'=\theta_2(\tilde y+u_2') \,\,\text{ since $\theta_2$ is the identity on $I_2$}.$$ 
Therefore, there exist $u_1,u_1'\in I_1$ such that $$\tilde x+u_2+u_1=\tilde y+u_2'+u_1'.$$
Now, since $M(F,D)$ is a refinement monoid, one can build the following refinement matrix:
\[
\begin{tikzpicture}
    \draw[step=1.cm,color=gray] (-2,-2) grid (2,2);
      \matrix(vector)[matrix of nodes, nodes={inner sep=0pt,text width=.9cm,align=center,minimum height=.8cm}]
      {
           & $\tilde x$ & $u_2$ & $u_1$\\
           $\tilde y$ & $y_{1,1}$ & $y_{1,2}$ & $y_{1,3}$\\
             $u_2'$ & $y_{2,1}$ & $y_{2,2}$ & $0$\\
               $u_1'$ & $y_{3,1}$ & $0$ & $y_{3,3}$\\
          };
\end{tikzpicture}
\]
where the two zeros arise since $I_1\cap I_2 = \{ 0 \}$. Set $z:= y_{1,1}+y_{1,2}+y_{3,1}$.
Then we have, using that $y_{1,2}, y_{2,1}\in I_2$,
$$\theta_1(z) = \theta_1(y_{1,1}+y_{1,2}+y_{3,1}) = \theta_1 (y_{1,1}+ y_{3,1})= \theta_1 (y_{1,1}+ y_{2,1} + y_{3,1}) =  \theta_1 (\tilde{x})= x . $$
Similarly, using that $y_{3,1}, y_{1,3}\in I_1$ we get 
$$\theta_2 (z) = \theta_2 (y_{1,1}+y_{1,2}+y_{3,1})= \theta _2 (y_{1,1}+ y_{1,2}+ y_{1,3}) =  \theta_2(\tilde y) = y.$$ 
Therefore, the element $z\in M(F,D)$ satisfies that $\theta(z)=(x,y)$ showing the desired surjectivity.

{\bf Injectivity of $\theta$:} As before, we will just show the injectivity in the case $r=2$.
Let $G$ be the free commutative monoid on the set $F^0$, and recall from \cite[Subsection 2.1]{ABP} that the monoid $M(F,D)$ can be described as 
$G/{\sim}$, where $\sim $ is the congruence on $G$ generated by $v\sim {\bf r}(X)$ for all $v\in E^0$ and $X\in C_v$.
(Here ${\bf r}(X)= \sum_{x\in X} r(x)$.) For $\alpha \in G$, we will denote the class of $\alpha $ in $M(F,D)$ by $\ol{\alpha}$.  
We employ a similar notation for the monoids $M(F_i,D^i) = G_i/{\sim_i}$, where $G_i$ is the free commutative monoid on the set $F_i^0$ and $\sim_i$ is the corresponding congruence, for $i=1,2$. 
We will use the relation $\to $ given in \cite[Definition 2.2]{ABP}.

Now, let $\overline \alpha$ and $\overline \beta$ in $M(F,D)$ be such that $\theta_i(\overline \alpha)=\theta_i(\overline \beta)$ for $i=1,2$. 
We can uniquely write $\alpha= \alpha_0+\alpha_1+\alpha_2$ and $\beta=\beta_0+\beta_1+\beta_2$, where $\mbox{supp}(\alpha_0),\mbox{supp}(\beta_0)\subseteq F^0\setminus H$ and 
$\mbox{supp}(\alpha_i),\mbox{supp}(\beta_i)\subseteq H_i$ for $i=1,2$.
Then, $$\theta_i(\overline \alpha)=\overline{\alpha_0+\alpha_i}=\overline{\beta_0+\beta_i}=\theta_i(\overline \beta)\text{ for }i=1,2.$$ 
Since $\alpha_0+ \alpha_1 \sim \beta_0 + \beta _1$ in $G_1$, it follows from \cite[Lemma 2.4]{ABP} that there exists $\gamma \in G_1$ such that
$\alpha _0+\alpha _1 \to \gamma $ and $\beta_0+\beta _1 \to \gamma $ in $G_1$. Now we can look at $\gamma $ as an element of $G$ and clearly the strings in $G_1$ can be
used to witness the relations $\alpha_0+\alpha_1\to \gamma $ and $\beta _0+\beta_1 \to \gamma $ in $G$. (This uses in a crucial way the fact, which follows from condition (F), 
that the only vertex in $F^0\setminus H$ that emits edges to 
$H$ is the vertex $v$.)
In particular we get $\ol{\alpha_0+\alpha_1}= \ol{\gamma } = \ol{\beta_0+\beta_1}$ in $M(F,D)$. 
Write $\gamma = \gamma _0+\gamma _1$, with $\mbox{supp} (\gamma _0)\subseteq F^0\setminus H$ and 
$\mbox{supp} (\gamma _1) \subseteq H_1$. We have 
$$\ol{\gamma _0 + \alpha _2} = \theta _2 (\ol{\gamma _0}+\ol{\gamma _1} + \ol{\alpha_2}) =\theta _2 (\ol{\alpha}) = \theta_2 (\ol{\beta}) = \theta _2 (\ol{\gamma _0}+\ol{\gamma _1} + \ol{\beta_2})= \ol{\gamma_0+ \beta_2} $$
in $M(F_2,D^2)$, so that $\gamma_0+\alpha_2 \sim \gamma_0+\beta_2$ in $G_2$. Applying again \cite[Lemma 2.4]{ABP}, we obtain $\gamma' \in G_2$ such that
$\gamma_0+\alpha_2 \to \gamma '$ and $\gamma_0 + \beta_2\to \gamma '$ in $G_2$. As above we can look $\gamma'$ as an element of $G$ and we have
$\gamma_0+\alpha_2\to \gamma' $ and $\gamma_0+\beta_2 \to \gamma'$ in $G$. But now we have
$$\alpha = \alpha _0+\alpha _1+\alpha_2 \to \gamma_0 + \gamma_1 + \alpha_2 \to \gamma ' + \gamma _1 $$
and similarly $\beta \to \gamma'+ \gamma _1$, showing that $\ol{\alpha} = \ol{\gamma' + \gamma_1}= \ol{\beta}$ in $M(F,D)$, as desired.  

This concludes the proof of the result.
\end{proof}


\subsection{K-algebras}
\label{subsect:alg-buildingblocks}
In this short subsection, we introduce our basic building blocks for the $K$-algebras. Throughout the subsection $K$ will be a field and $G$ a finite directed graph. In the final part of the subsection we will
give the definition and the key properties of the algebra building blocks $Q_K(F,\sigma)$ corresponding to $F\in \mathcal F$ (see Definition \ref{def:buildingblocks-graphs} for the definition of the family $\mathcal F$). 
This will provide the basis for our inductive arguments.

We first quickly review the theory developed in \cite{A18}. Let $(I,\le )$ be a finite poset. 
Following \cite{A18}, we define a {\it poset of fields} as a family $\mathbf K = \{ K_i : i\in I \}$ of fields $K_i$ with the property that $K_i\subseteq K_j$ if $j\le i$.
Let $G$ be a finite directed graph. We assume that there is a pre-order $\le $ on $G^0$ such that $v\le w$ whenever there is a directed path $\gamma $ such that $s_G(\gamma) = w$ and $r_G(\gamma) = v$, and 
we further assume that the partially ordered set $I:= G^0/{\sim}$, associated to the pre-ordered set $(G^0,\le)$, is a tree with greatest element $i_0:=[v_0]$. Denote by $[v]$ the class of $v\in G^0$ in $I$.

Given a poset of fields $\mathbf K$ over $I$, we define the algebra $P_{\mathbf{K}}((G))$ as the algebra of formal power series of the form $a=
\sum_{\gamma \in \text{Path} (G)} a_{\gamma} \gamma $, where each
$a_{\gamma}\in K_{[r(\gamma )]}$. The usual multiplication of formal
power series gives an structure of algebra over $K_0:=K_{i_0}$ on
$P_{\mathbf{K}}((G))$. Indeed if $(a\gamma)(b\mu)$ is nonzero, where $a\in K_{[r(\gamma)]}$ and $b\in K_{[r(\mu )]}$, then $s(\mu) = r(\gamma )$
and it follows from the property of $\le$ that $[r(\gamma)] \ge [r(\mu )]$ in $I$. Then we have $K_{[r(\gamma)]}\subseteq K_{[r(\mu)]}$ 
and so $ab\in K_{[r(\mu)]} = K_{[r(\gamma \mu)]}$, which shows that the product in $P_{\mathbf{K}}((G))$ is well-defined. 

The path algebra $P_{\mathbf{K}}(G)$ is defined as the subalgebra of
$P_{\mathbf{K}}((G))$ consisting of all the series in
$P_{\mathbf{K}}((G))$ having finite support. We have a natural augmentation homomorphism 
$$\epsilon \colon P_{\mathbf K} ((G)) \longrightarrow \bigoplus _{v\in G^0} K_{[v]}v .$$
We denote by $\Sigma $ the set of all square matrices over $P_{\mathbf K}(G)$ which are sent to invertible matrices by $\epsilon$.

Write $R:=P_{\mathbf{K}}(G)$. For any $v\in G^0$ such that $s^{-1}(v)\neq\emptyset$ we put
 $s^{-1}(v)=\{e^v_1,\dotsc,e^v_{n_v}\}$, and we
consider the left $R$-module homomorphism
 \begin{align*}
  \mu_v\colon Rv&\longrightarrow \bigoplus_{i=1}^{n_v}Rr(e^v_i)\\
  r&\longmapsto\left(re^v_1,\dotsc,re^v_{n_v}\right)
 \end{align*}
 Write $\Sigma_1=\{\mu_v\mid v\in G^0,\,s^{-1}(v)\neq \emptyset\}$.
 
We have

\begin{theorem}[\cite{A18}]
\label{thm:Poset-of-Fields}
With the previous notation, let $Q_{\mathbf K}(G) = P_{\mathbf K}(G) (\Sigma \cup \Sigma_1)^{-1}$. Then the following properties hold:
\begin{enumerate}
 \item $Q_{\mathbf{K}}(G)$ is a hereditary von Neumann regular ring.
 \item The natural map $M(G) \to \mathcal V (Q_{\mathbf K}(G))$ is a monoid isomorphism.  
\end{enumerate}
\end{theorem}

The algebra $Q_{\mathbf K}(G)$ is called the {\it regular algebra of $G$ over the poset of fields} $\mathbf K$. Note that it is an algebra over $K_0$ (where $K_0=K_{i_0}$).    

We are now interested in a particular type of these algebras. Suppose we are given positive integers $n(i)$ for each $i\in I_G\setminus \{ i_0\}$. For $i\in I_G\setminus \{ i_0\}$, let $i=i_k<i_{k-1}<\cdots <i_1<i_0$ be the 
unique maximal chain between $i$ and $i_0$. Then we set $N(i) = \sum_{j=1}^k n(i_j) -k$, and define fields $K_i$ by $K_0=L := K(t_1,t_2,\dots )$ and 
$$K_i = K(t_{-N(i)+1}, t_{-N(i)+2},t_{-N(i)+3},\dots ).$$
Obviously, we have $K_i\subseteq K_j$ if $j\le i$, so that $\mathbf K = \{K_i:i\in I_G\}$ is a poset of fields.

We are going to compare $Q_{\mathbf K}(G)$ with another similar construction. For this, we recall our standing assumption about the separated graph $(E,C)$ through this section, that is, $(E,C)$ is an
adaptable separated graph satisfying condition {\rm (F)} and such that the associated poset $I$ is a tree. 

\begin{definition}
 \label{def:buildingblocks-algebras} {\rm
Let $F\in \mathcal F $ be one of the building blocks for $(E,C)$ (see Definition \ref{def:buildingblocks-graphs}).
Recall that the graph $F$ is a non-separated graph.

Let $\mathcal S_K(F,\sigma)$ be the $*$-algebra with family of generators $F^0\cup F^1\cup \{t_i^v,(t_i^v)^{-1}: v\in F^0, i\in \N \}$ 
and defining relations (\ref{pt:KeyDefs}) thinking of $F$ as a separated graph with the trivial separation, but with a shift in the relations given by
$$t^{s(\beta)}_l \beta  = \beta t^{r(\beta )}_{l+|C_{r(\beta)}|-1}$$
for each connector $\beta $ in $F$ such that $r(\beta)$ is not a sink. Here $C_{r(\beta)}$ refers to our separated graph $(E,C)$. 
Note that this only affects some of the relations \ref{pt:KeyDefs}(ii), \ref{pt:KeyDefs}(iii), the rest of relations remain the same,
with the understanding that the separation on $F$ is the trivial one. Also observe that, for $\beta \in F^1$, $\beta $ is a connector in $F$ if and only if $\beta$ is connector in $E$.  

We then invert the same set of matrices $\Sigma_F$ as given in
Definition \ref{def:Q(E,C)} to get the algebra $$Q_K(F,\sigma) := \mathcal S_K(F,\sigma) \Sigma_F^{-1}.$$ 
Note that since the separation is trivial, 
the set $\Sigma (p)$ corresponding to a non-minimal free prime $p$ consists of univariate polynomials
$f(\alpha (p))\in L[\alpha (p)]$ such that $f(0)\ne 0$. \qed }  
\end{definition}

We are now ready to prove the basic result for our induction arguments.

\begin{proposition}
 \label{prop:basis-for-induction}
 Let $(E,C)$ and $F\in \mathcal F$ be as before. Then the algebra $Q_K(F,\sigma)$ is a separative von Neumann regular ring and the natural map $M(F)\to \mathcal V (Q_K(F,\sigma ))$ is a monoid isomorphism.
 \end{proposition}

\begin{proof}
Let $I=F^0/{\sim}$ be the partially ordered set associated to the pre-ordered set $(F^0,\leq)$, where $\le$ stands for the path-way pre-order $\le $ on $F^0$ (see Definition \ref{def:Idefined}). 
By Definition \ref{def:buildingblocks-graphs}, $I$ is a finite tree. 

Let $\mathcal G $ be the directed set of all finite complete subgraphs $G$ of $F$ such that the map 
$$G^0 \longrightarrow I, \quad v\mapsto [v]$$
is surjective. For $G\in \mathcal G$, let $\le_G$ be the pre-order on $G^0$ determined by $v\le_G w$ if and only if there is a directed path $\gamma $ in $F$ such that
$s(\gamma ) = w$ and $r(\gamma ) = v$. In other words, $\le_G$ is the restriction of $\le$ to $G^0$. Obviously, if $\mu$ is a path in $G$ connecting $w$ to $v$ 
then $v\le_G w$, and therefore the order $\le_G$ contains the path-way pre-order on $G^0$. Moreover it is clear that the map $G^0\to I$ induces an order-isomorphism 
$G^0/{\sim_G}\cong I$.

By using the pre-order $\le_G$, and the poset of fields $\mathbf K$ on $I=G^0/{\sim _G}$ given by the choices $n([v])= |C_v|-1$ for all $[v]\in I\setminus \{ i_0\}$ such that $v$ is not a sink, and $n([v])=0$ if $v$ is a sink, 
we obtain $K_0$-algebras $Q_{\mathbf K}(G)$ for each $G\in \mathcal G$. 
Note that, for $G_1\le G_2$ in $\mathcal G$, we have a natural map $Q_{\mathbf K}(G_1) \to Q_{\mathbf K}(G_2)$ such that the diagram
 \[
     \begin{tikzpicture}
     \matrix (m) [matrix of math nodes,row sep=3em,column sep=3em,minimum width=1em]
     {
        M(G_1)   &   M(G_2) \\
        \mathcal V(Q_{\mathbf K}(G_1)) & \mathcal V(Q_{\mathbf K}(G_2))\\
     };
     \path[-stealth]
     (m-1-1) edge [->] node  [above] {} (m-1-2)
     edge [->] node [left] {$\cong$} (m-2-1)

     (m-1-2) edge [->] node [right] {$\cong$} (m-2-2)

     (m-2-1)       edge [->] node  [below] {} (m-2-2)
     ;
     \end{tikzpicture}
     \]
is commutative. Hence we get a directed system of $K_0$-algebras
$\{ Q_{\mathbf K}(G) : G\in \mathcal G \}$ and setting
 $$Q_{\mathbf K} (F) := \varinjlim_{G\in \mathcal G} Q_{\mathbf K} (G),$$
we see from Theorem \ref{thm:Poset-of-Fields} that $Q_{\mathbf K}(F)$ is von Neumann regular. Moreover, using again Theorem \ref{thm:Poset-of-Fields}, we obtain
 $$M(F) \cong  \varinjlim_{G\in \mathcal G} M(G) \cong \varinjlim_{G\in \mathcal G} \mathcal V (Q_{\mathbf K} (G)) \cong  \mathcal V (Q_{\mathbf K}(F)),$$
 where we use continuity of the $M$-functor on the category of row-finite graphs and complete graph homomorphisms
 (\cite[Lemma 3.4]{AMFP}) and continuity of the $\mathcal V$-functor on algebras.
 
 Define
 $$\varphi \colon Q_K(F,\sigma) \longrightarrow Q_{\mathbf K}(F)$$
 by sending the generators $E^0\cup E^1\cup (E^1)^*$ to the corresponding generators in $Q_{\mathbf K}(F)$, and $$\varphi (t_l^v) = t_{-N([v])+l} v\in Q_{\mathbf K}(F).$$
 Now observing that $N([r(\beta )]) = N([s(\beta)]) + |C_{r(\beta)}| -1$ for every connector $\beta$ in $F$ such that $r(\beta)$ is not a sink, we have that
 $$\varphi (t_l^{s(\beta)} \beta ) =  t_{l-N([s(\beta)])} \beta  = \beta t_{l+|C_{r(\beta)}|-1-N([r(\beta)])}= \varphi (\beta t^{r(\beta)}_{l+|C_{r(\beta)}|-1}) .$$
 It follows that the defining relations $Q_K(F,\sigma )$ are preserved by $\varphi$. Moreover all the matrices in $\Sigma_F$ are clearly invertible over $Q_{\mathbf K}(F)$,
 and this shows that $\varphi$ is well-defined. By using \cite[Proposition 2.7]{A18}, we obtain a well-defined inverse map $\varphi^{-1}$ from $Q_{\mathbf K}(F)$ onto $Q_K(F,\sigma)$. 
 
We have shown that $Q_K(F,\sigma) \cong Q_{\mathbf K}(F)$, and thus the result follows from our previous computations. For the part of the separativity of $Q_K(F,\sigma)$, one needs to recall that a ring with local units $R$ is separative if and 
only if its monoid
$\mathcal V (R)$ is a separative monoid (\cite[Proposition 3.6.4]{AAS}) and that the monoids $M(F)$ associated to a directed graph $F$ are separative (\cite[Theorem 6.3]{AMFP}, \cite[Theorem 3.6.21]{AAS}).
Thus the ring $Q_K(F, \sigma ) $ is separative, because so is $\mathcal V (Q_K(F,\sigma))\cong M(F)$. 
\end{proof}

\subsection{Pullbacks of algebras} \label{subsection:PullbackAlgebras}
In this section, we will introduce algebras the $Q_K(F,D, \sigma)$ for $(F,D)\in \mathcal F_J$, generalizing the definition of the above subsection. 
Working under Notation \ref{not:NotationPullbacks}, we will show in Proposition \ref{prop:QisPullbackofQis} that the algebra $Q_K(F,D,\sigma )$ is the pullback of a family of algebra homomorphisms 
$\{ \rho_i \colon Q_K(F_i,D^i,\sigma_i) \to Q_K(\ol{F},\ol{D},\ol{\sigma}):\,  i=1,\dots , r \}$. This will be used in the next subsection to show inductively Theorem \ref{thm:main-for-F}. 

We start with a definition that extends the one of the previous subsection.

\begin{definition} {\rm We adopt the notation and caveats established in Notation \ref{not:NotationPullbacks}. In particular we have $(F,D) \in \mathcal F_{J'}$ for a fixed lower subset $J$ of $\Ifree$, containing all the sinks of $E$, 
and $v\in \Ifree \setminus J$ 
is a minimal element of $\Ifree \setminus J$ with $|C_v|>1$, where $J'= J \cup \{ v \}$. The algebra $Q_K(F,D, \sigma)$ is the algebra obtained by the same generators and relations than those used in Section \ref{sect:algebras}, but with 
a modification in the definition of the relations \ref{pt:KeyDefs}(2)(ii),(iii) at some particular vertices. Concretely let $w\in \Ifree \setminus J'$ and consider the endomorphism  
$\sigma^w $ of $K(t^w_l)$ given by 
$$\sigma ^w(t_l^w)= t_{l+|C_w|-1}^w ,\qquad (l=1,2,\dots ) . $$
We then modify the relations \ref{pt:KeyDefs}(2)(ii),(iii) for each connector $\beta $ such that $r(\beta)\in \Ifree \setminus J'$ in the following way:
$$f(t^{s(\beta)}_l) \beta = \beta \sigma^{r(\beta)}(f(t_l^{r(\beta)})).$$
In particular, 
$$ t^{s(\beta)}_l \beta = \beta t_{l+|C_{r(\beta)}|-1}^{r(\beta)}$$
for each $l\in \N$. Relations \ref{pt:KeyDefs}(2)(ii),(iii) remain the same for all the other connectors $\beta$ in $F$.  

We will denote this algebra by  $Q_K(F,D, \sigma)$. A similar definition applies to $Q_K(F_i,D^i, \sigma_i)$, where now the new relations involve connectors $\beta $ in $F_i$ such that 
$r(\beta) \in \Ifree \setminus J$ (including $v$).  Recall that we denote by $\alpha_1,\alpha_2,\dots ,\alpha _r$ the different loops at $v$, with $\alpha_i \in X_i$ for all $i$.

The algebras $Q_K(F,D,\sigma)$ have the same essential properties as the algebras $Q_K(F,D)$. In particular, all the results stated in Section \ref{sect:algebras} for $Q_K(F,D)$ hold also for the algebras $Q_K(F,D,\sigma)$ with 
very minor modifications in the proofs.} \qed  
\end{definition}

 We denote by $\mathcal H$ (respectively $\mathcal H_i$) the ideal of $Q_K(F,D,\sigma)$ (respectively $Q_K(F_i, D^i, \sigma_i)$) generated by $H=\tilde T^0_F(v)$ (respectively $H_i=\tilde T^0_{F_i}(v)$). 
 It follows from Proposition \ref{prop:injective-map-lattice ideals} that $\mathcal H=\bigoplus ^{r}_{i=1}\mathcal H_i$.
 We also define the separated graph $(\overline F,\overline D)$ by taking $(\overline F)^0= E^0\setminus H$, $\ol F ^1= F^1\setminus T^1_F(v)$,  $\overline D_v= \emptyset$ and $\overline D _w = D_w$ for 
 $w\in (\overline F)^0 \setminus \{ v\}$.
 Note that $v$ is a sink in $\ol{F}$. The algebra $Q_K(\ol F, \ol D, \ol \sigma)$ is build in a similar way as the algebra $Q_K(F,D,\sigma)$. Indeed for all the connectors $\beta $ in $\ol F$ such that 
 $r(\beta ) \ne v$, we take the same relations as in $Q_K(F,D,\sigma)$. If $\beta $ is a connector in $\ol F$ such that $r(\beta ) =v$, then we set 
 $$t_l^{s(\beta)} \beta = \beta t_{l+r}^v.$$
 We then have a well-defined surjective homomorphism $\theta \colon Q_K(F,D,\sigma) \to Q_K(\ol F, \ol D ,\ol \sigma)$ which is the identity on all generators $\ol F^0 \cup \ol F^1\cup (\ol F ^1)^*
 \cup \{ t_l^w : w\in \ol F^0\setminus \{ v\} \} $, sends all the vertices of $H$ to $0$
 and satisfies
 $$\theta (\alpha _i)=  t^v_i,\quad (i=1,\dots ,r) \qquad \text{and} \qquad \theta (t_l^v) = t_{l+r}^v \quad (l\in \N) .$$ 
 If $\beta$ is a connector in $\ol F$ with $r(\beta ) = v$ then we have 
$$\theta (t_l^{s(\beta)} \beta) = t_l^{s(\beta)} \beta = \beta t_{l+r}^v= \theta (\beta t_l^v) $$ 
so the relation $t_l^{s(\beta)} \beta = \beta t_l^{v} $ in $Q_K(F,D,\sigma)$ is preserved by $\theta $.  
It is easily seen that the kernel of the map $\theta $ is precisely $\mathcal H$ (by defining a suitable inverse map $Q_K(\ol F, \ol D, \ol \sigma) \to Q_K(F,D,\sigma )/\mathcal H)$), so that we get a
short exact sequence 

\begin{equation}
\label{eq:sesH}
0 \longrightarrow \mathcal H \longrightarrow Q_K(F,D,\sigma) \overset{\theta}{\longrightarrow} Q_K(\ol F ,\ol D ,\ol \sigma ) \longrightarrow 0.
\end{equation}

In order to ease notation, from now on, denote $Q:=Q_K(F,D,\sigma )$, $Q_i:=Q_K(F_i, D^i, \sigma _i)$ and $\ol Q:= Q_K(\ol F, \ol D, \ol \sigma )$.  

For $1\le i \le r$, define $\theta _i \colon Q \to Q_i$ which is the identity on the generators $F_i^0\cup F_i^1\cup (F_i^1)^*\cup \{ t_l^w : w\in F_i^0\setminus \{ v \} \}$, sends 
$H_j$ to $0$ for $j\ne i$,
and satisfies 
$$\theta _i (\alpha _j) = t^v_{\sigma^v _i (j)} \quad (j\ne i),\qquad \theta _i (t_l^v) = t^v_{l+r-1} \quad (l\in \N).$$
(See Notation \ref{not:sigmas}(1) for the definition of $\sigma^v_i$.)
Note that $\theta _i(\alpha_i) = \alpha_i$. Let us check that $\theta _i$ is a well-defined homomorphism. The only critical points are the relations of type \ref{pt:KeyDefs}(2)(ii),(iii) at $v$ for $Q$ and $Q_i$ respectively. 
Suppose first that we have a connector $\beta$ in $F$ such that $r(\beta ) = v$. Then, since $r(\beta ) \in J'$, the relations for the connector $\beta $ are the ones prescribed
in \ref{pt:KeyDefs}(2)(ii),(iii) for the separated graph $(F,D)$. But, since $s(\beta )\in F^0\setminus J'$ we have $|D_{s(\beta)}|= 1$ and so $t_l^{s(\beta)}\beta= \beta t_l^v$ for all $l\in \N$.
In the algebra $Q_i$, however, we have that $r(\beta ) = v\in F_i^0\setminus J$ and thus we get the modified relation $t_l^{s(\beta )} \beta = \beta t_{l+r-1}^v$ . 
Hence, we compute
$$\theta _i (t^{s(\beta)}_l \beta) = t^{s(\beta)}_l \beta = \beta t^v_{l+r-1} = \theta _i (\beta t^v_l),$$
and we see that the relation $t_l^{s(\beta)}\beta= \beta t_l^v$ is preserved by $\theta_i$. 

Now consider a connector $\beta \in X_i$ (so that $s(\beta ) = v$). In the algebra $Q$, we have that $v\in J'$ and so the relation for $\beta $ is the one prescribed by \ref{pt:KeyDefs}(2)(iii) for the separated graph $(F,D)$.
Since $|D_v| = r$ we get $t^v_l \beta = \beta t^{r(\beta )}_{l+r-1}$. Taking into account that $r(\beta) \in J$ and that $D^i_v = \{ X_i \}$, we have in the algebra $Q_i$ that
$t^v_l \beta = \beta t^{r(\beta)}_l$ for all $l\in \N$. Hence we get
$$\theta_i (t^v_l \beta) = t^v_{l+r-1} \beta = \beta t^{r(\beta)}_{l+r-1} = \theta _i ( \beta t^{r(\beta)}_{l+r-1})$$
and the relation  $t^v_l \beta = \beta t^{r(\beta )}_{l+r-1}$ is preserved by $\theta_i$. 

One can similarly show that the relation:
$$\alpha_j \beta = \beta t_{\sigma^v_i(j)}^{r(\beta )} \quad  \text{for} \quad \beta \in X_i \quad \text{and}\quad  j\ne i ,$$
which is valid in $Q$, is preserved by $\theta_i$. 

Now we define similar maps $\rho_i \colon Q_i \to \ol Q$, for $1\le i \le r$. Here we send all   
 generators $\ol F^0 \cup \ol F^1\cup (\ol F ^1)^*
 \cup \{ t_l^w : w\in \ol F^0\setminus \{ v\} \} $ to the corresponding generators in $\ol Q$, we send all the vertices in $H_i$ to $0$, and we let
 $$\rho_i (\alpha _i)=  t^v_i,\quad  \text{and} \qquad \rho_i (t_l^v) = \left\{\begin{array}{ll}
t_l^v& \text{ if }l<i\\
t^v_{l+1}& \text{ if }l\geq i\\
\end{array}\right .$$ 
Then $\rho_i$ is a well-defined surjective homomorphism with kernel $\mathcal H_i$, so that we get a short exact sequence 
\begin{equation}
\label{eq:sesHi}
0 \longrightarrow \mathcal H _i\longrightarrow Q_i \overset{\rho_i}{\longrightarrow} \ol Q \longrightarrow 0.
\end{equation}
Moreover, we have $\rho_i\circ \theta_i = \theta $ for all $i\in \{1,\dots , r\}$.

In the following Proposition, we show that the maps $\theta_i$ are the limit (pullback) of the maps $\rho_i$.
%
%

\begin{proposition}
\label{prop:QisPullbackofQis}
With the above notation, we have that the family 
	of maps $$\{\theta_i:Q\to Q_i\mid i=1,\ldots, r\}$$ is the limit (in the category of $K$-algebras) of the system of maps $$\{\rho_i:Q_i\to \overline Q\}.$$
   \end{proposition}
\begin{proof}
   Since $\theta = \rho_i\circ \theta_i$ for all $i$, the following diagram 
         \[
      \begin{tikzpicture}
      \matrix (m) [matrix of math nodes,row sep=1em,column sep=1em,minimum width=1em]
      {
      	Q & & & & \\
      	& Q_2   &Q_1&\\
      	& \ldots \ldots&&\\
      	& Q_r   &&\overline Q\\};
      \path[-stealth]
      (m-1-1) edge [bend left=40, dashed,->] node [below] {$\theta_1$} (m-2-3)
      edge [dashed,->] node [below] {$\theta_2$} (m-2-2)
      edge [bend right=40, dashed,->] node [below]  {$\theta_r$} (m-4-2)
      
      (m-2-3) edge node [right] {$\rho_1$} (m-4-4)
      (m-2-2) edge node [right] {$\rho_2$} (m-4-4)
      (m-4-2) edge node [below] {$\rho_r$} (m-4-4)       ;
      \end{tikzpicture}
      \]
   is commutative.  By the universal property of the pullback, denoted by $P$, there exists a unique morphism $\mu:Q\to P$ such that $\pi_i\circ \mu = \theta_i$ for $i=1,\dots ,r$, 
   where $\pi_i\colon P\to Q_i$ are the structural maps of the pullback, so that $\rho_i\circ \pi_i= \rho_j\circ \pi_j$ for all $i,j$. 
   For $r=2$ we can visualize the situation  in the following diagram:  
    \[
\begin{tikzpicture}
  \matrix (m) [matrix of math nodes,row sep=1em,column sep=1em,minimum width=1em]
  {
     Q& & & & \\
     & P   &Q_1&\\
       & Q_2 &\overline Q\\};
  \path[-stealth]
    (m-1-1) edge [bend left=40, dashed,->] node  [above] {$ \theta_1$} (m-2-3)
            edge [->] node [below] {$\mu$} (m-2-2)
            edge [bend right=40, dashed,->] node  [below] {$ \theta_2$} (m-3-2)

    (m-2-3) edge node [right] {$\rho_1$} (m-3-3)
         (m-2-2) edge node [above] {$\pi_1$} (m-2-3)
         (m-2-2) edge node [right] {$\pi_2$} (m-3-2)
             (m-3-2) edge node [below] {$\rho_2$} (m-3-3)       ;
\end{tikzpicture}
\]
Let $\pi\colon P\to \ol Q$ denote the composition $\rho_i\circ \pi_i$ (which does not depend on $i$).   
By the usual definition of the pullback, $$P:=\{(x_1,\dots ,x_r)\in \prod_{i=1}^r Q_i\mid \rho_i(x_i)=\rho_j(x_j) \forall i,j\},$$ and one has that 
 $\mathcal H_1\times \cdots \times \mathcal H_r$ is an ideal in $P$ and the quotient algebra is isomorphic to $\ol Q$. By using this observation and the exact sequence (\ref{eq:sesH}), we can build the following commutative 
 diagram with exact rows:
\[
\begin{tikzpicture}
  \matrix (m) [matrix of math nodes,row sep=3em,column sep=4em,minimum width=2em]
  {
     0&\mathcal H&Q &\overline Q&0\\
      0& \prod_{i=1}^r\mathcal H_i & P&\overline Q&0 \\};

  \path[-stealth]

    (m-1-1) edge [->] node  {} (m-1-2)
    (m-1-2) edge [ -> ]node [right] {$=$} (m-2-2)
    (m-1-2) edge [->] node  {} (m-1-3)
    (m-1-3) edge [->] node   [above] {$\theta$} (m-1-4)
        edge node [right] {$\mu$} (m-2-3)
    (m-1-4) edge [->] node  {} (m-1-5)
        edge [ -> ]node [right] {$=$} (m-2-4)
      (m-2-1) edge [->] node  {} (m-2-2)

    (m-2-2) edge [->] node  {} (m-2-3)
     (m-2-3) edge [->] node  [above] {$\pi$} (m-2-4)
     (m-2-4) edge [->] node  {}(m-2-5);

\end{tikzpicture}
\]
Therefore, we have that $\mu$ is an isomorphism by the Five's Lemma. This shows the result.  
\end{proof}

\subsection{The functor $\mathcal V$ on pullbacks}
In this last part of the section we will study the behaviour of the described pullbacks under the functor $\mathcal V(\_)$. This will enable us to provide the proof of Theorem \ref{thm:main-for-F}.

 Our main tool here is \cite[Theorem 3.2]{Ara10}, which for our purposes we state in the following way:

\begin{theorem}
\label{thm:K1Pullbacktheorem}
    Let $Q_1,\ldots,Q_r$ be separative von Neumann regular rings, and let $\rho_i : Q_i\to \overline Q$ be surjective homomorphisms. 
    Let $\theta_i:P\to Q_i$ be the limit (pullback) of the morphisms $\rho_i:Q_i\to \overline Q$. 
    Then, P is a separative von Neumann regular ring, and the maps $\mathcal V(\theta_i):\mathcal V(P)\to \mathcal V(Q_i)$ are the limit of the family of maps $\mathcal V(\rho_i):\mathcal V(Q_i)\to \mathcal V(\overline Q)$ in 
    the category of 
    monoids if and only if for each idempotent $e=(e_1,\ldots, e_r)$ in $P$, we have that for $i=1,\ldots, r$,
    $$K_1(\rho_i(e_i)\overline Q\rho_i(e_i))=(\rho_i)_*(K_1(e_iQ_ie_i))+\Big( \bigcap_{j\neq i}(\rho_j)_ *(K_1(e_jQ_je_j)) \Big).$$ 
\end{theorem}

In order to accomplish our goal, we will need the following. Note the essential use of property {\rm (F)} in its proof, and recall that $G(\_)$ stands for the Grothendieck group construction explained in Section \ref{Section1}. 

\begin{lemma}
 \label{lem:kernelOK}
 Assume Notation \ref{not:NotationPullbacks}, and fix $i\in \{ 1,\dots , r \}$. Let $M(H_i)$ be the order-ideal of $M(F_i,D^i)$ generated by $H_i$. 
 Then $M(H_i)$ is canonically isomorphic to $M((F_i)|_{H_i}, (D^i)^{H_i})$, the monoid associated to the restriction of $(F_i,D^i)$ to
 $H_i$ (Lemma \ref{lem:restriction-monoids}), and the kernel of the natural map
 $$G(M(H_i)) \longrightarrow G(M(F_i,D^i))$$
 is the cyclic subgroup of $G(M(H_i))$ generated by $ \sum _{\beta \in X_i'} x_{r(\beta)}$, where we denote by $x_w$ the generators of $G(M(H_i))$, and $X_i'= X_i\setminus \{\alpha_i\}$ is the family of connectors of $X_i$.  
 \end{lemma}

\begin{proof}
 By Lemma \ref{lem:restriction-monoids}, we have $M(H_i) \cong M((F_i)|_{H_i}, (D^i)^{H_i})$. 
 
 Let $G$ be the free abelian group on $F_i^0$. We write $G=G_1\oplus G_2$, where $G_1$ is the free group on $H_i$ and $G_2$ is the free group on $F_i^0\setminus H_i$. 
 Let $\pi \colon G\to G(M(F_i, D^i))$ be the natural projection map. Set $L:= \ker \pi $. 
 Then $L=L_1\oplus L_2$, where $L_1\subseteq G_1$ and $L_2\subseteq G_2$. 
 Indeed $L_1$ is the subgroup of $G_1$ generated by all the elements of the form $w- \sum_{e\in X} r(e)$ for $X\in C_w$ and $w\in H_i$ and the element $x := \sum _{\beta \in X_i'} r(\beta )$, and $L_2$
 is generated by all the elements of the form $d_X = w- \sum_{e\in X} r(e)$ for $X\in D^i_w$ and $w\nleq v$. Thus it suffices to observe that all elements $d_X$ as above 
 have their support completely contained in $F_i^0\setminus H_i$. 
 But this follows because our graph satisfies condition {\rm (F)}, so the only vertex in the set $H_i\cup \{v\}$ that can receive an edge emitted by a vertex in $E^0\setminus (H_i\cup \{v\})$ is the vertex $v$, and so the elements $d_X$ defined above
 are completely supported on the generators of $G_2$. 
 
 On the other hand, we have
 $$G(M(H_i)) =  G(M((F_i)|_{H_i}, (D^i)^{H_i})) = G_1/L_3 ,$$
 where $L_3$ is the subgroup of $G_1$ generated by all the elements of the form $w- \sum_{e\in X} r(e)$ for $X\in C_w$ and $w\in H_i$, thus $L_3\subseteq L_1$ and the map 
 $$G_1/L_3 =G(M(H_i)) \to G/L=G(M(F_i,D^i))$$ factors 
 as follows 
 $$ G_1/L_3 \longrightarrow G_1/L_1 \longrightarrow G_1/L_1\oplus G_2/L_2 = G/L $$
 thus the kernel of this map is precisely the element $ \sum _{\beta \in X_i'} x_{r(\beta)}$. 
\end{proof}

We next show that the $K_1$ statement in Theorem \ref{thm:K1Pullbacktheorem} holds in our situation.

\begin{proposition}
 \label{prop:conditionforK1}
    Under Notation \ref{not:NotationPullbacks} and the above notation, assume that $Q_i$ are separative regular rings and that the natural maps 
    $M(F_i,D^i)\to \mathcal V (Q_i)$ are isomorphisms for $i=1,\dots ,r$. Then, for each idempotent $(e_1,\dots , e_r) $ in $P\cong Q$ and each $1\le i \le r$ we have:
    $$ K_1(\rho_i(e_i)\overline Q\rho_i(e_i))=(\rho_i)_*(K_1(e_iQ_ie_i))+ \Big( \bigcap_{j\neq i}(\rho_j)_ *(K_1(e_jQ_je_j))\Big) .$$
Moreover, $Q=Q_K(F,D,\sigma)$ is a separative regular ring, and the natural map $M(F,D) \to \mathcal V (Q)$ is an isomorphism. 
\end{proposition}

\begin{proof}
Let $(e_1,\dots , e_r)$ be an idempotent in $P\cong Q$. Then each $e_i$ is an idempotent in $Q_i=Q_K(F_i,D^i,\sigma_i)$ and $\ol e = \rho_i(e_i)$ for all $i=1,\dots , r$.

We want to analyze the following exact sequence in $K$-theory, for a given $i\in \{1,\dots , r \}$:
$$K_1(e_iQ_ie_i) \overset{(\rho_i)_*}\longrightarrow K_1(\ol e \ol Q \ol e) \overset{\partial_i}{\longrightarrow} K_0(e_i\mathcal H _i e_i) \overset{\iota_i}\longrightarrow K_0(e_iQ_ie_i) .$$

Since $Q_i$ is regular by hypothesis, we have $\mathcal L (Q_i) \cong \mathcal L (\mathcal V (Q_i))$ by Proposition \ref{prop:basicpropsof-VR}(2). In addition we have $\mathcal V (Q_i) \cong M(F_i,D^i)$ also by hypothesis, therefore we get 
an isomorphism
$$ \mathcal L (Q_i) \cong \mathcal L (\mathcal V (Q_i)) \cong \mathcal L (M(F_i,D^i)) \cong \mathcal L (I_i) , $$
where $\mathcal L (I_i)$ is the lattice of lower subsets of $I_i:= F_i^0/{\sim}$, the partially ordered set associated to the pre-ordered set $(F_i^0,\leq)$ with respect to the path-way pre-order 
(use \cite[Proposition 2.9]{ABP} and \cite[Proposition 1.9]{AP16}).

Hence, there is a lower subset $I_i'$ of $I_i$ such that the ideal $Q_ie_iQ_i$ of $Q_i$ corresponds to $I_i'$. Similarly, the ideal $\mathcal H_i e_i \mathcal H_i = Q_ie_iQ_i  \cap \mathcal H_i$
corresponds to a lower subset $I_i''$ of $I_i$ such that $I_i''\subseteq I_i'$. 
We let $M_i'$ (respectively $M_i''$) denote the order-ideal of $M(F_i,D^i)$ generated by the hereditary subsets $H_{I_i'}$ and $H_{I_i''}$ respectively.
(Recall that, for a lower subset $L$ of $I_i$, we denote by $H_L$ the hereditary subset of $F_i^0$ consisting of all the vertices $w\in F_i^0$ such that $[w]\in L$.)
Observe that, using \cite[Corollary 5.6]{G09}, we have 
$$\mathcal V (e_iQ_ie_i) = \mathcal V (Q_ie_iQ_i) \cong M_i',\qquad  \mathcal V (e_i\mathcal H_ie_i) = \mathcal V (\mathcal H_ie_i\mathcal H_i) \cong M_i''.$$  



We distinguish two cases. Suppose first that $v\notin I_i'$. Then the  map
$$K_0(e_i\mathcal H_i e_i) \longrightarrow K_0(e_i Q_ie_i),$$
which corresponds to the map
$$G(M_i'') \longrightarrow  G(M_i')$$
is injective. In this case, we get $K_1(\ol e \ol Q \ol e)= (\rho_i)_* (K_1(e_iQ_ie_i))$ and we are finish. 

Now assume that $v\in I_i '$. Then necessarily $I_i''= H_i/{\sim}$ and $M_i''=M(H_i)$. 
The map 
$$\iota _i \colon K_0(e_i\mathcal H_i e_i) \longrightarrow K_0(e_i Q_ie_i)$$
corresponds to the natural map
$$\eta \colon G(M(H_i))\longrightarrow  G(M_i') ,$$
and the element $\sum _{\beta \in X_i'} x_{r(\beta )} \in G(M(H_i))$ belongs to the kernel of $\eta$, where we are using the notation introduced in Lemma \ref{lem:kernelOK}.
On the other hand, the canonical map $G(M(H_i))\to G(M(F_i,D^i))$ considered in Lemma \ref{lem:kernelOK} factors through $\eta$, so that we conclude from that lemma that the kernel of
$\eta$ is precisely the cyclic subgroup of $G(M(H_i))$ generated by  $\sum _{\beta \in X_i'} x_{r(\beta )} $. Since  $\sum _{\beta \in X_i'} x_{r(\beta )} $ corresponds to 
$\sum_{\beta \in X_i'} [\beta \beta^*]$ under the isomorphism $K_0(e_i\mathcal H_i e_i)\cong G(M(H_i))$, we conclude that the kernel of $\iota_i$ is precisely the cyclic subgroup of 
$K_0(e_i\mathcal H_i e_i)=K_0(\mathcal H_i)$  generated by $\sum_{\beta \in X_i'} [\beta \beta^*]$.   

So the cokernel of the map $(\rho_i)_* \colon K_1(e_iQ_ie_i) \to K_1(\ol e \ol Q \ol e)$ is isomorphic to the cyclic subgroup generated by $\sum_{\beta \in X_i'} [\beta \beta^*]$, and since 
$$\partial_i ([t_i^v]) = \sum_{\beta \in X_i'} [\beta \beta^*] $$
(cf. \cite[p. 110]{AP17}),
we get that it suffices to show that $[t_i^v] \in (\rho_j)_*(K_1(e_jQ_je_j))$ for all $j\ne i$. But this is certainly true by the definition of the map $\rho_j$.
Indeed, if $i<j$ then $t_i^v = \rho_j (t_i^v)$ and if $i>j$ then $t_i^v= \rho_j(t_{i-1}^v)$.

This shows that the $K_1$ conditions in Theorem \ref{thm:K1Pullbacktheorem} are satisfied in our situation. Since, by Proposition \ref{prop:QisPullbackofQis},
the family of maps $\{\theta_i:Q\to Q_i\mid i=1,\ldots, r\}$ is the limit of the system $\{\rho_i:Q_i\to \overline Q\}$, we obtain from Theorem \ref{thm:K1Pullbacktheorem}
that $Q= Q_K(F,D,\sigma)$ is a separative von Neumann regular and that the family of maps $\{ \mathcal V (\theta_i):\mathcal V (Q)\to \mathcal V (Q_i)\mid i=1,\ldots, r\}$ is the limit of the 
system $\{ \mathcal V (\rho_i):\mathcal V (Q_i)\to \mathcal V (\overline Q )\}$. Now, by hypothesis, the natural maps $M(F_i,D^i) \to \mathcal V (Q_i)$ are isomorphisms for all $i=1,\dots , r$, and thus it follows from 
Theorem \ref{thm:monoid-limit} and the naturality of all the morphisms involved that the canonical map $M(F,D)\to \mathcal V (Q)$ is an isomorphism. 

This concludes the proof.
\end{proof}

We can finally prove Theorem \ref{thm:main-for-F}.

\medskip

\noindent{\it Proof of Theorem \ref{thm:main-for-F}:} The proof is by order-induction with respect to the separated graphs in the families $\mathcal F_J$, for the lower subsets $J$ of $\Ifree$, defined at the beginning of this section.
Indeed, we will show by order-induction that for any lower subset $J$ of $\Ifree$, all the algebras $Q_K(F,D, \sigma)$, for $(F,D)\in \mathcal F_J$, satisfy the conclusions of Theorem \ref{thm:main-for-F}.
Since $\mathcal F _{\Ifree} = \{ (E,C) \}$ and $Q_K(E,C,\sigma) = Q_K(E,C)$, the result follows from this. 

When $J=\emptyset$, the family $\mathcal F_{\emptyset} $ is just the class $\mathcal F$ of building blocks of $(E,C)$ (Definition \ref{def:buildingblocks-graphs}). 
For $F\in \mathcal F$, the algebra $Q_K(F,\sigma)$ satisfies the properties in the thesis of Theorem \ref{thm:main-for-F} by Proposition \ref{prop:basis-for-induction}. 

This establishes the basis for the induction. Now let $J$ be a lower subset of $\Ifree$ such that all the algebras $Q_K(F,D,\sigma)$ with $(F,D)\in \mathcal F_J$ satisfy the conclusions of Theorem \ref{thm:main-for-F},
and let $v$ be a minimal element in $\Ifree \setminus J$.  We may further assume that $J$ contains all the sinks of $E$ and that $|C_v|>1$. We can now apply Proposition \ref{prop:conditionforK1} to deduce that 
the conclusions of Theorem \ref{thm:main-for-F} hold for all the separated graphs $(F,D)\in \mathcal F_{J\cup \{ v \}}$. Now the result follows from the fact the poset $I$ is finite.   \qed



\section{Push-Outs}\label{sec:Pushouts}
In this final section we plan to explain the behaviour of the push-out construction in our setting. In particular, we develop the last step of the strategy displayed in the introduction. A related method was used in \cite{Ara10}, and as happens there, we will subsequently work with the notion of a {\em crowned pushout}, which we describe below.

Let $(E,C)$ be an adaptable separated graph, and let $\phi \colon (\tilde{E},\tilde{C}) \to (E,C)$ be a cover morphism, in the sense of Definition \ref{def:covermorphism}, where $ ( \tilde{E}, \tilde{C})$
satisfies condition {\rm (F)}, see Theorem \ref{thm:covering-theorem}. Then it follows from Theorem \ref{thm:main-for-F} that $Q_K(\tilde{E},\tilde{C})$ is a von Neumann regular ring and that
the natural map $M(\tilde{E},\tilde{C})\to \mathcal V (Q_K(\tilde{E},\tilde{C}))$ is an isomorphism. In this section we will show that the same properties hold for $(E,C)$. 

Recall the definitions and notations introduced in Section \ref{sect:covermap}.

By the proof of Theorem \ref{thm:covering-theorem}, there is a finite chain of adaptable separated graphs $(F_i,D^i)$, for $i=0,\dots , m$,  such that $(F_0,D^0)= (\tilde{E},\tilde{C})$ satisfies condition {\rm (F)},
$(F_m,D^m)=(E,C)$, and each pair $((F_i,D^i), (F_{i+1}, D^{i+1}))$ satisfies the conditions in the following definition:

\begin{definition}
 \label{def:crownedpair}
{\rm  Let $(E_1,C^1)$ and $(E_2,C^2)$ be two adaptable separated graphs. We say that the pair $((E_1,C^1), (E_2,C^2))$ is a {\it crowned pair} if there is a cover morphism
 $\phi\colon (E_1,C^1)\to (E_2,C^2)$ of separated graphs and vertices $v_1,v_2\in E_1^0$ such that $v:= \phi^0(v_1)=\phi^0(v_2)$ and: 
 \begin{enumerate}[(i)]
  \item For each $w\in  \tilde{T}^0(v)$ there is at most one
   $X\in \ol{C^2}\setminus (\ol{C^2})_w$ such that $r(X)\cap [w] \ne
   \emptyset $.
   \item $T^0(v_1)\cap T^0(v_2) =\emptyset$,
  \item $\phi$ induces an isomorphism of separated graphs from $T (v_i)$ to $T(v)$, for $i=1,2$, 
  \item Let $E_1'$ be the restriction of the graph $E_1$ to the set of vertices $E_1^0\setminus (T^0(v_1)\sqcup T^0(v_2))$ and let $E_2'$ be the restriction of the graph $E_2$ to the set of vertices $E_2^0\setminus T^0(v)$.
  Then $\phi$ restricts to a graph isomorphism from $E_1'$ onto $E_2'$. 
  \item Let $w\in (E_1')^0$ and let $X\in (\ol{C^2})_{\phi^0(w)}$ be such that $r(X)\cap [v]\ne \emptyset$. Since $\phi$ is a cover map, it follows from (i)-(iv) that there is exactly one $Y\in (\ol{C^1})_w$ such that $\phi^1(Y)=X$
  (see Lemma \ref{lem:crownedproperties}(3) below). 
  We ask that there is exactly one  $i\in \{ 1,2 \}$ such that  $r(Y)\cap [v_i] \ne \emptyset$.  \qed
   \end{enumerate}}
\end{definition}

We collect in the next Lemma several useful properties of the cover maps that appear in Definition \ref{def:crownedpair}. We denote by $I_i$ the posets $E_i^0/{\sim}$ of strongly connected components, for $i=1,2$

\begin{lemma}
\label{lem:crownedproperties}
Let $\phi \colon (E_1,C^1)\to (E_2,C^2)$ be a cover morphism between adaptable separated graphs $(E_1,C^1), (E_2,C^2)$. 
Assume that conditions (i)-(iv) in Definition \ref{def:crownedpair} hold. Then the following properties hold:
\begin{enumerate}
 \item For $w\in E_1^0$, we have $[w]\in (I_1)_{\rm free} \iff [\phi^0(w)]\in (I_2)_{\rm free}$.
 \item For each $w\in E_1^0$ and each $X\in (\ol{C^1})_w$ we have that $\phi^1(X) \in (\ol{C^2})_{\phi^0(w)}$. Moreover, $\phi^1$  restricts to a bijection from $X$ onto $\phi^1(X)$.  
 \item For each $w\in E_1^0$ and each $X\in (\ol{C^2})_{\phi^0(w)}$ there is exactly one $Y\in (\ol{C^1})_{w}$ such that
 $\phi^1(Y)=X$.
 \item If $w\in \tilde{T}^0(v_i)$ for some $i$, then there exists exactly one $X\in \ol{C^1}\setminus (\ol{C^1})_w$ such that $r(X)\cap [w]\ne \emptyset$. Moreover, we have $X\in (\ol{C^1})_{w'}$ for $w'\in T^0(v_i)$. 
 \item Suppose that in addition condition (v) also holds, so that $((E_1,C^1),(E_2,C^2))$ is a crowned pair. Then if $w\in (E_1')^0$, and $X\in (\ol{C^1})_{w}$
 is such that $r(X)\cap T^0(v_i) \ne \emptyset$ for some $i$, then 
 $$r(X)\cap (T^0(v_1) \cup T^0(v_2)) = r(X)\cap [v_i].$$
\end{enumerate}
\end{lemma}

\begin{proof}
 (1) If $w \in T^0(v_i)$ for some $i$, this holds by condition (iii) in Definition \ref{def:crownedpair}. Otherwise, if $w\in (E_1')^0$, then the result follows from the cover property and condition (iv).
 Indeed, if $[w]\in (I_1)_{\rm free}$ and $|C^1_w|>1$ then $|C^2_{\phi^0(w)}|= |C^1_w| >1$ because $\phi$ is a cover, and so $[\phi^0(w)] \in (I_2)_{\rm free}$. If $[w]\in (I_1)_{\rm free}$ and $|C^1_w|=1$ then
 there is only one loop at $w$ in $E_1'$ and so there is only one loop at $\phi^0(w)$ in $E_2'$, and thus in $E_2$. Therefore $[\phi^0(w)]\in (I_2)_{\rm free}$. Similarly, if $[w]\in (I_1)_{\rm reg}$ then
 $[\phi^0(w)]\in (I_2)_{\rm reg}$. (Note that in an arbitrary adaptable separated graph $(E,C)$, the sets $I_{\rm free}$ and $I_{\rm reg}$ are completely determined by the structure of $(E,C)$ as a separated graph.)
 
 (2) If $w\in (I_1)_{\rm free}$, this is a consequence of the fact that $\phi$ is a cover map. So assume that $w\in (I_1)_{\rm reg}$. Recall that in this case we have defined $X_w$ as $s_{E_1}^{-1}([w])$, that is, 
 the set of all the edges emitted by vertices in the strongly connected component $[w]$. If in addition $w\in T^0(v_i)$ for some $i$, then 
 the result follows from condition (iii) in Definition \ref{def:crownedpair}. If $w\in (E_1')^0$ then $\phi$ sends the strongly connected component $[w]$ of $w$ in $E_1$ bijectively to the strongly connected component 
 $[\phi ^0(w)]$ of $\phi^0(w)$ in $E_2$, by condition (iv).  If $e\in X_w =s_{E_1}^{-1}([w])$, then clearly $\phi^1(e) \in s_{E_2}^{-1}([\phi^0(w)])=X_{\phi^0(w)}$. Conversely, if $f\in  s_{E_2}^{-1}([\phi^0(w)])$ then there exists
 $w'\in [w]$ such that $\phi^0(w') = s(f)\in [\phi^0(w)]$ and so, by the cover property,
 there exists a unique $e\in s_{E_1}^{-1}(w')$ such that $\phi^1(e)= f$. If $e,e'\in X_w$ are such that $\phi^1(e)=\phi^1(e')$, then 
 $$\phi^0(s(e)) = s(\phi^1(e)) = s(\phi ^1(e'))= \phi^0(s(e')) \in X_{\phi^0(w)}$$
 and since $s(e),s(e')\in [w]$, we get that $s(e)=s(e')$ by the injectivity of $\phi^0|_{[w]}$. Now $e=e'$ follows from the fact that $\phi$ is a cover map. 
 
 (3) If $w\in (I_1)_{\rm free}$ this follows from the cover property of $\phi$. If $w\in (I_1)_{\rm reg}$, this follows from (1) and (2).
 
 (4) Assume for definiteness that $w\in \tilde{T}^0(v_1)$. Clearly there exists $X\in (\ol{C^1})_{w'}$, with $w'\in T^0(v_1)$ and $[w']\ne [w]$, such that
 $r(X)\cap [w]\ne \emptyset$. Let $X' \in (\ol{C^1})_{w''}$ be such that $r(X') \cap [w]\ne \emptyset$, where $[w'']\ne [w]$. By (2) we have that
 $\phi^1(X')\in (\ol{C^2})_{\phi^0(w'')}$ and clearly $r(\phi^1(X'))\cap [\phi^0(w)]\ne \emptyset$. If $[\phi^0(w'')]=[\phi^0(w)]$, then by (3) there is a unique
 $Y\in (\ol{C^1})_w$ such that $\phi^1(Y)= \phi^1(X')$. By condition (iv) we have from $[\phi^0(w'')]=[\phi^0(w)]$ that $w''\in T^0(v_1) \cup T^0(v_2)$ and by condition 
 (iii) we get that $w''\in T^0(v_2)$, because $[w]\ne [w'']$. But now we get that $r(X') \cap [w] \subseteq T^0(v_2)\cap T^0(v_1) = \emptyset $ (using (ii)), which is a contradiction. 
 Hence we get that $\phi^1(X') \in \ol{C^2}\setminus (\ol{C^2})_{\phi^0(w)}$, and $\phi^1(X')\cap [\phi^0(w)]\ne \emptyset$. By condition (i) we get $\phi^1(X')=\phi ^1(X)$. This in particular implies 
 that $[\phi^0(w')]=[\phi^0(w'')]$ which, using again conditions (ii)-(iv), implies that $w''\in T^0(v_1)$. Now by condition (iii) we get that $X=X'$, as desired.  
 
 (5) Let $w \in (E_1')^0$ and $X\in (\ol{C^1})_w$ be such that $r(X)\cap T^0(v_i)\ne \emptyset$ for some $i$. By (4) 
 we have $r(X)\cap T^0(v_j) = r(X)\cap [v_j]$ for $j=1,2$.
 Hence, using condition (v), we get 
  $$r(X)\cap (T^0(v_1)\cup T^0(v_2)) = r(X) \cap  ([v_1]\cup [v_2]) = r(X)\cap [v_i]. $$
  \end{proof}

\begin{remark}
 \label{rem:crownedremark}
{\rm  In the conditions of Theorem \ref{thm:covering-theorem}, we may obtain the desired chain of adaptable separated graphs $(F_i,D^i)$, $i=0,1,\dots ,m$, such that each pair $((F_i,D^i),(F_{i+1},D^{i+1}))$ is a crowned pair,
just by considering each time only two copies of the target separated graph $T(\ol{v})$ (instead of the $r$ copies considered in the proof of that theorem). With this procedure, for a given inductive step
$\psi'\colon (F',D')\to (F,D)$
as in the proof of Theorem \ref{thm:covering-theorem}, we arrive at the same result by using a chain of adaptable separated graphs 
$$(F_0',D_0')\to (F_1',D_1')\to \cdots \to  (F_{r-1}',D_{r-1}')$$ 
of length $r-1$, with $(F_0',D_0')=(F',D')$ and $(F_{r-1}',D_{r-1}')=(F,D)$. One directly verifies that each pair of consecutive terms 
in this chain is a crowned pair in the sense of Definition \ref{def:crownedpair}.} \qed
 \end{remark}

Let $P$ be a conical monoid, and suppose it contains two order-ideals $I$ and $I'$, with $I\cap I'=\{ 0 \}$, such that there is an isomorphism  $\varphi:I\to I'$. We consider the square:

      \[
            \begin{tikzpicture}
      \matrix (m) [matrix of math nodes,row sep=3em,column sep=3em,minimum width=1em]
      {
        I   &I\\
      I' &P\\
      };
      \path[-stealth]
      (m-1-1) edge [->] node  [above] {=} (m-1-2)
              edge [->] node [left] {$\varphi$} (m-2-1)

      (m-1-2) edge [->] node [right] {$\iota_1$} (m-2-2)

      (m-2-1)       edge [->] node  [below] {$\iota_{2}$} (m-2-2)
          ;
      \end{tikzpicture} 
      \]

\begin{definition}\label{def:crownedPush}
    The {\em crowned pushout} $Q$ of $(P,I,I',\varphi)$ is by definition the coequalizer of the maps $\iota_1\colon I\to P$ and $\iota_2\circ\varphi:I\to P$, so that there is a map $f\colon P\to Q$ with 
    $f(\iota_1(x))=f(\iota_2(\varphi(x)))$ for all $x\in I$, and given any other map $g\colon P\to Q'$ such that  $g(\iota_1(x))=g(\iota_2(\varphi(x)))$ for all $x\in I$, we have that $g$ factors uniquely through $f$.\qed
\end{definition}

We now show that if $((E_1,C^1),(E_2,C^2))$ is a crowned pair, we can obtain the monoid $M(E_2,C^2)$ as a crowned pushout of $M(E_1,C^1)$. Note that the cover map $\phi\colon (E_1,C^1)\to (E_2,C^2)$
induces a surjective monoid homomorphism $M(\phi) \colon M(E_1,C^1)\to M(E_2,C^2)$.

    \begin{proposition}\label{prop:PushOutMonoid} In the notation of Definition \ref{def:crownedpair}, let $((E_1,C^1), (E_2,C^2))$ be a crowned pair. Assume $P= M(E_1,C^1)$, and for $i=1,2$, let $N_i= M(T^0(v_i))$ be the order-ideal of $P$ generated by the hereditary $C^1$-saturated subset $T^0(v_i)$ of $E_1^0$. 
        Then the natural homomorphism $M(\phi) \colon P=M(E_1,C^1) \to M(E_2,C^2)$ is the crowned pushout of $(P,N_1,N_2, M(\varphi))$, where $M(\varphi)$ is the monoid isomorphism induced by the isomorphism 
        of separated graphs $\varphi :=(\phi|_{T(v_2)})^{-1}\circ (\phi|_{T(v_1)})$. 
        \end{proposition}    
    
    \begin{proof} Let $\theta\colon  P\to Q$ be the canonical map from $P$ to the crowned pushout $Q$ of $(P,N_1,N_2, M(\varphi))$.
    Observe that $\theta (a_w)= \theta (a_{\varphi (w)})$ for all $w\in T^0(v_1)$. 
    
    Clearly we have, for $w\in T^0(v_1)$,
       $$M(\phi) (a_w) = a_{\phi (w)} = a_{\phi (\varphi (w))} = M(\phi) (M(\varphi)(a_w)).$$
   so $M(\phi)$ coequalizes the maps the maps $\iota_1$ and $\iota_2 \circ M(\varphi)$ in the diagram
        \begin{equation} 
         {\label{eq:ncdiagram-monoids}}
        \begin{tikzpicture}[baseline=(current  bounding  box.center)]
        \matrix (m) [matrix of math nodes,row sep=3em,column sep=3em,minimum width=1em]
        {
            M(T(v_1))   &M(T(v_1))&&\\
            M(T(v_2)) &M( E_1, C^1)&&\\
        };
        \path[-stealth]
        (m-1-1) edge [->] node  [above] {=} (m-1-2)
        (m-1-1) edge [->] node [left] {$M(\varphi)$} (m-2-1)
        (m-1-2) 
        edge [->] node [right] {$\iota_1$} (m-2-2)
        (m-2-1)  edge [->] node  [below] {$\iota_2$} (m-2-2)
        ;
        \end{tikzpicture} 
               \end{equation}
  Therefore there is a unique monoid homomorphism $\rho \colon Q \to M(E_2,C^2)$ such that $M(\phi) = \rho \circ \theta $. Since $M(\phi)$ is surjective, we see that $\rho$ is also surjective.  
  We now define a homomorphism $\gamma \colon M(E_2,C^2) \to Q$ by the rules
  $$\gamma (a_w) = \theta (a_{(\phi|_{E_1'})^{-1}(w)}) \qquad  \text{if } w\in E_2^0\setminus T^0(v)$$
  and
  $$\gamma (a_w) = \theta (a_{(\phi|_{T(v_1)})^{-1} (w)}) = \theta (a_{(\phi|_{T(v_2)})^{-1} (w)}) \qquad \text{if } w\in T^0(v) .$$
  We need to check that $\gamma $ is well-defined. So let $w\in E_2^0$ and $X\in C_w$. If $w\in T^0(v)$, then it follows from condition 
  (iii) in Definition \ref{def:crownedpair} that the relation $a_w= \sum _{x\in X} a_{r(x)}$ is preserved by $\gamma $.
  Suppose now that $w\in E_2^0\setminus T^0(v) = (E_2')^0$. If $r(X)\cap T^0(v) =\emptyset$, then the fact  that the above relation is preserved by
  $\gamma $ follows from condition (iv) in Definition \ref{def:crownedpair}. Assume finally that $r(X)\cap T^0(v)\ne \emptyset$. By condition (i) in Definition \ref{def:crownedpair},
  we then have that $r(X) \cap T^0(v) = r(X)\cap [v]$. We write $X= X_1\sqcup X_2$ with
  $$X_1 = \{ x\in X\mid r(x)\notin T^0(v) \},\qquad X_2= \{ x\in X\mid r(x) \in [v] \}.$$
  Let $w'$ be the unique vertex in $E_1$ such that $\phi^0 (w')= w$. Clearly $w'\in (E_1')^0$.
  Let $Y$ be the unique element in $(C^1)_{w'}$ such that $\phi^1(Y)=X$. Since $\phi$ is a cover map, it follows that $\phi^1$ induces a bijection between $Y$ and $X$, so we can write
  $Y=Y_1\sqcup Y_2$ with $\phi^1 (Y_i)=X_i$ for $i=1,2$. Note that necessarily $r(y)\in (E_1')^0$ for $y\in Y_1$. 
   By condition (v) in Definition \ref{def:crownedpair} and Lemma \ref{lem:crownedproperties}(5), there is a unique $i\in \{ 1,2 \}$ such that $r(y)\in [v_i]$
   for all $y\in Y_2$.
  
  We now have
\begin{align*}
\gamma (a_w) & = \theta (a_{w'}) = \sum _{y\in Y_1} \theta (a_{r(y)}) + \sum _{y\in Y_2} \theta (a_{r(y)})\\ 
 &   = \sum _{x\in X_1} \theta(a_{(\phi|_{E_1'})^{-1}(r(x))}) + \sum_{x\in X_2} \theta (a_{(\phi|_{T(v_i)})^{-1}(r(x))})\\
 & = \sum _{x\in X} \gamma (a_{r(x)}) ,
  \end{align*}
   which shows that the relation $a_w= \sum _{x\in X} a_{r(x)}$ is also preserved in this case. 
  
  Hence we have a well-defined monoid homomorphism $\gamma \colon M(E_2,C^2) \to Q$. Observe that
  $$\gamma \circ \rho \circ \theta = \gamma \circ M(\phi) = \theta ,$$
  and so by the universal property of the crowned pushout we have that $\gamma \circ \rho = \text{id}_Q$. Therefore $ \rho $ is injective, and since we already know it is surjective, we
  conclude that $\rho $ is a monoid isomorphism. This shows the result. 
   \end{proof}
   
\subsection{Crowned Pushouts and von Neumann regular rings}\label{Subsection:CrownedPushouts} Now we describe the crowned pushout construction at the level of algebras and its relationship
with the corresponding $\mathcal V$-monoids.  
For this, we strongly use the following result appearing in \cite{Ara10}:
\begin{proposition}\cite[Proposition 4.5]{Ara10}
\label{prop:pushoutVNregular}
    Let $R$ be a (not necessarily unital) von Neumann regular ring with ideals $I$ and $I'$ such that $I\cap I' = 0$, and suppose that $I$ and $I'$ are Morita equivalent. Then there is a von Neumann regular ring $U$ with an ideal $J$ such that the following holds:
    \begin{enumerate}
        \item There exists an injective ring homomorphism $\alpha:R\to U$ such that $\alpha(I),\alpha(I')\subseteq J$.
        \item The map $\mathcal V(\alpha)\colon\mathcal V(R)\to\mathcal V(U)$ restricts to an isomorphism from $\mathcal V(I)$ onto $\mathcal V(J)$, and it also restricts to an isomorphism from $\mathcal V(I')$ onto $\mathcal V(J)$.
        \item Let $\varphi\colon \mathcal V(I)\to\mathcal V(I')\subseteq \mathcal V(R)$ be the isomorphism defined by $$\varphi\colon =(\mathcal V(\alpha)_{|\mathcal V(I')})^{-1}\circ (\mathcal V(\alpha)_{|\mathcal V(I)}).$$Then, $\mathcal V(\alpha)\colon\mathcal V(R)\to\mathcal V(U)$ is the 
        coequalizer of the following (noncommutative) diagram:

     \[
     \begin{tikzpicture}
     \matrix (m) [matrix of math nodes,row sep=3em,column sep=3em,minimum width=1em]
     {
        \mathcal V(I)   &   \mathcal V(I) \\
        \mathcal V(I') &\mathcal V(R)\\
     };
     \path[-stealth]
     (m-1-1) edge [->] node  [above] {=} (m-1-2)
     edge [->] node [left] {$\varphi$} (m-2-1)

     (m-1-2) edge [->] node [right] {$\iota_1$} (m-2-2)

     (m-2-1)       edge [->] node  [below] {$\iota_{2}$} (m-2-2)
     ;
     \end{tikzpicture}
     \]
    \end{enumerate}
\end{proposition}
The main result of this section is the next theorem. It is the key ingredient for the induction step in the proof of Theorem \ref{ThmA}, which will be given at the end of the section.

\begin{theorem}
\label{thm:main-pushoutthem}
Let $((E_1,C^1), (E_2,C^2))$ be a crowned pair and assume the notation in Definition \ref{def:crownedpair}.
        Suppose that $Q_K(E_1,C^1)$ is von Neumnan regular and that the natural map $M(E_1,C^1)\to \mathcal V (Q_K(E_1, C^1))$ is an isomorphism.
        Then $Q_K(E_2,C^2)$ is von Neumann regular and the natural map $M(E_2,C^2)\to \mathcal V (Q_K(E_2,C^2))$ is an isomorphism.
\end{theorem}
\begin{proof}
  Easing notation, we denote by $Q:=Q_K( E_1, C^1)$,  $I_1:=\langle T^0(v_1)\rangle$ and $I_2:=\langle T^0(v_2)\rangle$ the ideals of $Q$ generated by $T^0(v_1)$ and $T^0(v_2)$ respectively . 
  Notice that, using the idempotents $$e_1:=e(v_1)=\sum_{w\in T^0(v_1)}w\quad \,\,\,\text{ and }\quad \,\,\,e_2:=e(v_2)=\sum_{w\in T^0(v_2)}w,$$
  in the multiplier algebra of $Q$, one has a canonical isomorphism $e_iQe_i \cong Q_K(T(v_i))$ for $i=1,2$, by Theorem \ref{theor:repres}. Moreover, we have an isomorphism of separated graphs $\varphi \colon T(v_1) \to T(v_2)$ given by
  $$\varphi = (\phi|_{T(v_2)})^{-1} \circ (\phi|_{T(v_1)}),$$
which gives us an isomorphism $e_1Qe_1 \cong e_2Qe_2$ given by the composition $e_1Qe_1 \cong Q_K(T(v_1)) \cong Q_K(T(v_2)) \cong e_2Qe_2$.
  We will denote this isomorphism $e_1Qe_1\to e_2Qe_2$ also by $\varphi$. 
 Since the rings $I_i$ and $e_iQe_i$ are Morita-equivalent, we obtain a Morita-equivalence between the non-unital rings $I_1$ and $I_2$.
 This Morita equivalence is explicitly realized as follows. We set
 $$N = Qe_1\otimes_{e_1Qe_1} e_2Q, \qquad M = Qe_2\otimes _{e_2Qe_2} e_1Q,$$
 where the action of $e_1Qe_1$ on $e_2Q$ is defined by $x\cdot y = \varphi (x) y$, and similarly the action of $e_2Qe_2$ on $e_1Q$ is defined by $x\cdot y= \varphi^{-1} (x) y$.
 Observe that $N$ is an $I_1$-$I_2$-bimodule, $M$ is an $I_2$-$I_1$-bimodule and that there are isomorphisms
 $$M\otimes _{I_1} N \to I_2, \qquad  N\otimes_{I_2} M \to I_1$$
 implementing a Morita equivalence between $I_1$ and $I_2$.

 Write $Q_1:=Q/I_2$ and $Q_2:=Q/I_1$. In order to use Proposition \ref{prop:pushoutVNregular}, we now describe the ring $U$ associated to the $K$-algebra $Q$, the pair of ideals $I_1$ and $I_2$ and 
 the concrete Morita context described above
 (see the proof of \cite[Proposition 4.5]{Ara10}).  
  The following commutative diagram is a pullback:  \[
 \begin{tikzpicture}
 \matrix (m) [matrix of math nodes,row sep=3em,column sep=3em,minimum width=1em]
 {
    Q   & Q_1 \\
    Q_2 &Q/(I_1+I_2)\\
 };
 \path[-stealth]
 (m-1-1) edge [->] node  [above] {$\pi_1$} (m-1-2)
 edge [->] node [left] {$\pi_2$} (m-2-1)

 (m-1-2) edge [->] node [right] {$\pi_4$} (m-2-2)

 (m-2-1)       edge [->] node  [below] {$\pi_3$} (m-2-2)
 ;
 \end{tikzpicture}
 \] where each $\pi_i$ is the corresponding canonical projection map. 
 
 Following \cite[Proposition 4.5]{Ara10}, we now define the ring $U$ whose elements are all the matrices
 \[
 X=\left (\begin{array}{ll}
       q_1 & n\\
       m& q_2\\
      \end{array}\right) 
      \]
    such that $q_1\in Q_1$, $q_2\in Q_2$, $n\in N$, $m\in M$, and  $\pi_4(q_1)=\pi_3(q_2)$. Moreover, we consider the ideal $J$ in $U$ given by
       $ J:=\left (\begin{array}{ll}
     I_1 & N\\
     M& I_2\\
     \end{array}\right ),$ and the map $\alpha : Q \to U$ defined by $\alpha(q)={\rm diag}(\pi_1(q),\pi_2(q))\in  U$.

Since $Q$ is von Neumann regular by hypothesis, we obtain from Proposition \ref{prop:pushoutVNregular} that $U$ is von Neumann regular and that the map $\mathcal V (Q) \to \mathcal V (U)$ is the coequalizer of the (non-commutative) diagram
\[
     \begin{tikzpicture}
     \matrix (m) [matrix of math nodes,row sep=3em,column sep=3em,minimum width=1em]
     {
        \mathcal V(I_1)   &   \mathcal V(I_1) \\
        \mathcal V(I_2) &\mathcal V(Q)\\
     };
     \path[-stealth]
     (m-1-1) edge [->] node  [above] {=} (m-1-2)
     edge [->] node [left] {$\mathcal V (\varphi )$} (m-2-1)

     (m-1-2) edge [->] node [right] {$\iota_1$} (m-2-2)

     (m-2-1)       edge [->] node  [below] {$\iota_{2}$} (m-2-2)
     ;
     \end{tikzpicture}
     \]
But now observe that since $Q$ is regular and the natural map $M(E_1,C^1)\to \mathcal V (Q)$ is an isomorphism, this diagram translates to the diagram \eqref{eq:ncdiagram-monoids}.
Therefore we obtain from Proposition \ref{prop:PushOutMonoid} a natural isomorphism $\eta \colon M(E_2,C^2) \to \mathcal V (U)$. It is easily seen that $\eta$ is given by
$\eta (a_{w'}) = [\alpha (w)]\in \mathcal V (U)$, for $w'\in E_2^0$, where $w$ is the unique vertex in $E_1^0\setminus T^0(v_2)$ such that $ \phi^0(w)= w'$  (the fact that there is such a unique vertex $w$
follows from Definition \ref{def:crownedpair}).

 The rest of the proof consists of showing that $Q_K(E_2,C^2)\cong eUe$ for a certain full idempotent $e\in  \mathcal M (U)$. To this end, we define the full idempotent element $$e:=\left (\begin{array}{ll}
      1_{\mathcal M (Q_1)} & 0\\
       0& 1_{\mathcal M (Q_2)} -e_2\\
      \end{array}\right )\in \mathcal M (U) .$$ 
      Using it, we show that the map $$\delta\colon Q_K(E_2,C^2)\to e Ue,$$ defined below, is an algebra isomorphism such that the composition 
      $$M(E_2,C^2) \overset{\eta}{\to} \mathcal V (U) \cong \mathcal V (eUe) \overset{\mathcal V(\delta^{-1})}{\to} \mathcal V (Q_K(E_2,C^2))$$
is the natural map $M(E_2,C^2)\to \mathcal V (Q_K(E_2,C^2))$.
      This will imply that $Q_K(E_2,C^2)$ is regular and that the natural map $M(E_2,C^2)\to \mathcal V (Q_K(E_2,C^2))$ is
 an isomorphism.

     Based on the description of $Q_K(E_2,C^2)$ provided in section \ref{sect:algebras}, we define $\delta$ depending on the vertices and edges used in its definition. 
     For the {\bf vertices}, let $w'\in E_2^0$. Then, by the conditions in Definition \ref{def:crownedpair}, there is a unique $w\in E_1^0\setminus T^0(v_2)$ such that $\phi^0(w)= w'$, and we define
     $\delta(w')=\alpha(w)\in eUe$. Similarly, we set $\delta (t_i^{w'}) = \alpha (t_i^w)$, where $w'$ and $w$ are as above.  

     For the {\bf edges}, we differentiate two different types of edges. If $e'\in E_2^1$ is of the form $\phi^1(e)$ for $e\in E_1^1$ satisfying that $s(e), r(e) \in E_1^0\setminus T^0(v_2)$, then we 
     define $\delta (e') = \alpha (e)\in eUe$ and $\delta ((e')^*)= \alpha (e^*)$. Otherwise, again using the conditions in Definition \ref{def:crownedpair}, there is $\beta\in E_1^1$ with $s(\beta) \in (E_1')^0$ and $r(\beta ) \in T^0(v_2)$ such that
     $\phi^1 (\beta) = e'$. Note that $\beta$ is necessarily a connector. In this case, we define:
          $$\delta(e')= \delta (\phi^1(\beta))=\left (\begin{array}{ll}
       \qquad 0 & 0\\
       \beta\otimes \varphi^{-1} (r(\beta)) & 0\\
      \end{array}\right ), \qquad \delta((e')^*)=\left (\begin{array}{ll}
       0& \varphi^{-1} (r(\beta))\otimes \beta^*\\
       0 & \qquad 0\\
      \end{array}\right ).$$
Note that if $e'= \phi^1(\beta) $ as above, we have 
$$\delta (e')\delta ((e')^*) =  \left (\begin{array}{ll}
       \qquad 0 & 0\\ 
       \beta\otimes \varphi^{-1} (r(\beta)) & 0\\
      \end{array}\right ) \left (\begin{array}{ll}
       0& \varphi^{-1} (r(\beta))\otimes \beta^*\\
       0 & \qquad  0\\
       \end{array}\right ) = \alpha (\beta \beta^*) .$$

One can easily see that the defining relations of $\mathcal S_K(E_2,C^2)$, given in \eqref{pt:KeyDefs},  are preserved by $\delta$.  Let us just check that if $w'\in E_2^0$ and
$X\in C_{w'}$, then 
$$\delta (w') = \sum_{x\in X} \delta (x) \delta (x^*).$$
Let $w$ be the unique vertex in $E_1^0\setminus T^0(v_2)$ such that $\phi^0 (w) = w'$. Since $\phi$ is a cover map, there is a unique $Y\in C^1_w$ such that $\phi^1 (Y) =X$, and $\phi^1$ restricts to a bijection from $Y$ to $X$.
If $r(e)\in E_1^0\setminus T^0(v_2)$ for all $e\in Y$, then the above relation is straightforward. Otherwise, by condition (v) in Definition \ref{def:crownedpair}, there is a non-trivial partition
$Y=Y_1\sqcup Y_2$ such that $r(Y_1) \subseteq E_1^0\setminus (T^0(v_1) \cup T^0(v_2))$ and $r(Y_2) \subseteq T^0(v_2)$. Therefore, using the above observation, we get 
$$\sum _{x\in X} \delta (x)\delta (x^*) = \sum _{y\in Y} \delta (\phi^1 (y))\delta (\phi^1 (y)^*) = \sum _{y\in Y_1} \alpha (yy^*) + \sum _{y\in Y_2} \alpha (yy^*)= \alpha ( \sum_{y\in Y} yy^*)= \alpha (w) = \delta (w').$$
This shows the desired equality. 

We have thus a well-defined $K$-algebra homomorphism $\delta \colon \mathcal S _K(E_2,C^2) \to eUe$, and it is readily seen that this map extends to a $K$-algebra homomorphism, also denoted by $\delta$, from
$Q_K(E_2,C^2)$ to $eUe$.

      Let us show that $\delta\colon Q_K(E_2,C^2) \to eUe$ is an isomorphism. To prove this, we will use the following diagram:
      \[
      \begin{tikzpicture}
      \matrix (m) [matrix of math nodes,row sep=3em,column sep=4em,minimum width=2em]
      {
      	\langle T^0(v)\rangle&Q_K(E_2,C^2)& Q_K(E_2,C^2)/\langle T^0(v)\rangle\\
      	 eJe&  eUe & eUe/eJe\\};
      
      \path[-stealth]
      
      (m-1-1) edge [->] node  {} (m-1-2)
      edge node [right] {$\delta_{T(v)}$} (m-2-1)
      (m-1-2) edge [->] node  {} (m-1-3)
     
      edge node [right] {$\delta$} (m-2-2)
     (m-1-3) edge [->] node [right] {$\overline\delta$} (m-2-3)
      (m-2-1) edge [->] node  {} (m-2-2)
      
      (m-2-2) edge [->] node  {} (m-2-3)
    ;
      
      \end{tikzpicture}.
      \]
       Since
       $$Q_K(E_2,C^2)/\langle T^0(v)\rangle \cong Q/(I_1+I_2) \cong U/J \cong eUe/eJe , $$
       we have that the map $\ol{\delta}$ is an isomorphism. 
 We will conclude that $\delta$ is an isomorphism proving that $\delta_{T(v)}$ is also an isomorphism.
For this we will rely on the decomposition \eqref{eq:A22} (see Theorem \ref{thm:reducedmonomialsdirectsum}). 
Therefore we write 
$$Q_K(E_2,C^2)= \bigoplus _{(\gamma_1 , \gamma_2)\in \mathcal P} Q_{(\gamma_1,\gamma_2)},$$
where $\mathcal P $ is the set of pairs of finite paths $(\gamma_1,\gamma_2)$ in the reduced graph $(E_2)_{{\rm red}}$ with $r(\gamma_1) = r(\gamma_2)$.
Note that $$\langle T^0(v)\rangle = \bigoplus _{(\gamma_1,\gamma_2) \in \mathcal P : r(\gamma_1)=r(\gamma_2)\in T^0(v)} Q_{(\gamma_1,\gamma_2)}.$$ 
We classify the pairs $(\gamma_1,\gamma_2)\in \mathcal P $ such that $r(\gamma_1)=r(\gamma_2)\in T^0(v)$ into four classes.
For this, it is convenient to introduce a bit of terminology: let us say that a finite path $\gamma = \beta_1\beta_2\cdots \beta _n$, with $r(\gamma) \in T^0(v)$
{\it crosses the border through $v_2$} if there is a (necessarily unique) $i\in \{ 1,\dots , n\}$ such that $\beta _i = \phi^1(\beta)$ for $\beta \in E_1^1$ such that
$s(\beta ) \in (E_1')^0$ and $r(\beta) \in [v_2]$. 

Let $\gamma_1 = \beta_1\beta_2 \cdots \beta _r$ and $\gamma_2 = \beta _1'\beta_2'\cdots \beta_s'$ for connectors $\beta_i, \beta_j'$ in $E_2$. 
Then the four classes are as follows:        
      \begin{enumerate}[\rm(i)]
      \item $\gamma_1$ and $\gamma _2$ do not cross the border through $v_2$.
      \item $\gamma_1$ crosses the border through $v_2$ and $\gamma_2$ do not cross the border thorugh $v_2$.
      \item  $\gamma_1$ do not cross the border through $v_2$ and $\gamma_2$ crosses the border through $v_2$.
      \item  Both $\gamma_1$ and $\gamma_2$ cross the border through $v_2$. 
      \end{enumerate}
      Now each of the above four classes corresponds, through $\delta|_{T(v)}$, to a different corner in the ring 
      $$eJe = \left (\begin{array}{ll}
       \qquad I_1 & \qquad  N(1_{\mathcal M (Q_2)} -e_2) \\
       (1_{\mathcal M (Q_2)} -e_2)M & (1_{\mathcal M (Q_2)} -e_2)I_2(1_{\mathcal M (Q_2)} -e_2)\\
      \end{array}\right ).$$ 
      It follows that $\delta |_{T(v)}$ is an isomorphism from $\langle T^0(v) \rangle $ onto $eJe$. Therefore $\delta $ is an isomorphism, as required. 
      
      Finally we check that the composition $\mathcal V (\delta^{-1}) \circ \eta $ is the natural map $M(E_2,C^2) \to \mathcal V (Q_K(E_2,C^2))$. 
      Indeed, for $w'\in E_2^0$, let $w$ be the unique vertex in $E_1^0\setminus T^0(v_2)$ such that $\phi^0(w) = w'$. Then we have
      $$(\mathcal V (\delta^{-1})\circ \eta )(a_{w'})= \mathcal V (\delta^{-1}) ([\alpha (w)]) = [w'], $$
      showing the result.
               \end{proof}

We can finally complete the proof of our main result.

\medskip

\noindent{\it Proof of Theorem \ref{ThmA}:} Let $(E,C)$ be an adaptable separated graph. As we have already remarked (see Remark \ref{rem:crownedremark}), the proof of Theorem \ref{thm:covering-theorem} enables us to build a 
finite chain of adaptable separated graphs $(F_i,D^i)$, for $i=0,\dots , m$, such that $(F_0,D^0)= (\tilde{E},\tilde{C})$ satisfies condition {\rm (F)}, and
$(F_m,D^m)=(E,C)$. Moreover, each pair $((F_i,D^i), (F_{i+1}, D^{i+1}))$, $i=0,\dots ,m-1$, is a crowned pair in the sense of Definition \ref{def:crownedpair}.

By Theorem \ref{thm:main-for-F}, $Q_K(F_0,D^0)$ is von Neumann regular and the natural map $M(F_0,D^0)\to \mathcal V (Q_K(F_0,D^0))$ is a monoid isomorphism. 
Now using Theorem \ref{thm:main-pushoutthem} $m$ times, we get the same conclusion for $Q_K(E_m,C^m)=Q_K(E,C)$. This completes the proof.\qed

\section*{Acknowledgments}

This research project was initiated when the authors were at the Centre de Recerca Matem\`atica as part of the Intensive Research Program \emph{Operator
	algebras: dynamics and interactions} in 2017, and continued at the Universidad de Cádiz  due to an invitation of the third author to the first two authors in 2018.  The work developed was significantly supported by the
research environment and facilities provided in both centers.
We would also like to thank the anonymous referee for very useful comments and suggestions.
                            

\end{document}